%% file: quasiflat_cube.tex
\documentclass{amsart}  
 
\usepackage{pinlabel}
\usepackage[all]{xy}

\usepackage{amsmath}
\usepackage{amssymb}
\usepackage{graphicx}
\usepackage{amsmath,amscd}
\usepackage{tikz}
\usetikzlibrary{positioning,arrows}
\usepackage[utf8]{inputenc} 
\usepackage[T1]{fontenc}
\usepackage{hyperref}
\usepackage{theoremref}

\numberwithin{equation}{section}
\theoremstyle{plain}
\newtheorem{lem}[equation]{Lemma}

\newtheorem{thm}[equation]{Theorem}
\newtheorem{cor}[equation]{Corollary}

\theoremstyle{definition}
\newtheorem{definition}[equation]{Definition}

\newtheorem{remark}[equation]{Remark}

\newtheorem{claim}{Claim}

\newtheorem*{que*}{Question}
\newtheorem*{remark*}{Remark}
\newtheorem*{definition*}{Definition}

\newcommand*{\inc}{\ensuremath{\mathcal{I}}}
\def\co{\colon\thinspace}

\begin{document}
\title{Top dimensional quasiflats in $CAT(0)$ cube complexes}
\author{Jingyin Huang}

\address{The Department of Mathematics and Statistics\\
McGill University\\
Burnside Hall, Room 1242 \\
805 Sherbrooke W.\\
Montreal, QC, H3A 0B9\\
Canada}
\email{jingyin.huang@mcgill.ca}

\begin{abstract}
We show that every $n$-quasiflat in an $n$-dimensional $CAT(0)$ cube complex is at finite Hausdorff distance from a finite union of $n$-dimensional orthants. Then we introduce a class of cube complexes, called {\em weakly special} cube complexes and show that quasi-isometries between their universal covers preserve top dimensional flats. This is the foundational towards the quasi-isometry classification of right-angled Artin groups with finite outer automorphism group.

Some of our arguments also extend to $CAT(0)$ spaces of finite geometric dimension. In particular, we give a short proof of the fact that a top dimensional quasiflat in a Euclidean building is Hausdorff close to finite union of Weyl cones, which was previously established in \cite{kleiner1997rigidity,eskin1997quasi,wortman2006quasiflats} by different methods.
\end{abstract}

\maketitle

\setcounter{tocdepth}{2}
\tableofcontents

\input{Body}

\appendix
\input{Appendix}

\bibliographystyle{alpha}
\bibliography{1}

\end{document}

%% file: Body.tex
\section{Introduction}
\subsection{Summary of results}
A quasiflat of dimension $d$ in a metric space $X$ is a quasi-isometric embedding $\phi\co\Bbb E^{d}\to X$, i.e. there exist positive constants $L,A$ such that for all $x,y\in \Bbb E^{d}$,
\begin{equation*}
L^{-1}d(x,y)-A\le d(\phi(x),\phi(y))\le Ld(x,y)+A\,.
\end{equation*}

Top dimensional (or maximal) flats and quasiflats in spaces of higher rank are analogues of geodesics and quasi-geodesics in Gromov hyperbolic spaces, which play a key role in understanding the large scale geometry of these spaces. In particular, several quasi-isometric rigidity results were established on the study of such flats or quasiflats, here is a list of examples:
\begin{itemize}
\item Euclidean buildings and symmetric spaces of non-compact type: \cite{mostow1973strong,kleiner1997rigidity,eskin1997quasi,kramer2009coarse}.
\item Universal covers of certain Haken manifolds: \cite{kapovich1997quasi}; higher-dimensional graph manifolds: \cite{frigerio2011rigidity}; two-dimensional tree groups and their higher dimensional analogues: \cite{behrstock2008quasi,MR2727658}.
\item $CAT(0)$ 2-complexes: \cite{bestvina2008quasiflats}, with applications to the quasi-isometric rigidity of atomic right-angled Artin groups in \cite{MR2421136}.
\item Flats generated by Dehn-twists in mapping class groups: \cite{MR2928983}.
\end{itemize}

In this paper, we will mainly focus on top dimensional quasiflats and flats in $CAT(0)$ cube complexes. All cube complexes in this paper will be finite dimensional. Our first main result shows how the cubical structure interacts with quasiflats. 

\begin{thm}
\label{1.1}
If $X$ is a $CAT(0)$ cube complex of dimension $n$, then for every $n$--quasiflat $Q$ in $X$, there is a finite collection {$O_{1},...,O_{k}$} of $n$--dimensional orthant subcomplexes in $X$ such that
\begin{center}
$d_{H}(Q, \cup_{i=1}^{k}O_k)<\infty$
\end{center}
Here $d_{H}$ denotes the Hausdorff distance.
\end{thm}

An \textit{orthant} $O$ of $X$ is a convex subset which is isometric to the Cartesian product of finitely many half-lines $\Bbb R_{\ge 0}$. If $O$ is both a subcomplex and an orthant, then $O$ is called an \textit{orthant subcomplex}. We caution the reader that the definition of orthant subcomplex here is slightly different from other places, i.e. we require an orthant subcomplex to be convex with respect to the $CAT(0)$ metric.

The $2$--dimensional case of Theorem \ref{1.1} was proved in \cite{bestvina2008quasiflats}. We will use this theorem as one of main ingredients to study the coarse geometry of right-angled Artin groups (see Corollary~\ref{1.4} below and the remarks after). Also note that recently Behrstock, Hagen, and Sisto have obtained a quasi-flat theorem of quite different flavour in \cite{behrstock2014hierarchically}. Their result does not imply our result and vice versa.

Based on Theorem \ref{1.1}, we study how the top dimensional flats behave under quasi-isometries. In general, quasi-isometries between $CAT(0)$ complexes of the same dimensional do not necessarily preserve top dimension flats up to finite Hausdorff distance, even if the underlying spaces are cocompact. However, motivated by \cite{MR2377497}, we can define a large class of cube complexes such that top dimensional flats behave nicely with respect to quasi-isometries between universal covers of these complexes. Our class contains all compact non-positively curved special cube complexes up to finite cover (\cite[Proposition 3.10]{MR2377497}).

\begin{definition}
\label{1.2}
A cube complex $W$ is \textit{weakly special} if and only if it has the following properties:
\begin{enumerate}
\item $W$ is non-positively curved.
\item No hyperplane \textit{self-osculates} or \textit{self-intersects}.
\end{enumerate}
\end{definition}

The notions of self-osculate and self-intersect were introduced in \cite[Definition 3.1]{MR2377497}. 

\begin{thm}
\label{1.3}
Let $W'_{1}$ and $W'_{2}$ be two compact weakly special cube complexes with $\dim(W'_{1})=\dim(W'_{2})=n$, and let $W_{1}$, $W_{2}$ be the universal covers of $W'_{1}$, $W'_{2}$ respectively. If $f\co W_{1}\to W_{2}$ is a $(L,A)$--quasi-isometry, then there exists a constant $C=C(L,A)$ such that for any top dimensional flat $F\subset W_{1}$, there exists a top dimensional flat $F'\subset W_{2}$ with $d_{H}(f(F),F')<C$.
\end{thm}

We now apply this result to \textit{right-angled Artin groups} (RAAGs). Recall that for every finite simplicial graph $\Gamma$ with its vertex set denoted by $\{v_i\}_{i\in I}$, one can define a group using the following presentation.
\begin{center}
\{$v_i$, for $i\in{I}\ |\ [v_i,v_j]=1$ if $v_{i}$ and $v_{j}$ are adjacent\}
\end{center}
This is called the \textit{right-angled Artin group with defining graph $\Gamma$}, and we denote it by $G(\Gamma)$. Each $G(\Gamma)$ can be realized as the fundamental group of a non-positively curved cube complex $\bar{X}(\Gamma)$, which is called the Salvetti complex (see \cite{charney2007introduction} for a precise definition). The $2$--skeleton of the Salvetti complex is the usual presentation complex for $G(\Gamma)$. The universal cover of $\bar{X}(\Gamma)$ is a $CAT(0)$ cube complex, which we denote by $X(\Gamma)$. 

\begin{cor}
\label{1.4}
Let $\Gamma_{1}$, $\Gamma_{2}$ be finite simplicial graphs, and let $\phi\co X(\Gamma_{1})\to X(\Gamma_{2})$ be an $(L, A)$--quasi-isometry. Then
\begin{enumerate}
\item $\dim(X(\Gamma_{1}))=\dim(X(\Gamma_{2}))$.
\item There is a constant $D=D(L, A)$ such that for any top-dimensional flat $F_{1}$ in $X(\Gamma_{1})$, we can find a flat $F_{2}$ in $X(\Gamma_{2})$ such that $d_{H}(\phi(F_{1}),F_{2})< D$.
\end{enumerate}
\end{cor}

This is the foundation for a series of work on quasi-isometric classification and rigidity of RAAGs \cite{huang2014quasi,huang2015quasi,huang2016groups,huang2016rigid}.

\begin{remark*}
We could also use Theorem \ref{1.3} to obtain an analogous statement for quasi-isometries between the Davis complexes of certain right-angled Coxeter groups, but in general the dimension of maximal flats in a Davis complex are strictly smaller than the dimension of complex itself, so we need extra condition on the right-angled Coxeter groups, see Corollary \ref{5.20} for a precise statement.
\end{remark*}

Corollary \ref{1.4} implies that $\phi$ maps chains of top dimensional flats to chains of top dimensional flats, and this gives rise to several quasi-isometry invariant for RAAGs. More precisely, we consider a graph $\mathcal{G}_{d}(\Gamma)$ where the vertices are in 1--1 correspondence to top dimensional flats in $X(\Gamma)$ and two vertices are connected by an edge if and only if the coarse intersection of the corresponding flats has dimension $\ge d$. The connectedness of $\mathcal{G}_{d}(\Gamma)$ can be read off from $\Gamma$, which gives us the desired invariants. 

\begin{definition}
\label{1.5}
Let $d\ge 1$ be an integer. Let $\Gamma$ be a finite simplicial graph and let $F(\Gamma)$ be the flag complex that has $\Gamma$ as its $1$--skeleton. $\Gamma$ has \textit{property $(P_{d})$} if and only if
\begin{enumerate}
\item Any two top dimensional simplexes $\Delta_{1}$ and $\Delta_{2}$ in $F(\Gamma)$ are connected by a \textit{$(d-1)$--gallery}.
\item For any vertex $v\in F(\Gamma)$, there is a top dimensional simplex $\Delta\subset F(\Gamma)$ such that $\Delta$ contains at least $d$ vertices that are adjacent to $v$.
\end{enumerate}
A sequence of $n$--dimensional simplexes $\{\Delta_{i}\}_{i=1}^{p}$ in $F(\Gamma)$ is a \textit{$k$--gallery} if $\Delta_{i}\cap\Delta_{i+1}$ contains a $k$--dimensional simplex for $1\le i\le p-1$.
\end{definition}

\begin{thm}
\label{1.6}
$\mathcal{G}_{d}(\Gamma)$ is connected if and only if $\Gamma$ has property $(P_{d})$. In particular, for any $d\ge 1$, property $(P_{d})$ is a quasi-isometry invariant for RAAGs.
\end{thm}

\begin{remark*}
Another interesting fact in the case $d=1$ is that one can tell whether $\Gamma$ admits a non-trivial join decomposition by looking at the diameter of $\mathcal{G}_{1}(\Gamma)$. This basically follows from the argument in \cite{dani2012divergence}. See Theorem \ref{5.32} for a precise statement. Thus in the case of $X(\Gamma)$, one can determine whether the space splits as a product by looking at the intersection pattern of top dimensional flats. We ask whether this is true in general: if $Z$ is a cocompact geodesically complete $CAT(0)$ space that has $n$--flats but not $(n+1)$--flats, can one determine whether $Z$ splits as a product of two unbounded $CAT(0)$ spaces by looking at the intersection pattern of $n$--flats in $Z$?
\end{remark*}

Actually, a large portion of our discussion generalizes to $n$--dimensional quasiflats in $CAT(0)$ spaces of geometric dimension $=n$ (the notion of geometric dimension and its relation to other notions of dimension are discussed in \cite{kleiner1999local}). This will be discussed in the appendix and see Theorem \ref{6.19} and Theorem \ref{6.20} for a summary.

In particular, this leads to a short proof of the following result, which was previously established in \cite{kleiner1997rigidity,eskin1997quasi,wortman2006quasiflats} by different methods and it is one of the main ingredients in proving quasi-isometric rigidity for Euclidean buildings.
\begin{thm}
\label{1.8}
Let $Y$ be a Euclidean building of rank $n$, and let $Q\subset Y$ be an $n$--quasiflat. Then there exist finitely many Weyl cones $\{W_{i}\}_{i=1}^{h}$ such that 
\begin{equation*}
d_{H}(Q,\cup_{i=1}^{h}W_{i})<\infty\,.
\end{equation*}
\end{thm}

On the way to Theorem \ref{1.8}, we also give a more accessible proof of the following weaker version of one of the main results in \cite{quasimini}.
\begin{thm}
\label{1.7}
Let $q\co Y\to Y'$ be a quasi-isometric embedding, where $Y$ and $Y'$ are $CAT(0)$ spaces of geometric dimension $\le n$. Then $q$ induces a monomorphism $q_{\ast}\co H_{n-1}(\partial_{T}Y)\to H_{n-1}(\partial_{T}Y')$. If $q$ is a quasi-isometry, then $q_{\ast}$ is an isomorphism.
\end{thm} 

Here $\partial_{T}Y$ and $\partial_{T}Y'$ denote the Tits boundary of $Y$ and $Y'$ respectively. 

\subsection{Sketch of proofs}
\subsubsection{Proof of Theorem \ref{1.1} and Theorem \ref{1.8}}
The proof of Theorem \ref{1.1} has 5 steps as below. The first one follows \cite{bestvina2008quasiflats} closely, but the others are different, since part of the argument in \cite{bestvina2008quasiflats} depends heavily on special features of dimension $2$, and does not generalize to the $n$--dimensional case.

Let $X$ be a $CAT(0)$ piecewise Euclidean polyhedral complex with $\dim(X)=n$, and let $Q\co\Bbb E^{n}\to X$ be a top dimensional quasiflats in $X$.

\textit{Step 1:} Following \cite{bestvina2008quasiflats}, one can replace the top dimensional quasiflats, which usually contains local wiggles, by a minimizing object which is more rigid. 

More precisely, let us assume without of loss of generality that $Q$ is a continuous quasi-isometric embedding. Let $[\Bbb E_{n}]$ be the fundamental class in the $n$--th locally finite homology group of $\Bbb E^{n}$ and let $[\sigma]=Q_{\ast}([\Bbb E_{n}])$. Let $S$ be the support set (Definition \ref{3.1}) of $[\sigma]$. It turns out that $S$ has nice local property (it is a subcomplex with geodesic extension property) and asymptotic property (it looks like a cone from far away). Moreover, $d_{H}(S,Q)<\infty$.

In the next few steps, we study the structure of $S$ by looking at its \textquotedblleft boundary\textquotedblright.

Recall that $X$ has a Tits boundary $\partial_{T}X$, whose points are asymptotic classes of geodesic rays in $X$, and the asymptotic angle between two geodesic rays induces a metric on $\partial_{T}X$. See Section \ref{cat space} for a precise definition. We define boundary of $S$, denoted $\partial_{T}S$, to be the subset of $\partial_{T}X$ corresponding to geodesic rays inside $S$. 

\textit{Step 2:} We produce a collection of orthants in $X$ from $S$. More precisely, we find an embedded simplicial complex $K\subset\partial_{T}X$ such that $\partial_{T}S\subset K$. Moreover, $K$ is made of right-angled spherical simplexes, each of which is the boundary of an isometrically embedded orthant in $X$. This step depends on the cubical structure of $X$, and is discussed in Section \ref{orthants}.

\textit{Step 3:} We show $\partial_{T}S$ is actually a cycle. Namely, it is the \textquotedblleft boundary cycle at infinity\textquotedblright\ of the homology class $[\sigma]$. This step does not depend on the cubical structure of $X$ and is actually true in greater generality by the much earlier, but still unpublished work of Kleiner and Lang (\cite{quasimini}). However, their paper was based on metric current theory. Under the assumption of Theorem \ref{1.1}, we are able to give a self-contained account which only requires homology theory, see Section \ref{cycle at infinity}.

\textit{Step 4:} We deduce from the previous two steps that $\partial_{T}S$ is a cycle made of $(n-1)$--dimensional all-right spherical simplices. Moreover, each simplex is the boundary of an orthant in $X$. 

\textit{Step 5:} We finish the proof by showing $S$ is Hausdorff close to the union of these orthants. See Section \ref{cubical filling} for the last two steps.

If $X$ a is Euclidean building, then it is already clear that the cycle at infinity can be represented by a cellular cycle, since the Tits boundary is a polyhedral complex (a spherical building). The problem is that $X$ itself may not be a polyhedral complex. There are several ways to get around this point. Here we deal with it by generalizing several results of \cite{bestvina2008quasiflats} to $CAT(0)$ spaces of finite geometric dimension, which is of independent interests.

\subsubsection{Proof of Theorem \ref{1.3}}
Let $W_{1}$ and $W_{2}$ be the universal covers of two weakly special cube complexes. We also assume $\dim(W_{1})=\dim(W_{2})=n$. Our starting point is similar to the treatment in \cite{kleiner1997rigidity,MR2421136}. Let $\mathcal{KQ}(W_{i})$ be the lattice generated by finite unions, coarse intersections and coarse subtractions of top dimensional quasiflats in $W_{i}$, modulo finite Hausdorff distance. Any quasi-isometry $q\co W_{1}\to W_{2}$ will induces a bijection $q_{\sharp}\co\mathcal{KQ}(W_{1})\to\mathcal{KQ}(W_{2})$.

It suffices to study the combinatorial structure of $\mathcal{KQ}(W_{i})$. By Theorem \ref{1.1}, each element $[A]\in \mathcal{KQ}(W_{i})$ is made of a union of top dimensional orthants, together with several lower dimensional objects. We denote the number of top dimensional orthants in $[A]$ by $|[A]|$. $[A]$ is \textit{essential} if $|[A]|>0$ and $[A]$ is \textit{minimal essential} if for any $[B]\in\mathcal{KQ}(W_{i})$ with $[B]\subset[A]$ (i.e. $B$ is coarsely contained in $A$) and $[B]\neq [A]$, we have $|[B]|=0$. 

It suffices to study the minimal essential elements of $\mathcal{KQ}(W_{i})$, since every elements in $\mathcal{KQ}(W_{i})$ can be decomposed into minimal essential elements together with several lower dimensional objects. In the case of universal covers of special cube complexes, these elements have nice characterizations and behave nicely with respect to quasi-isometries:

\begin{thm}
If $[A]\in \mathcal{KQ}(W_{i})$ is minimal essential, then there exists a convex subcomplex $K\subset W_{i}$ which is isometric to $(\Bbb R_{\ge 0})^{k}\times \Bbb R^{n-k}$ such that $[K]=[A]$.
\end{thm}

\begin{thm}
$|q_{\sharp}([A])|=|[A]|$ for any minimal essential element $[A]\in\mathcal{KQ}(W_{1})$.
\end{thm}

Theorem \ref{1.3} essentially follows from the above two results.

\subsection{Organization of the paper}
In Section \ref{prelim} we will recall several basic facts about $CAT(\kappa)$ spaces and $CAT(0)$ cube complexes. We also collect several technical lemmas in this section, which will be used later.

In Section \ref{support of homology class} we will review the discussion in \cite{bestvina2008quasiflats} which enable us to replace the top dimensional quasiflat by the support set of the corresponding homology class. In Section \ref{topdim support set} we will study the geometry of this support set and prove Theorem \ref{1.1}. In Section \ref{lattice} and Section \ref{weakly special}, we look at the behaviour of top dimensional flats in the universal covers of weakly special cube complexes and prove Theorem \ref{1.3} and Corollary \ref{1.4}. In Section \ref{applications}, we use Corollary \ref{1.4} to establish several quasi-isometric invariants for RAAGs.

In the appendix, we generalize some results of Section \ref{support of homology class} and Section \ref{topdim support set} to $CAT(0)$ spaces of finite geometric dimension and prove Theorem \ref{1.7} and Theorem \ref{1.8}.

\subsection{Acknowledgement}
This paper is part of the author's Ph.D. thesis and it was finished under the supervision of B. Kleiner. The author would like to thank B. Kleiner for all the helpful suggestions and stimulating discussions. The author is grateful to B. Kleiner and U. Lang for sharing the preprint \cite{quasimini}, which influences several ideas in this paper. The author also thanks the referee for extremely helpful comments and clarifications.

\section{Preliminaries}
\label{prelim}
We start with several basic notations. The open and closed ball of radius $r$ in a metric space will be denoted by $B(p,r)$ and $\bar{B}(p,r)$ respectively. The sphere of radius $r$ centered at $p$ is denoted by $S(p,r)$. The open $r$--neighbourhood of a set $A$ in a metric space is denoted by $N_{r}(A)$. The diameter of $A$ is denoted by $diam(A)$.

For a metric space $K$, $C_{\kappa}K$ denotes the $\kappa$--cone over $K$ (see Definition I.5.6 of \cite{MR1744486}). When $\kappa=0$, we call it the Euclidean cone over $K$ and denote it by $CK$ for simplicity. All products of metric spaces in this paper will be $l^{2}$--products.

The closed and open star of a vertex $v$ in a polyhedral complex are denoted by $\overline{st}(v)$ and $st(v)$ respectively. We use \textquotedblleft$\ast$\textquotedblright\  for the join of two polyhedral complexes and \textquotedblleft$\circ$\textquotedblright\ for the join of two 
graphs.
\subsection{$M_{k}$--Polyhedral complexes with finite shapes}
In this section, we summary some results about $M_{k}$--polyhedral complexes with finitely many isometry types of cells from Chapter I.7 of \cite{MR1744486}, see also \cite{MR2686786}.

A $M_{k}$--polyhedral complex is obtained by taking a disjoint union of a collection of convex polyhedra from the complete simply-connected $n$--dimensional Riemannian manifolds with constant curvature equal to $k$ ($n$ is not fixed) and gluing them along isometric faces. The complex is endowed with the quotient metric (see \cite[Definition I.7.37]{MR1744486}). Note that the topology induced by the quotient metric may be different from the topology as a cell complex.

A $M_1$--polyhedral complex is also called a piecewise spherical complex. If the complex is made of right-angled spherical simplexes, the it is also called a all-right spherical complex. A $M_0$--polyhedral complex is also called a piecewise Euclidean complex. 

We are mainly interested in the case where the collection of convex polyhedra we use to build the complex has only finitely many isometry types. Following \cite{MR1744486}, we denote the isometry classes of cells in $K$ by $Shape(K)$. Note that we can barycentrically subdivide any $M_{k}$--polyhedral complex twice to get a $M_{k}$--simplicial complex.

For a $M_{k}$--polyhedral complex $K$ and a point $x\in K$, we denote the unique open cell of $K$ which contains $x$ by $supp(x)$ and the closure of $supp(x)$ by $Supp(x)$. We also denote the geometric link of $x$ in $K$ by $Lk(x,K)$ (see \cite[I.7.38]{MR2686786}). In this paper, we always truncate the usual length metric on $Lk(x,K)$ by $\pi$. If a $\epsilon$--ball $B(x,\epsilon)$ around $x$ satisfies:
\begin{itemize}
\item $B(x,\epsilon)$ is contained in the open star of $x$ in $K$.
\item $B(x,\epsilon)$ is isometric to the $\epsilon$--ball centered at the cone point in $C_{k}(Lk(x,K))$.
\end{itemize}
Then we call $B(x,\epsilon)$ a \textit{cone neighbourhood of $x$}.

\begin{thm}(Theorem I.7.39 of \cite{MR1744486})
\label{2.1}
If $K$ is a $M_{k}$--polyhedral complex with $Shape(K)$ finite, then for every $x\in K$, there exists a positive number $\epsilon$ (depended on $x$) such that $B(x,\epsilon)$ is a cone neighbourhood of $x$.
\end{thm}

\begin{thm}(Theorem I.7.50 of \cite{MR1744486})
\label{2.2}
If $K$ is a $M_{k}$--polyhedral complex with $Shape(K)$ finite, then $K$ is a complete geodesic metric space.
\end{thm}

\begin{lem}
\label{2.3}
If $K$ is a $M_{k}$--polyhedral complex with $Shape(K)$ finite, then there exist positive constants $c_{1}, c_{2}$ $(c_{1}$ and $c_{2}$ depend on $Shape(K))$ such that every geodesic segment in K of length $L$ is contained in a subcomplex which is a union of at most $c_{1}L+c_{2}$ closed cells.
\end{lem}
This lemma follows from Corollary I.7.29, Corollary I.7.30 of \cite{MR1744486}.

\subsection{$CAT(\kappa)$ spaces}
\label{cat space}
Please see \cite{MR1744486} for an introduction to $CAT(\kappa)$ spaces.

Let $X$ be a $CAT(0)$ spaces and pick $x,y\in X$, we denote by $\overline{xy}$ the unique geodesic segment joining $x$ and $y$. For any $y,z\in X\setminus \{x\}$, we denote the comparison angle between $\overline{xy}$ and $\overline{xz}$ at $x$ by $\overline{\angle}_{x}(y,z)$ and the Alexandrov angle by $\angle_{x}(y,z)$. 

The Alexandrov angle induces a distance on the space of germs of geodesics emanating from $x$. The completion of this metric space is called the \textit{space of directions} at $x$ and is denoted by $\Sigma_{x}X$. The tangent cone at $x$, denoted $T_{x}X$, is the Euclidean cone over $\Sigma_{x}X$. Following \cite{bestvina2008quasiflats}, we define the logarithmic map $\log_{p}\co X\setminus\{x\}\to\Sigma_{x}X$ by sending $y\in X\setminus\{x\}$ to the point in $\Sigma_{x}X$ represented by $\overline{xy}$. Similarly, one can define $\log_{x}\co X\to T_{x}X$. For constant speed geodesic $l\co[a,b]\to X$, we denote by $l^{-}(t)$ and $l^{+}(t)$ the incoming and outcoming direction in $\Sigma_{l(t)}X$ for $t\in [a,b]$. Note that if $X$ is a $CAT(0)$ $M_{k}$--polyhedral complex with finitely many isometry types of cells, then $\Sigma_{x}X$ is naturally isometric to $Lk(x,X)$, so we will identify this two objects.

Denote the Tits boundary of $X$ by $\partial_{T}X$. We also have a well-defined map $\log_{x}\co\partial_{T}X\to\Sigma_{x}X$. For $\xi_{1},\xi_{2}\in\partial_{T}X$, recall that the Tits angle $\angle_{T}(\xi_{1},\xi_{2})$ between them is defined as follows:
\begin{equation*}
\angle_{T}(\xi_{1},\xi_{2})=\sup_{x\in X}\angle_{x}(\xi_{1},\xi_{2})\,.
\end{equation*}
This induces a metric on $\partial_{T}X$, which is called the \textit{angular metric}. There are several different ways to define $\angle_{T}(\xi_{1},\xi_{2})$ (see \cite[Section 2.3]{kleiner1997rigidity} or \cite[Chapter II.9]{MR1744486}):

\begin{lem}
\label{2.4}
Let $X$ be a complete $CAT(0)$ space and let $\xi_{1},\xi_{2}$ be as above. Pick base point $p\in X$, and let $l_{1}$ and $l_{2}$ be two unit speed geodesic ray emanating from $p$ such that $l_{i}(\infty)=\xi_{i}$ for $i=1,2$. Then
\begin{enumerate}
\item $\angle_{T}(\xi_{1},\xi_{2})=\lim_{t,t'\to\infty}\overline{\angle}_{p}(l_{1}(t),l_{2}(t'))$.
\item $\angle_{T}(\xi_{1},\xi_{2})=\lim_{t\to\infty}\angle_{l_{1}(t)}(\xi_{1},\xi_{2})$.
\item $2\sin(\angle_{T}(\xi_{1},\xi_{2})/2)=\lim_{t\to\infty}d(l_{1}(t),l_{2}(t))/t$.
\end{enumerate}
\end{lem}

$(\partial_{T}X,\angle_{T})$ is a $CAT(1)$ space (see \cite[Chapter II.9]{MR1744486}). We denote the Tits cone, which is the Euclidean cone over $\partial_{T}X$, by $C_{T}X$. Note that $C_{T}X$ is $CAT(0)$. Denote the cone point of $C_{T}X$ by $o$. Then for each $p\in X$, there is a well-defined 1--Lipschitz logarithmic map $\log_{p}\co C_{T}X\to X$ sending geodesic ray $\overline{o\xi}\subset C_{T}X$ ($\xi\in\partial_{T}X$) to geodesic ray $\overline{p\xi}\subset X$. This also gives rise to two other 1--Lipschitz logarithmic maps:
\begin{equation*}
\log_{p}\co C_{T}X\to T_{p}X;\ \ \log_{p}\co \partial_{T}X\to \Sigma_{p}X\,.
\end{equation*}

We always have $\angle_{p}(\xi_{1},\xi_{2})\le\angle_{T}(\xi_{1},\xi_{2})$ and the following flat sector lemma (see \cite[Section 2.3]{kleiner1997rigidity} or \cite[Chapter II.9]{MR1744486}) describes when the equality holds:

\begin{lem}
\label{2.5}
Let $X,\xi_{1},\xi_{2},l_{1},l_{2}$ and $p$ be as above. If $\angle_{T}(\xi_{1},\xi_{2})=\angle_{p}(\xi_{1},\xi_{2})<\pi$, then the convex hull of $l_{1}$ and $l_{2}$ in $X$ is isometric to a sector of angle $\angle_{p}(\xi_{1},\xi_{2})$ in the Euclidean plane.
\end{lem}

Any convex subset $C\subset X$ is also a $CAT(0)$ space (with the induced metric) and there is an isomeric embedding $i\co\partial_{T}C\to\partial_{T}X$. There is a well-defined nearest point projection $\pi_{C}\co X\to C$, which has the following properties:
\begin{lem}
\label{2.6}
Let $X,C$ and $\pi_{C}$ be as above. Then:
\begin{enumerate}
\item $\pi_{C}$ is 1--Lipschitz.
\item For $x\notin C$ and $y\in C$ such that $y\neq\pi_{C}(x)$, we have $\angle_{\pi_{C}(x)}(x,y)\ge\frac{\pi}{2}$.
\end{enumerate}
\end{lem} 

See Chapter II.2 of \cite{MR1744486} for a proof of the above lemma.

Two convex subset $C_ {1}$ and $C_{2}$ are \textit{parallel} if $d(\cdot,C_{1})|_{C_{2}}$ and $d(\cdot,C_{2})|_{C_{1}}$ are constant. In this case, the convex hull of $C_{1}$ and $C_{2}$ is isometric to $C_{1}\times[0,d(C_{1},C_{2})]$. Moreover, $\pi_{C_{1}}|_{C_{2}}$ and $\pi_{C_{2}}|_{C_{1}}$ are isometric inverse to each other (see \cite[Section 2.3.3]{kleiner1997rigidity} or \cite[Chapter II.2]{MR1744486}).

Let $Y\subset X$ be a closed convex subset. We define the \textit{parallel set} of $Y$, denoted by $P_{Y}$, to be the union of all convex subsets which are parallel to $Y$. $P_{Y}$ is not a convex set in general, but when $Y$ has geodesic extension property, $P_{Y}$ is closed and convex. 

Now we turn to $CAT(1)$ spaces. In this paper, $CAT(1)$ spaces is assumed to have diameter $\le\pi$ (we truncate the length metric on the space by $\pi$). And we say a subset of a $CAT(1)$ space is \textit{convex} if it is $\pi$--convex.

For a $CAT(1)$ space $Y$ and $p\in Y$, $K\subset Y$, we define the \textit{antipodal set of z in $K$} to be $Ant(p,K):=\{v\in K\ |\ d(v,p)=\pi\}$.

Let $Y$ and $Z$ be two metric spaces. Their \textit{spherical join}, denoted by $Y\ast Z$, is the quotient space of $Y\times Z\times[0,\frac{\pi}{2}]$ under the identifications $(y,z_{1},0)\sim (y,z_{2},0)$ and $(y_{1},z,\frac{\pi}{2})\sim(y_{2},z,\frac{\pi}{2})$. One can write the elements in $Y\ast Z$ as formal sum $(\cos\alpha)y+(\sin\alpha)z$ where $\alpha\in[0,\frac{\pi}{2}]$, $y\in Y$ and $z\in Z$. Let $w_{1}=(\cos\alpha_{1})y_{1}+(\sin\alpha_{1})z_{1}$ and $w_{2}=(\cos\alpha_{2})y_{2}+(\sin\alpha_{2})z_{2}$. Their distance in $Y\ast Z$ is defined to be
\begin{equation*}
d_{Y\ast Z}(w_{1},w_{2})=\cos\alpha_{1}\cos\alpha_{2}\cos(d^{\pi}_{Y}(y_{1},y_{2}))+\sin\alpha_{1}\sin\alpha_{2}\sin(d^{\pi}_{Z}(z_{1},z_{2}))\,,
\end{equation*}
here $d^{\pi}_{Y}$ is the metric on $Y$ truncated by $\pi$, similarly for $d^{\pi}_{Z}$.

When $Y$ is only one point, $Y\ast Z$ is the spherical cone over $Z$. When $Y$ is two point with distance $\pi$ from each other, $Y\ast Z$ is the spherical suspension of $Z$. The spherical join of two $CAT(1)$ spaces is still $CAT(1)$.

\begin{definition}[cell structure on the link]
\label{cell structure}
Let $X$ be a $M_{\kappa}$--polyhedral complex and pick point $x\in X$. Suppose $\sigma_x$ is the unique closed cell which contains $x$ as its interior point. Then $Lk(x,X)$ is isometric to $Lk(x,\sigma_x)\ast Lk(\sigma_x,X)=\Bbb S^{k-1}\ast Lk(\sigma_x,X)$, here $k$ is the dimension of $\sigma_x$. Note that $Lk(\sigma_x,X)$ has a natural $M_1$--polyhedral complex structure which is induced from the ambient space $X$. 

When $X$ is made of Euclidean rectangles, $Lk(\sigma_x,X)$ is an all-right spherical complex. Moreover, there is a canonical way to triangulate $Lk(x,\sigma_x)$ into an all-right spherical complex which is isomorphic to an octahedron as simplicial complexes. The vertices of $Lk(x,\sigma_x)$ come from segments passing through $x$ which are parallel to edges of $\sigma_x$. Thus $Lk(\sigma_x,X)$ has a natural all-right spherical complex structure. In general, there is no canonical way to triangulate $Lk(x,\sigma_x)$. However, there are still cases when we want to treat $Lk(x,X)$ as a piecewise spherical complex. In such cases, one can pick an arbitrary all-right spherical complex structure on $Lk(x,\sigma_x)$.

If $X$ is $CAT(0)$, then we can identify $\Sigma_{x}X$ with $Lk(x,X)$. In this case, $\Sigma_{x}X$ is understood to be equipped with the above polyhedral complex structure.
\end{definition}

Any two points of distance less than $\pi$ in a $CAT(1)$ space are joined by a unique geodesic. A generalization of this fact would be:

\begin{lem}
\label{2.8}
Let $Y$ be $CAT(1)$ and let $\Delta\subset Y$ be an isometrically embedded spherical $k$--simplex with its vertices denoted by $\{v_{i}\}_{i=0}^{k}$. Pick $v\in \Delta$ and $v'\in Y$. If $d(v',v_{i})\leq d(v,v_{i})$ for all $i$, then $v=v'$.
\end{lem}
By spherical simplexes, we always means those which are not too large, i.e. they are contained in an open hemisphere.

\begin{proof}
We prove by induction. When $k=1$, it follows from the uniqueness of geodesic. In general, since $\bigtriangleup=\bigtriangleup_{1}\ast\bigtriangleup_{2}$ where $\bigtriangleup_{1}$ is spanned by vertices $\{v_{i}\}_{i=0}^{k-2}$ and $\bigtriangleup_{2}$ is spanned by $v_{k-1}$ and $v_{k}$, there exists $w\in\bigtriangleup_{2}$ such that $v\in \bigtriangleup_{1}\ast\{w\}$. Triangle comparison implies $d(v',w)\leq d(v,w)$, so we can apply the induction assumption to the $(k-1)$-simplex $\bigtriangleup_{1}\ast\{w\}$, which implies $v=v'$.
\end{proof}

\begin{lem}
\label{2.9}
Let $Y$ be a $CAT(1)$ piecewise spherical complex with finitely many isometry types of cells, and let $K\subset Y$ be a subcomplex which is a spherical suspension (in the induced metric). Pick a suspension point $v\in K$, then all points in $Supp(v)$ are suspension points of $K$ and we have splitting $K=\Bbb S^{k}\ast K'$, here $k=\dim(Supp(v))$ and $\Bbb S^{k}$ is the standard sphere of dimension $k$. 
\end{lem}

\begin{proof}
By Theorem \ref{2.1}, $v$ has a small neighbourhood isometric to the $\epsilon$--ball centered at the cone point in the spherical cone over $\Sigma_{v}K$. Since $v$ is a suspension point, $K=\Bbb S^{0}\ast Lk(v,K)=\Bbb S^{0}\ast\Sigma_{v}K$. However, $\Sigma_{v}K=\Sigma_{v}Supp(v)\ast K'=\Bbb S^{k-1}\ast K'$ for some $K'$, thus $K=\Bbb S^{k}\ast K'$. Also every point in $Supp(v)$ belongs to the $\Bbb S^{k}$--factor, hence is a suspension point.
\end{proof}

\subsection{$CAT(0)$ cube complexes}
\textbf{All cube complexes in this paper are assumed to be finite dimensional.}

Every cube complex $X$ (a polyhedral complex whose building blocks are cubes) have a canonical cubical metric: endow each $n$--cube with the standard metric of unit cube in Euclidean $n$--space $\Bbb E^{n}$, then glue these cubes together to obtain a piecewise Euclidean metric on $X$. This metric is complete and geodesic if $X$ is of finite dimension and is $CAT(0)$ if the link of each vertex is a flag complex \cite{MR1744486,MR919829}.

Now we come to the notion of hyperplane, which is the cubical analogue of \textquotedblleft track\textquotedblright\  introduced in \cite{dunwoody1985accessibility}. A \textit{hyperplane} $h$ in a cube complex $X$ is a subset such that 
\begin{enumerate}
\item $h$ is connected.
\item For each cube $C\subset X$, $h\cap C$ is either empty or a union of mid-cubes of $C$. 
\item $h$ is minimal, i.e. if there exists $h'\subset h$ satisfying (1) and (2), then $h=h'$.
\end{enumerate}
Recall that a \textit{mid-cube} of $C=[0,1]^{n}$ is a subset of form $f_{i}^{-1}(1/2)$, where $f_{i}$ is one of the coordinate functions.

For each edge $e\in X$, there exists a unique hyperplane which intersects $e$ in one point. This is called the hyperplane \textit{dual} to the edge $e$. Following \cite{MR2377497}, we say a hyperplane $h$ \textit{self-intersects} if there exists a cube $C$ such $C\cap h$ contains at least two different mid-cubes. A hyperplane $h$ \textit{self-osculates} if there exist two different edges $e_{1}$ and $e_{2}$ such that (1) $e_{1}\cap e_{2}\neq\emptyset$; (2) $e_{1}$ and $e_{2}$ are not consecutive edges in a $2$--cube; (3) $e_{i}\cap h\neq\emptyset$ for $i=1,2$.

Let $X$ be a $CAT(0)$ cube complex, and let $e\subset X$ be an edge. Denote the hyperplane dual to $e$ by $h_{e}$. Suppose $\pi_{e}\co X\to e\cong [0,1]$ is the $CAT(0)$ projection. It is known that: 
\begin{enumerate}
\item $h_{e}$ is embedded, i.e. the intersection of $h_{e}$ with every cube in $X$ is either a mid-cube, or an empty set. 
\item $h_{e}$ is a convex subset of $X$ and $h_{e}$ with the induced cell structure from $X$ is also a $CAT(0)$ cube complex.
\item $h_{e}=\pi^{-1}_{e}(1/2)$.
\item $X\setminus h_{e}$ has exactly two connected components, they are called \textit{halfspaces}.
\item If $N_{h}$ is the union of closed cells in $X$ which has nontrivial intersection with $h_{e}$, then $N_{h}$ is a convex subcomplex of $X$ and $N_{h}$ is isometric to $h_{e}\times [0,1]$. We call $N_{h}$ the \textit{carrier} of $h_{e}$. Note that $N_{h}=P_{e}$ ($P_{e}$ is the parallel set of $e$). 
\end{enumerate}

We refer to \cite{MR1347406} for more information about hyperplanes.
 
Now we investigate the coarse intersection of convex subcomplexes. The following lemma adjusts Lemma \ref{2.4} in \cite{MR2421136} to our cubical setting. 
\begin{lem}
\label{2.10}
Let $X$ be a $CAT(0)$ cube complex of dimension $n$, and let $C_{1}$, $C_{2}$ be convex subcomplexes. Suppose $\bigtriangleup=d(C_{1},C_{2})$, $Y_{1}=\{y\in C_{1}\mid d(y,C_{2})=\bigtriangleup\}$ and $Y_{2}=\{y\in C_{2}\mid d(y,C_{1})=\bigtriangleup\}$. Then:
\begin{enumerate}
\item $Y_{1}$ and $Y_{2}$ are not empty.
\item $Y_{1}$ and $Y_{2}$ are convex; $\pi_{C_{1}}$ map $Y_{2}$ isometrically onto $Y_{1}$ and $\pi_{C_{2}}$ map $Y_{1}$ isometrically onto $Y_{2}$; the convex hull of $Y_{1}\cup Y_{2}$ is isometric to $Y_{1}\times [0,\bigtriangleup]$.
\item $Y_{1}$ and $Y_{2}$ are subcomplexes.
\item There exist $A=A(\Delta,n,\epsilon)$ such that if $p_{1}\in C_{1}$, $p_{2}\in C_{2}$ and $d(p_{1},Y_{1})\geq \epsilon>0$, $d(p_{2},Y_{2})\geq \epsilon>0$, then 
\begin{equation}
\label{2.11}
d(p_{1}, C_{2})\geq \bigtriangleup + Ad(p_{1},Y_{1})\,,\ d(p_{2}, C_{1})\geq \bigtriangleup + Ad(p_{2},Y_{2})\,.
\end{equation}
\end{enumerate}
\end{lem}

\begin{proof}
For the first assertion, since $X$ has finite dimension, $X$ has only finitely many isometry types of cells, we use the \textquotedblleft quasi-compact\textquotedblright\ argument of Bridson (\cite{MR2686786}). Suppose we have sequences of points $\{x_{n}\}_{n=1}^{\infty}$ in $C_{1}$ and $\{y_{n}\}_{n=1}^{\infty}$ in $C_{2}$ such that
\begin{equation}
\label{2.12}
d(x_{n},y_{n})<\bigtriangleup +\frac{1}{n}\,.
\end{equation}
Then by Lemma \ref{2.3}, there exists an integer $N$ such that for every $n$, the geodesic joining $x_{n}$ and $y_{n}$ is contained in a subcomplex $K_{n}$ which is a union of at most $N$ closed cell. Denote $C_{1n}=C_{1}\cap K_{n}$ and $C_{2n}=C_{2}\cap K_{n}$, which are also subcomplexes. Since there are only finitely many isomorphisms types among $\{K_{n}\}_{n=1}^{\infty}$, we can assume, up to a subsequence, that there exist a finite complex $K_{\infty}$ and subcomplex $C_{1\infty}$, $C_{2\infty}$ of $K_{\infty}$ such that for any $n$, there is a simplicial isomorphism $\varphi_{n}\co K_{n}\to K$ with $\varphi_{n}(C_{1n})=C_{1\infty}$ and $\varphi_{n}(C_{2n})=C_{2\infty}$. By $(\ref{2.12})$, $d_{K_{\infty}}(C_{1\infty}, C_{2\infty})\leq\bigtriangleup$ in the intrinsic metric of $K_{\infty}$, so there exist $x_{\infty}\in C_{1\infty}$ and $y_{\infty}\in C_{2\infty}$
such that $d_{K_{\infty}}(x_{\infty},y_{\infty})\leq \bigtriangleup$ by compactness of $K_{\infty}$. It follows that $d_{X}(\varphi_{n}^{-1}(x_{\infty}),\varphi_{n}^{-1}(y_{\infty}))\leq d_{K_{n}}(\varphi_{n}^{-1}(x_{\infty}),\varphi_{n}^{-1}(y_{\infty}))\leq\bigtriangleup$.

We prove (4) with $\epsilon=1$, the other cases are similar. A similar argument as above implies that there is a constant $A>0$, such that if $x\in C_{1}$ and $d(x,Y_{1})=1$, then $d(x,C_{2})>A+\bigtriangleup$. Note that the combinatorial complexity depends on $\Delta$ and $n$, so $A$ also depends on $\Delta$ and $n$. Now for any $p_{1}\in C_{1}$ and $d(p_{1},Y_{1})\geq 1$, let $p_{0}=\pi_{Y_{1}}(p_{1})$ and let $l\co[0,d(p_{0},p_{1})]\to X$ be the unit speed geodesic from $p_{0}$ to $p_{1}$. We have $l(1)\in \{x\in C_{1}\mid d(x, Y_{1})=1\}$, so $d(l(1),C_{2}))>A+\bigtriangleup$ while $d(l(0),C_{2}))=\bigtriangleup$. Then $(\ref{2.11})$ follows from the convexity of the function $d(\cdot,C_{2})$.

The second assertion is a standard fact in Chapter II.2 of \cite{MR1744486}.

To prove $(3)$, it suffices to prove for every $y_{1}\in Y_{1}$, we have $Supp(y_{1})\in Y_{1}$. Denote $y_{2}=\pi_{C_{2}}(y_{1})\in Y_{2}$ (hence $y_{1}=\pi_{C_{1}}(y_{2})$ by $(2)$) and $l\co[0,\Delta]\to X$ the unit speed geodesic from $y_{1}$ to $y_{2}$. Recall that we use $l^{-}(t)$ and $l^{+}(t)$ to denote the incoming and outcoming direction of $l$ in $\Sigma_{l(t)}X$ for $t\in [0,\Delta]$. Our goal is to construct a \textquotedblleft parallel transport\textquotedblright\ of $Supp(y_{1})$ (which is a $k$--cube) along $l$. 

Since $X$ has only finitely many isometry types of cells, $l$ is contained in finite union of closed cells, and we can find a sequence of closed cubes $\{B_{i}\}_{i=1}^{N}$ and $0=t_{0}<t_{1}<...<t_{N-1}<t_{N}=\bigtriangleup$ such that each $B_{i}$ contains $\{l(t)\mid t_{i-1}<t<t_{i}\}$ as interior points. We denote $Supp(y_{1})$ by $\Box_{t_{0}}$ from now on.

\textit{Starting:} At $l(0)=y_{1}$, we have splitting $\Sigma_{y_{1}}X=\Sigma_{y_{1}}\Box_{t_{0}}\ast K_{1}$ for some convex subset $K_{1}\subset \Sigma_{y_{1}}X$. Since $y_{1}=\pi_{C_{1}}(y_{2})$ and $\Box_{t_{0}}\subset C_{1}$, by Lemma \ref{2.6} we know $d_{\Sigma _{y_{1}}X}(l^{+}(t_{0}),\Sigma _{y_{1}}\Box_{t_{0}})\geq\frac{\pi}{2}$. Thus $l^{+}(t_{0})\in K_{1}$ and $d_{\Sigma _{y_{1}}X}(v,l^{+}(t_{0}))=\frac{\pi}{2}$ for any $v\in\Sigma _{y_{1}}\Box_{t_{0}}$. It follows that the segment $B_{1}\cap l$ is orthogonal to $\Box_{t_{0}}$ in $B_{1}$. Since $\Box_{t_{0}}$ is a sub-cube of $B_{1}$, by geometry of cube, there is an isometric embedding $e_{1}\co\Box_{t_{0}}\times [0,t_{1}]\to B_{1}$ with $e_{1}(y_{1},t)=l(t)$. Denote $\Box_{t_{1}}=e_{1}(\Box_{t_{0}}\times \{t_{1}\})$, then $l(t_{1})\in\Box_{t_{1}}\subseteq Supp(l(t_{1}))\subseteq B_{1}\cap B_{2}$. Note that $\Box_{t_{1}}$ is not necessarily a subcomplex of $B_{1}$ (or $B_{2}$), but it is always parallel to some sub-cube of $B_{1}$ (or $B_{2}$).

\textit{Continuing:} By construction we know $d_{\Sigma_{l(t_{1})}}(l^{-}(t_{1}),v)=\frac{\pi}{2}$ for $v\in\Sigma_{l(t_{1})}\Box_{t_{1}}$, so $d_{\Sigma_{l(t_{1})}}(l^{+}(t_{1}),\Sigma_{l(t_{1})}\Box_{t_{1}})\geq\frac{\pi}{2}$ (since $d_{\Sigma_{l(t_{1})}}(l^{-}(t_{1}),l^{+}(t_{1}))=\pi$). However, there is a splitting $\Sigma_{l(t_{1})}X=\Sigma_{l(t_{1})}\Box_{t_{1}}\ast K_{2}$ for some convex subset $K_{2}\subset \Sigma_{l(t_{1})}X$. Thus $l^{+}(t_{1})\in K_{2}$ and $d_{\Sigma _{l(t_{1})}X}(v,l^{+}(t_{1}))=\frac{\pi}{2}$ for any $v\in\Sigma _{l(t_{1})}\Box_{t_{1}}$. It follows that inside $B_2$, the segment $B_{2}\cap l$ is orthogonal to $\Box_{t_{1}}$. Recall that $\Box_{t_{1}}$ is parallel to a sub-cube of $B_{2}$, hence by geometry of cube, we have an isometric embedding $e_{2}\co\Box_{t_{1}}\times [t_{1},t_{2}]\to B_{2}$ with $e_{2}(y,t)=l(t)$ for some $y\in \Box_{t_{1}}$. Denote $\Box_{t_{2}}=e_{2}(\Box_{t_{1}}\times \{t_{2}\})$ and we know $\Box_{t_{2}}$ is parallel to some sub-cube of $B_{3}$, so one can proceed to construct isometric embedding $e_{3}$ as before. More generally, we can build $e_{i}\co\Box_{t_{i-1}}\times [t_{i-1},t_{i}]\to B_{i}$ with $e_{i}(y,t)=l(t)$ for some $y\in \Box_{t_{i-1}}$ and $\Box_{t_{i}}=e_{i}(\Box_{t_{i-1}}\times \{t_{i}\})$ inductively. Note that $l(t_{i})\in\Box_{t_{i}}\subseteq Supp(l(t_{i}))\subseteq B_{i}\cap B_{i+1}$ by construction.

\textit{Arriving:} Since $y_{2}=l(t_{N})\in B_{N}\cap C_{2}$ where $B_{N}$ and $C_{2}$ are subcomplexes, we have $l(t_{N})\in\Box_{t_{N}}\subseteq Supp(l(t_{N}))\subseteq B_{N}\cap C_{2}$ by construction. Moreover, we can concatenate the embeddings $\{e_{i}\}_{i=1}^{N}$ constructed in the previous step to obtain a map $e\co\Box_{t_{0}}\times [0,\bigtriangleup]\to X$ such that
\begin{itemize}
\item $e(y,t)=l(t)$ for some $y\in \Box_{t_{0}}$;
\item $e(\Box_{t_{0}}\times \{0\})\subseteq C_{1}$;
\item $e(\Box_{t_{0}}\times \{\bigtriangleup\})\subseteq C_{2}$;
\item $e$ is $1$--lipschitz ($e$ is actually an isometric embedding, since $e$ is local isometric embedding by construction).
\end{itemize}
Therefore $d(y,C_{2})\leq \bigtriangleup$ for any $y\in\Box_{t_{0}}$ (recall that $Supp(y_{1})=\Box_{t_{0}}$), which implies assertion $(3)$.
\end{proof}

\begin{remark}\
\label{2.13}
\begin{enumerate}
\item By the same proof, items (1), (2) and (4) in the above lemma are true for piecewise Euclidean $CAT(0)$ complexes with finitely many isometry types of cells. However, (3) might not be true in such generality.
\item If $C_{1}$ and $C_{2}$ are orthant subcomplexes, then by item (2) and (3), $Y_{1}$ (or $Y_{2}$) is isomeric to $O\times\prod_{i=1}^{k}I_{i}$ where $O$ is an orthant and each $I_{i}$ is a finite interval. In other words, there exists an orthant subcomplex $O\subset X$ such that $d_{H}(Y_{1},O)<\infty$.
\item Equation (\ref{2.11}) implies for any $R_{1},R_{2}>0$, we have $N_{R_{1}}(C_{1})\cap N_{R_{2}}(C_{2})\subset N_{R'_{1}}(Y_{1})$ and $N_{R_{1}}(C_{1})\cap N_{R_{2}}(C_{2})\subset N_{R'_{2}}(Y_{2})$, here $R'_{1}=\min(1,(R_{1}+R_{2}-\Delta)/A)+R_{2}$ and $R'_{2}=\min(1,(R_{1}+R_{2}-\Delta)/A)+R_{1}$ ($A=A(\Delta,n,1)$). Moreover, $\partial_{T}C_{1}\cap\partial_{T}C_{2}=\partial_{T}Y_{1}=\partial_{T}Y_{2}$.
\end{enumerate}
\end{remark}

The last remark implies that $Y_{1}$ and $Y_{2}$ capture the information about how $C_{1}$ and $C_{2}$ intersect coarsely. We use the notation $\inc(C_{1},C_{2})=(Y_{1},Y_{2})$ to describe this situation, where $\mathcal{I}$ stands for the word \textquotedblleft intersect\textquotedblright. The next lemma gives a combinatorial description of $Y_{1}$ and $Y_{2}$.

\begin{lem}
\label{2.14}
Let $X$, $C_{1}$, $C_{2}$, $Y_{1}$ and $Y_{2}$ be as in Lemma \ref{2.10}. Pick an edge $e$ in $Y_{1}$ $($or $Y_{2})$, and let $h$ be the hyperplane dual to $e$. Then $h\cap C_{i}\neq\emptyset$ for $i=1,2$. Conversely, if a hyperplane $h'$ satisfies $h'\cap C_{i}\neq\emptyset$ for $i=1,2$, then $\inc(h'\cap C_{1},h'\cap C_{2})=(h'\cap Y_{1},h'\cap Y_{2})$ and $h'$ comes from the dual hyperplane of some edge $e'$ in $Y_{1}$ $($or $Y_{2})$.
\end{lem}

\begin{proof}
The first part of the lemma follows from the proof of Lemma \ref{2.10}. Let $\inc(h'\cap C_{1},h'\cap C_{2})=(Y'_{1},Y'_{2})$. Pick $x\in Y'_{1}$ and set $x'=\pi_{h'\cap C_{2}}(x)\in Y'_{2}$. Then $\pi_{h'\cap C_{1}}(x')=x$. We identify the carrier of $h'$ with $h'\times[0,1]$. Since $C_{i}$ is a subcomplex, $(h'\cap C_{i})\times(\frac{1}{2}-\epsilon,\frac{1}{2}+\epsilon)=C_{i}\cap(h'\times(\frac{1}{2}-\epsilon,\frac{1}{2}+\epsilon))$ for $i=1,2$ and $\epsilon<\frac{1}{2}$. Thus for any $y\in C_{2}$, $\angle_{x'}(x,y)\ge\pi/2$, which implies $x'=\pi_{C_{2}}(x)$. Similarly, $x=\pi_{C_{1}}(x')=\pi_{C_{1}}\circ\pi_{C_{2}}(x)$, hence $x\in Y_{1}$ and $Y'_{1}\subset Y_{1}$. By the same argument, $Y'_{2}\subset Y_{2}$, thus $Y'_{i}=Y_{i}\cap h'$ for $i=1,2$ and the lemma follows.
\end{proof}

\begin{definition}
\label{2.15}
We call an isometrically embedded orthant $O$ \textit{straight} if for any $x\in O$, $\Sigma_{x}O$ is a subcomplex of $\Sigma_{x}X$ (see Definition \ref{cell structure} for the cell structure on $\Sigma_{x}X$). In particular, if the orthant is $1$--dimensional, we will call it a \textit{straight geodesic ray}. Note that $O$ itself may not be a subcomplex.
\end{definition}

\begin{remark}
\label{2.16}
Any $k$--dimensional straight orthant $O\subset X$ is Hausdorff close to an orthant subcomplex of $X$.

To see this, let $k'=\max_{x\in O}\{\dim(Supp(x))\}$, we prove by induction on $k'-k$. The case $k'-k=0$ is clear. Assume $k'-k=m\ge 1$ and pick $x_{0}\in O$ such that $\dim(Supp(x_{0}))=k'$, then there exists $B\subset Supp(x_{0})$ such that $B\cong [0,1]^{k}$, $B$ is parallel to a $k$--dimensional sub-cube of $Supp(x_{0})$ and $O\cap Supp(x_{0})\subset B$. Choose line segment $e\cong [0,1]$ in $Supp(x_{0})$ such that $x_{0}\in e$, $e$ is orthogonal to $B$ and $e$ is parallel to some edge $e'$ of $Supp(x_{0})$. 

Suppose $h$ is the hyperplane dual to $e$ and suppose $N_{h}\cong e\times h$ is the carrier of $h$. For any other point $x\in O$, the segment $\overline{x_{0}x}$ is orthogonal to $e$ by our construction, thus there exists a point $y\in e$ such that $O\subset \{y\}\times h\subset N_{h}$. Now we can endow $\{y\}\times h$ with the induced cubical structure and use our induction hypothesis to find an orthant complex $O_{1}$ in the $k$--skeleton of $\{y\}\times h$ such that $d_{H}(O,O_{1})<\infty$. Since $N_{h}\cong e\times h$, we can slide $O_{1}$ along $e$ in $N_{h}$ to get an orthant subcomplex in the $k$--skeleton of $X$.
\end{remark}

\begin{lem}
\label{2.17}
Let X be a $CAT(0)$ cube complex. If $l_{1}$ and $l_{2}$ are two straight geodesic rays in $X$, then either $\angle_{T}(l_{1},l_{2})=0$, or $\angle_{T}(l_{1},l_{2})\geq \frac{\pi}{2}$.
\end{lem}

\begin{proof}
We can assume without loss of generality that $l_{1}$ and $l_{2}$ are in $1$--skeleton and $l_{1}(0)$ is a vertex of $X$. We parametrize these two geodesic rays by unit speed. Let $\{b_{m}\}_{m=1}^{\infty}$ be the collection of hyperplanes in $X$ such that $b_{m}\cap l_{1}=l_{1}(\frac{1}{2}+m)$, and let $h_{m}$ be the halfspace bounded by $b_{m}$ which contains $l_{1}$ up to a finite segment. Suppose $N_{m}$ is the carrier of $b_{m}$.

Suppose $l_{2}\cap b_{m}\neq\emptyset$ for infinitely many $m$. Since each $b_m$ separates $X$, there exists $m_0$ such that $l_{2}\cap b_{m}\neq\emptyset$ for all $m\ge m_{0}$. Recall that $l_{2}$ is in 1--skeleton, so for each $m\ge m_0$, there exists an edge $e_{m}$ such that $e_{m}\subset l_{2}$, $e_{m}\subset N_{m}$ and $e_{m}\cap b_{m}$ is a point. Consider the function $f(t)=d(l_{2}(t),l_{1})$ for $t\ge 0$, then $f$ is convex and there exists infinitely many intervals of unit length (they come from $e_{m}$ for $m\ge m_{0}$) such that $f$ restricted each interval is constant, so there exists $t_{0}$ such that $f|_{[t_{0},\infty)}$ is constant, which implies $\angle_{T}(l_{1},l_{2})=0$.

If $l_{2}\cap b_{m}\neq\emptyset$ for finitely many $m$, then there exists $m_{0}$ such that $h_{m_{0}}\cap l_{2}=\emptyset$, which implies the $CAT(0)$ projection of $l_{2}$ to $l_{1}$ is an finite segment. If $\angle_{T}(l_{1},l_{2})<\frac{\pi}{2}$, then $\pi_{l_{1}}(l_{2})$ is an infinite segment by Lemma \ref{2.4}, which is a contradiction, so $\angle_{T}(l_{1},l_{2})\geq \frac{\pi}{2}$.
\end{proof}

We will see later on that the subset of $\partial_{T}X$ which is responsible for the behavior of top dimensional quasiflats is spanned by those points represented by straight geodesic rays. The following lemma make the word \textquotedblleft span\textquotedblright\ precise.

\begin{lem}
\label{2.18}
Let $X$ be a $CAT(0)$ cube complex, and let $\{l_{i}\}_{i=1}^{k}$ be a collection of straight geodesic rays in $X$ emanating from the same base point $p$ such that $\angle_{T}(l_{i},l_{j})=\angle_{p}(l_{i},l_{j})=\frac{\pi}{2}$ for $i\neq j$. Then the convex hull of $\{l_{i}\}_{i=1}^{k}$ is a k--dimensional straight orthant.
\end{lem}

One may compare this lemma with Proposition 2.10 and Proposition 2.11 of \cite{behrstock2012cubulated}.
\begin{proof}
By Lemma \ref{2.5}, $l_{i}$ and $l_{j}$ together bound an isometrically embedded quarter plane for $i\neq j$. We prove the lemma by induction and assume the claim is true for $\{l_{i}\}_{i=1}^{k-1}$. We parametrize $l_{k}$ by arc length and denote by $O_{0}$ the straight orthant spanned by $\{l_{i}\}_{i=1}^{k-1}$, note that $O_{0}\cap l_{k}=p$. 

For $s>0$ and $1\leq i\leq k-1$, let $c_{i}$ be the geodesic ray such that (1) $c_{i}$ is in the quarter plane spanned by $l_{k}$ and $l_{i}$; (2) $c_{i}$ starts at $l_{k}(s)$; (3) $c_{i}$ is parallel to $l_{i}$. So $\angle_{T}(c_{i},c_{j})=\frac{\pi}{2}$ and $\angle_{l_{k}(s)}(c_{i},c_{j})\leq\frac{\pi}{2}$ for $i\neq j$. Note that $\{c_{i}\}_{i=1}^{k-1}$ are also straight geodesic rays, and $\{\log_{l_{k}(s)}c_{i}\}_{i=1}^{k-1}$ are distinct points in the 0--skeleton of $\Sigma_{l_{k}(s)}X$. It follows that actually $\angle_{l_{k}(s)}(c_{i},c_{j})=\frac{\pi}{2}$ for $i\neq j$. Hence by the induction assumption, there is a straight orthant $O_{s}$ spanned by $\{c_{i}\}_{i=1}^{k-1}$. 

By Lemma \ref{2.8}, $\partial_{T}O_{0}=\partial_{T}O_{s}$. Let $l\subset O_{s}$ be a unit speed geodesic ray emanating from $l_{k}(s)$. Then $d(l(t),O_{0})$ is a bounded convex function. Since $\Sigma_{l_k(s)}O_s$ is an all-right spherical simplex in $\Sigma_{l_k(s)}X$ spanned by $\{\log_{l_{k}(s)}c_{i}\}_{i=1}^{k-1}$, we have $\angle_{l_{k}(s)}(l(t),l_k(0))=\pi/2$ for any $t>0$. Similarly, $\angle_{l_{k}(0)}(y,l_k(s))=\pi/2$ for any $y\in O_{0}\setminus\{l_s(0)\}$. Hence by triangle comparison, $d(l(t),O_{0})$ attains its minimum at $t=0$. Thus $d(l(t),O_{0})$ has to be a constant function. Thus $d(x,O_{0})\equiv s$ for any $x\in O_{s}$, similarly $d(x,O_{s})\equiv s$ for any $x\in O_{0}$, which implies the convex hull of $O_{0}$ and $O_{s}$ is isometric to $O_{0}\times [0,s]$ (see e.g. Chapter II.2 of \cite{MR1744486}). Moreover, the convex hull of $O_{0}$ and $O_{s}$ is contained in the convex hull of convex hull of $O_{0}$ and $O_{s'}$ for $s\leq s'$. So the convex hull of $\{l_{i}\}_{i=1}^{k}$ is a straight orthant $O$.
\end{proof}

\section{Proper homology classes of bounded growth}
\label{support of homology class}
In this section we summary some results from \cite{bestvina2008quasiflats} about locally finite homology classes of certain polynomial growth and make some generalizations to adjust the results to our situation. 

\subsection{Proper homology and supports of homology classes}
\label{proper homology}
For arbitrary metric space $Z$, we define the proper (singular) homology of $Z$ with coefficients in Abelian group $G$, denoted $H^{\textmd{p}}_{\ast}(Z;G)$, as follows. Elements in the proper $n$--chain group $C^{\textmd{p}}_{n}(Z;G)$ are of form $\Sigma_{\lambda\in\Lambda}g_{\lambda}\sigma_{\lambda}$ ($\Lambda$ may be infinite, $g_{\lambda}\in G$ and $\sigma_{\lambda}$'s are singular $n$--simplices) such that for every bounded set $K\subset Z$, the set $\{\lambda\in\Lambda|g_{\lambda}\neq  Id\ and\  \sigma_{\lambda}(\bigtriangleup^{n})\cap K\neq\emptyset\}$ is finite. The usual boundary map gives rise to group homomorphism $\partial\co C^{\textmd{p}}_{n}(Z;G)\to C^{\textmd{p}}_{n-1}(Z;G)$, which yields a chain complex $C^{\textmd{p}}_{\ast}(Z;G)$. $H^{\textmd{p}}_{\ast}(Z;G)$ is defined to be the homology of this chain complex. 

We will use Greek letters $\alpha$, $\beta\cdots$ to denote (proper) singular chains. We denote the union of images of singular simplices in a (proper) singular chain $\alpha$ by Im $\alpha$. If $\alpha$ is a (proper) cycle, we denote the corresponding (proper) homology class by $[\alpha]$.

We also define the relative version of proper homology $H^{\textmd{p}}_{\ast}(Z,Y)$ for $Y\subset Z$ in a similar way ($Y$ is endowed with the induced metric). Then there is a long exact sequence
\begin{equation*}
\cdots \to H^{\textmd{p}}_{n}(Y)\to H^{\textmd{p}}_{n}(Z)\to H^{\textmd{p}}_{n}(Z,Y)\to H^{\textmd{p}}_{n-1}(Y)\to H^{\textmd{p}}_{n-1}(Z)\to\cdots
\end{equation*}
Moreover, by the usual procedure of subdividing the chains, we know excision holds. Namely for subspace $W$ such that the closure of $W$ is in the interior of $Y$, the map $H^{\textmd{p}}_{\ast}(Z-W,Y-W)\to H^{\textmd{p}}_{\ast}(Z,Y)$ induced by inclusion is an isomorphism. As a corollary, if $B\subset Z$ is bounded, then there is a natural isomorphism $H^{\textmd{p}}_{\ast}(Z,Z-A)\cong H_{\ast}(Z,Z-A)$, since we can replace the pair $(Z,Z-A)$ by $(O,O-B)$ by excision, where $O$ is a bounded open neighbourhood of $B$. Pick a point $z\in Z\setminus Y$, there is a homomorphism $i\co H^{\textmd{p}}_{k}(Z,Y)\to H^{\textmd{p}}_{k}(Z,Z\setminus\{z\})\cong H_{k}(Z,Z\setminus\{z\})$ induced by the inclusion of pairs $(Z,Y)\to (Z,Z-\{z\})$. The map $i$ is called the \textit{inclusion homomorphism}.

If $Z$ is also a simplicial complex or polyhedral complex, we can similarly define the proper simplicial (or cellular) homology by considering the former sum of simplexes or cells such that for every bounded subset $K\subset Z$, we have only finitely many terms which intersect $K$ non-trivially. The simplicial version (or the cellular version) of the homology theory is isomorphic to the singular version in a simplicial complex (or polyhedral complex) by the usual proof in algebraic topology.

The proper homology depends on the metric of the space, so it is not a topological invariant. By definition, every proper chain is locally finite and we have a group homomorphism $H^{\textmd{p}}_{\ast}(Z,Y)\to H^{\textmd{lf}}_{\ast}(Z,Y)$, where $H^{\textmd{lf}}_{\ast}(Z,Y)$ is the locally finite homology defined in \cite{bestvina2008quasiflats}. If $Z$ is a proper metric space, then these two homology theories are the same.

A continuous map $f\co Z_{1}\to Z_{2}$ is \textit{(metrically) proper} if the inverse image of every bounded subset is bounded. In this case, we have an induced map on proper homology $f_{\ast}\co H^{\textmd{p}}_{k}(Z_{1},G)\to H^{\textmd{p}}_{k}(Z_{2},G)$.

In the rest of this paper, we will always take $G=\Bbb Z/2\Bbb Z$ and omit $G$ when we write the homology. 

\begin{definition}
\label{3.1}
For $z\in Z\setminus Y$, let $i\co H^{\textmd{p}}_{k}(Z,Y)\to H_{k}(Z,Z\setminus\{z\})$ be the inclusion homomorphism defined as above. For $[\sigma]\in H^{\textmd{p}}_{k}(Z,Y)$, we define the \textit{support set} of $[\sigma]$, denoted $S_{[\sigma],Z,Y}$, to be $\{z\in Z\setminus Y\mid i_{\ast}[\sigma]\neq \textmd{Id}\}$. We will write $S_{[\sigma],Z}$ if $Y$ is empty. We will also use $S_{[\sigma]}$ to denote the support set if the underlying space $Z$ and $Y$ is clear.
\end{definition} 

It follows that $S_{[\sigma]}=(\cap_{[\sigma']=[\sigma]\in H^{\textmd{p}}_{k}(Z,Y)}\textmd{Im}\ \sigma')\setminus Y$.

If $Z\subset Z_{1}$, then $S_{[\sigma],Z,Y}\supseteq S_{[\sigma],Z_{1},Y}$. These two sets are equal if $Z$ is open in $Z_{1}$. If $Z$ is a polyhedral complex and $Y=\emptyset$, then the support set is always a subcomplex. In particular, if $[\sigma]\in H^{\textmd{p}}_{n}(Z)$ is a nontrivial top dimensional class, then $[\sigma]$ can be represented by a top dimensional polyhedral cycle, which implies the support set $S_{[\sigma]}\neq\emptyset$. But the support of a non-trivial class can be trivial if it is not top dimensional.

The support sets of (proper) homology class behave like the support sets of currents in the following situation:

\begin{lem}
\label{3.2}
Let $Z_{1}$ be a metric space of homological dimension $\leq n$, and let $Y_{1}\subset Z_{1}$ be a subspace. Pick $[\sigma]\in H^{\textmd{p}}_{n}(Z_{1},Y_{1})$. If $f\co(Z_{1},Y_{1})\to(Z_{2},Y_{2})$ is a proper map, then $S_{f_{\ast}[\sigma]}\subset f(S_{[\sigma]})$.
\end{lem}

Recall that $Z_{1}$ has ($\Bbb Z/2\Bbb Z$)--homological dimension $\le n$ if $H_{r}(U,V)=0$ for any $r>n$ and $U,V$ open in $Z_{1}$.
\begin{proof}
Pick $y\in S_{f_{\ast}[\sigma]}$, since $f^{-1}(y)$ is bounded, we have the following commutative diagram:
\begin{center}
$\begin{CD}
H^{\textmd{p}}_{n}(Z_{1},Y_{1})                            @>f_{\ast}>>        H^{\textmd{p}}_{n}(Z_{2},Y_{2})\\
@VVi_{\ast}V                                                              @VVi_{\ast}V\\
H_{n}(Z_{1},Z_{1}\setminus f^{-1}(y))          @>f_{\ast}>>        H_{n}(Z_{2},Z_{2}\setminus \{y\})
\end{CD}$
\end{center}

Thus if $S_{f_{\ast}[\sigma]}\neq\emptyset$, $[\sigma']=i_{\ast}[\sigma]\in H_{n}(Z_{1},Z_{1}\setminus f^{-1}(y))$ is nontrivial. It suffices to show there exists $x\in f^{-1}(y)$ such that $[\sigma']$ is nontrivial when viewed as an element in $H_{n}(Z_{1},Z_{1}\setminus \{x\})$.

Fix a singular chain $\sigma'\in C_{n}(Z_{1},Z_{1}\setminus f^{-1}(y))$ which represents $[\sigma']$. We argue by contradiction and assume $[\sigma']$ is trivial in $H_{n}(Z_{1},Z_{1}\setminus \{x\})$ for any $x\in f^{-1}(y)$. Let $K=f^{-1}(y)\cap \textmd{Im}\ \sigma'$. For each $x\in K$, there exists $\epsilon(x)>0$ such that $\bar{B}(x,2\epsilon(x))\cap \textmd{Im}\ \partial\sigma'=\emptyset$ and $[\sigma']$ is trivial in $H_{n}(Z_{1},Z_{1}\setminus \bar{B}(x,2\epsilon(x)))$. Since $f^{-1}(y)$ is bounded and closed, $K$ is compact and we can find finite points $\{x_{i}\}_{i=1}^{N}$ in $K$ such that $K\subset\cup_{i=1}^{N} B(x_{i},\epsilon(x_{i}))$. Let $U=\cup_{i=1}^{N} B(x_{i},\epsilon(x_{i}))$ and let $U'=(Z_{1}\setminus f^{-1}(y))\cup U$. Then $\textmd{Im}\ \sigma'\subset U'$ and $[\sigma']$ is trivial in $H_{n}(Z_{1},U')$. Put $U''=U'\setminus(\cup_{i=1}^{N}\bar{B}(x_{i},2\epsilon(x_{i})))$, then $\textmd{Im}\ \partial\sigma'\subset U''$ and $U''\subset Z_{1}\setminus f^{-1}(y)$. So if we can show $[\sigma']$ is trivial in $H_{n}(Z_{1},U'')$, then $[\sigma']$ must be trivial in $H_{n}(Z_{1},Z_{1}\setminus f^{-1}(y))$, which yields a contradiction.  

Let us assume $N=1$. Then there is a Mayer--Vietoris sequence:
\begin{align*}
&H_{n+1}(Z_{1},U'\cup (Z_{1}\setminus \bar{B}(x_{1},2\epsilon(x_{1})))\to H_{n}(Z_{1},U'') \\
&\to H_{n}(Z_{1},U')\oplus H_{n}(Z_{1},Z_{1}\setminus \bar{B}(x_{1},2\epsilon(x_{1})))
\end{align*}
The first term is trivial since $Z_{1}$ has homological dimension $\leq n$ and $[\sigma']$ is trivial in the last term by construction, so $[\sigma']$ has to be trivial in the second term. Using an induction argument, we can obtain the contradiction similarly for $N\ge 2$.
\end{proof}

\begin{remark}
\label{3.3}
\
\begin{enumerate}
\item The assumption on $Z_{1}$ is satisfied if $Z_{1}$ is a $CAT(\kappa)$ space of geometric dimension $\le n$, see \cite[Theorem A]{kleiner1999local}.
\item The assumption on $Z_{1}$ is satisfied if $Z_{1}$ is a locally-finite $n$--dimensional polyhedral complex (with topology of a cell complex) or $M_{k}$--polyhedral complex of finite shape, since such space supports a $CAT(1)$ metric which induces the same topology as its original metric.
\end{enumerate}
\end{remark}

\subsection{The growth condition}
\label{growth condition}
In Section \ref{growth condition} and Section \ref{splitting}, $Y$ will be a piecewise Euclidean $CAT(0)$ complex of dimension $n$. The following result shows every top dimensional quasiflat in $Y$ is Hausdorff close to the support set of some proper homology class. Therefore to understand quasiflats, we could focus on the support sets, which have nice local and asymptotic properties.
\begin{lem}
\label{3.4}
(Lemma 4.3 of \cite{bestvina2008quasiflats}) If $Q\subset Y$ be an $(L,A)$--quasiflat of dimension $n$, then there exists $[\sigma] \in H^{\textmd{p}}_{n}(Y)$ satisfying the following conditions:
\begin{enumerate}
\item There exists constant $D=D(L,A)$ such that $d_{H}(S,Q)\le D$, where $S$ is the support set of $[\sigma]$.
\item There exists $a=a(L,A)$ such that for every $p\in Y$, 
\begin{equation}
\label{3.5}
\mathcal{H}^{n}(B(p,r)\cap S)\leq a(1+r)^{n}\,.
\end{equation} 
\end{enumerate}
\end{lem}
Here $\mathcal{H}^{n}$ denotes the $n$--dimensional Hausdorff measure and $d_{H}$ denotes the Hausdorff distance.

Since $Y$ is uniformly contractible, we can approximate the $(L,A)$--quasi-isometric embedding $q\co\Bbb R^{n}\to Y$ by a Lipschitz $(L,A)$--quasi-isometric embedding $\widetilde{q}$, which is proper. Then $[\sigma]$ is chosen to be the pushforward of the fundamental class of $\Bbb R^{n}$ under $\widetilde{q}$.

The support set of top dimensional support set enjoys the following geodesic extension property.
\begin{lem}
\label{3.6}
(Lemma 3.1 of \cite{bestvina2008quasiflats}) Let $S$ be the support set of some top dimensional proper homology class in $Y$. Pick arbitrary $p\in Y$ and $y\in S$. Then there is a geodesic ray $\overline{y\xi}\subset S$ which fits together with the geodesic segment $\overline{py}$ to form a geodesic ray $\overline{p\xi}$.
\end{lem}

Note that this lemma does not imply $S$ is convex (see Remark 3.2 in \cite{bestvina2008quasiflats}), however, we still can define the Tits boundary of $S$.

\begin{definition}
\label{3.7}
Let $Z$ be a $CAT(0)$ space and let $A\subset Z$ be any subset. We define the \textit{Tits boundary} of $A$, denoted $\partial_{T}A$, to be the set of points $\xi\in\partial_{T}Z$ such that there exists a geodesic ray $\overline{x\xi}$ such that $\overline{x\xi}\subset A$. $\partial_{T}A$ is endowed with the usual Tits metric. We define the \textit{Tits cone} of $A$, denoted $C_{T}A$, to be the Euclidean cone over $\partial_{T}A$.
\end{definition} 

Let $S$ be as in Lemma \ref{3.6}. Then $\partial_{T}S$ is nonempty if $S$ is nonempty.

We can state a similar version of geodesic extension property for the link $\Sigma_{y}S\subset\Sigma_{y}Y$ ($y\in S$).

\begin{lem}
\label{link of support set}
Let $S$ be as in Lemma \ref{3.6}. Then for any point $y\in Y$, $\Sigma_{y}S$ is the support set of some top dimensional homology class in $\Sigma_{y}Y$.
\end{lem}

\begin{proof}
By subdividing $Y$ in an appropriate way, we assume $y$ is a vertex of $Y$. Suppose $S=S_{[\sigma]}$. We can represent $[\sigma]$ as a cellular cycle $\sigma=\Sigma_{\lambda\in\Lambda}\eta_{\lambda}$, where $\eta_{\lambda}$'s are closed top dimensional cells in $Y$ (recall that we are using $\Bbb Z/2\Bbb Z$ coefficients, so all the coefficients are either $0$ or $1$). Then $S=\cup_{\lambda\in\Lambda}\eta_{\lambda}$. Let $\Lambda_y=\{\lambda\in\Lambda\mid y\in\eta_{\lambda}\}$. Since $\eta$ is a cycle, $\eta_y=\Sigma_{\lambda\in\Lambda_{y}}Lk(y,\eta_{\lambda})$ is a top dimensional cycle in $Lk(y,Y)\cong \Sigma_{y}Y$. Moreover, $S_{[\eta_y]}=\cup_{\lambda\in\Lambda_{y}}Lk(y,\eta_{\lambda})=Lk(y,S)$.
\end{proof}

\begin{lem}
\label{3.8}
Let $K$ be a $k$--dimensional $CAT(1)$ piecewise spherical complex, and let $K'\subset K$ be the support set of a top dimensional homology class. Pick arbitrary $w\in K$, $v\in K'$, and suppose $\overline{wv}$ is a local geodesic joining $v$ and $w$. Then there is a (nontrivial) local geodesic segment $\overline{vv'}\subset K'$ which fits together with $\overline{wv}$ to form a local geodesic segment $\overline{wv'}$, moreover, $length(\overline{vv'})$ can be as large as we want.
\end{lem}

Now we turn to the global properties of $S$. Since we are in a $CAT(0)$ space, for any $p\in Y$ and $0<r\leq R$, we have a map $\Phi_{r,R}\co B(p,R)\to B(p,r)$ by contracting points toward $p$ by a factor of $\frac{r}{R}$. This contracting map together with Lemma \ref{3.6} implies $B(p,r)\cap S\subset \Phi_{r,R}(B(p,R)\cap S)$ (Corollary 3.3, item 1 of \cite{bestvina2008quasiflats}).

Since $\Phi_{r,R}$ is $\frac{r}{R}$--Lipschitz, we have the following result:

\begin{thm}
\label{3.9}
(Corollary 3.3 of \cite{bestvina2008quasiflats}) Let $S$ be the support set of some top dimensional proper homology class in $Y$, and let $n=\dim (Y)$. Then:
\begin{enumerate}
\item (Monotonicity of density) For all $0\leq r\leq R$, 
\begin{equation}
\label{3.10}
\frac{\mathcal{H}^{n}(B(p,r)\cap S)}{r^{n}}\leq \frac{\mathcal{H}^{n}(B(p,R)\cap S)}{R^{n}}\,.
\end{equation}
\item (Lower density bound) For all $p\in S$, $r>0$,
\begin{equation}
\label{3.11}
\mathcal{H}^{n}(B(p,r)\cap S)\geq \omega_{n} r^{n}\,,
\end{equation}
\end{enumerate}
with equality only if $B(p,r)\cap S$ is isometric to an $r$--ball in $\Bbb E^{n}$, here $\omega_{n}$ is the volume of an $n$--dimensional Euclidean ball of radius $1$.
\end{thm}

From $(\ref{3.10})$ we know the quantity 
\begin{equation}
\label{3.12}
\frac{\mathcal{H}^{n}(B(p,r)\cap S)}{r^{n}}
\end{equation} 
is monotone increasing with respect to $r$, and $(\ref{3.5})$ tells us if $S$ comes from a top dimensional quasiflat, then $(\ref{3.12})$ is bounded above by some constant. Thus the limit exists and is finite as $r\to\infty$. More generally, we will consider those top dimensional proper homology classes whose support sets $S$ satisfy 
\begin{equation}
\label{3.13}
\lim_{r\to +\infty}\frac{\mathcal{H}^{n}(B(p,r)\cap S)}{r^{n}}<\infty\,.
\end{equation}
Here $n=\dim(Y)$. We call them \textit{proper homology classes of $Cr^{n}$ growth}. These classes form a subgroup of $H^{\textmd{p}}_{n}(Y)$, which will be denoted by $H^{\textmd{p}}_{n,n}(Y)$.

The following lemma can be proved by a packing argument:

\begin{lem}
\label{3.14}
(Lemma 3.12 of \cite{bestvina2008quasiflats}) Pick $[\sigma]\in H^{\textmd{p}}_{n,n}(Y)$ and let $S=S_{[\sigma]}$. Then given base point $p\in Y$, for all $\epsilon>0$ there is an $N$ such that for all $r\geq 0$, $B(p,r)\cap S$ does not contain an $\epsilon r$--separated subset of cardinality greater than $N$.
\end{lem}

\begin{lem}
\label{3.15}
Let $S$ and $p$ be as in Lemma \ref{3.14}. Denote the cone point in $C_{T}S$ by $o$. Then 
\begin{equation}
\label{3.16}
\lim_{r\to +\infty}d_{GH}(\frac{1}{r}(B(p,r)\cap S),B(o,1))=0\,.
\end{equation}
\end{lem}

Here $d_{GH}$ denotes the Gromov--Hausdorff distance, $B(o,1)$ is the ball of radius 1 in $C_{T}S$ centered at $o$ and $\frac{1}{r}(S\cap B(p,r))$ means we rescale the metric on $S\cap B(x,r)$ by a factor $\frac{1}{r}$.
\begin{proof}
We follow the argument in \cite{bestvina2008quasiflats}. It suffices to prove for any $\epsilon>0$, there exists $R>0$ such that for any $r>R$, we can find an $\epsilon$--isometry between $\frac{1}{r}(B(p,r)\cap S)$ and $B(o,1)$.

For $r>0$, we denote the maximal cardinality of $\epsilon r$--separated net in $B(p,r)\cap S$ by $m_{r}$. By Lemma \ref{3.14}, there exists $N_{0}$ such that $m_{r}\le N_{0}$ for all $r$. Pick $R_{1}$ such that $m_{r}\le m_{R_{1}}$ for all $r\neq R_{1}$ and denote the corresponding $\epsilon R_{1}$--net in $B(p,R_{1})\cap S$ by $\{x_{i}\}_{i=1}^{N}$. By Lemma \ref{3.6}, for each $i$, we can extent the geodesic $\overline{px_{i}}$ to obtain a geodesic ray $\overline{p\xi_{i}}$ such that $\overline{x_{i}\xi_{i}}\subset S$. Let $l_{i}\co[0,\infty)\to Y$ be a constant speed geodesic ray joining $p$ and $\xi_{i}$ such that $l_{i}(R_{1})=x_{i}$ and $l_{i}(0)=p$. 

Since the quantity $d(l_{i}(t),l_{j}(t))/t$ is monotone increasing, $\{l_{i}(t)\}_{i=1}^{N}$ is a maximal $\epsilon t$--separated net in $B(p,R_{1})\cap S$ for $t\geq R_{1}$. We pick $R>R_{1}$ such that for all $t>R$ and $1\leq i,j\leq N$, we have
\begin{equation}
\label{3.17}
|\frac{d(l_{i}(t),l_{j}(t))}{t}-\lim_{t\to +\infty}\frac{d(l_{i}(t),l_{j}(t))}{t}|<\epsilon.
\end{equation}
Now we fix $t>R$ and define a map such that for each $i$, $l_{i}(t)\in B(p,t)\cap S$ is mapped to the point $y_{i}\in B(o,1)\subset C_{T}S$ such that $y_{i}\in\overline{o\xi_{i}}$ and $d(y_{i},o)=d(l_{i}(t),p)/t$. It follows from (\ref{3.17})
\begin{equation}
\label{3.18}
|\frac{d(l_{i}(t),l_{j}(t))}{t}-d(y_{i},y_{j})|<\epsilon\,.
\end{equation}
We claim $\{y_{i}\}_{i=1}^{N}$ is an $\epsilon$--net in $B(o,1)$.

Pick arbitrary $y\in B(o,1)$ and suppose $y\in\overline{o\xi}$ for $\xi\in\partial_{T}S$. We parametrize the geodesic ray $\overline{p\xi}$ by constant speed $=d(y,o)$ and denote this ray by $l$. Since there exists a geodesic $\overline{p'\xi}\subset S$ such that $d_{H}(\overline{p\xi},\overline{p'\xi})=C<\infty$, we can find $x\in\overline{p'\xi}\subset S$ with $d(x,l(t))<C$ for every $t$. Thus $x\in B(p,td(y,o)+C)\cap S\subset B(p,t+C)\cap S$, which implies there exists some $i$ such that $d(l_{i}(t+C),x)\leq\epsilon(t+C)$. These estimates together with $d(l_{i}(t+C),l_{i}(t))\leq C$ ($l_{i}$ has speed $\leq 1$) implies:
\begin{equation}
\label{3.19}
d(l(t),l_{i}(t))\leq \epsilon t+(2+\epsilon)C\,.
\end{equation}
Here $i$ might depend on $t$, but we can choose a sequence $\{t_{k}\}_{k=1}^{\infty}$ such that $t_{k}\to+\infty$ and 
\begin{equation}
\label{3.20}
d(l(t_{k}),l_{i_{0}}(t_{k}))\leq \epsilon t+(2+\epsilon)C
\end{equation}
for every $k$ with $i_{0}$ fixed, thus
\begin{equation}
\label{3.21}
d(y,y_{i_{0}})=\lim_{k\to+\infty}\frac{d(l(t_{k}),l_{i_{0}}(t_{k}))}{t_{k}}\leq\epsilon\,.
\end{equation}
So $\{y_{i}\}_{i=1}^{N}$ is an $\epsilon$--net in $B(o,1)$, this fact together with (\ref{3.18}) give us the $\epsilon$--isometry as required.
\end{proof}

\begin{remark}
\label{3.22}
\
\begin{enumerate}
\item Define $\partial_{p,r}S=\{\xi\in \partial_{T}S|\ \overline{p\xi}\subset B(p,r)\cup S\}$, then the above proof shows 
\begin{equation}
\label{3.23}
\lim_{r\to+\infty}d_{H}(\partial_{p,r}S,\partial_{T}S)=0\,.
\end{equation}
\item $\partial_{T}S$ has similar behavior to the Tits boundary of a convex subset in the following aspect. Let $l\co[0,\infty)\to Y$ be a constant speed geodesic ray. If there exist a constant $C<\infty$ and a sequence $t_{i}\to+\infty$ such that $d(l(t_{i}),S)<C$, then $\partial_{T}l$ is an accumulation point of $\partial_{T}S$. The proof is similar to the above argument.
\end{enumerate}
\end{remark}

\subsection{$\epsilon$--splitting}
\label{splitting}
As we have seen from Lemma \ref{3.15}, the growth bound (\ref{3.13}) implies that $S$ looks more and more like a cone if one observes $S$ from a more and more far way point (this is called asymptotic conicality in \cite{bestvina2008quasiflats}). So one would expect some regularity of $S$ near infinity, the following key lemma from \cite{bestvina2008quasiflats} will be our starting point.
\begin{lem}
\label{3.24}
(Lemma 3.13 of \cite{bestvina2008quasiflats}) Let $S$ and $p$ be as in Lemma \ref{3.14}. Then for all $\beta>0$ there is an $r<\infty$ such that if $x\in S\setminus B(p,r)$, then the antipodal set
\begin{equation}
\label{3.25}
diam(Ant(\log_{x}p,\Sigma_{x}S))<\beta\,.
\end{equation}
\end{lem}

The proof of this lemma in \cite{bestvina2008quasiflats} actually shows something more. Given base point $p\in Y$ and $x\in S$, we define the \textit{antipode at $\infty$} of $\log_{x}p$ in $S$, denoted $Ant_{\infty}(\log_{x}p,S)$, to be $\{\xi\in\partial_{T}S|\ \overline{x\xi}\subset S\ and\ x\in\overline{p\xi}\}$. Then we have:

\begin{lem}
\label{3.26}
Let $S$ and $p$ be as in Lemma \ref{3.14}. Then for all $\beta>0$ there is an $r<\infty$ such that if $x\in S\setminus B(p,r)$, then 
\begin{equation}
\label{3.27}
diam(Ant_{\infty}(\log_{x}p,S))<\beta\,.
\end{equation}
The diameter here is with respect to the angular metric on $\partial_{T}X$.
\end{lem}

Lemma \ref{3.24} tells us $\Sigma_{y}S$ looks more and more like a suspension as $d(y,p)\to\infty$ ($y\in S$). If we also assume $Shape(Y)$ is finite, then for all $y\in S$, $\Sigma_{y}S$ is built from cells of finitely many isometry type, moreover, by Theorem \ref{3.9}, there is a positive constant $N$, such that $\Sigma_{y}S$ has at most $N$ cells for any $y\in S$. Thus the family $\{\Sigma_{y}S\}_{y\in S}$ has only finite possible combinatorics. As $\beta\to 0$, one may expect $\Sigma_{y}S$ is actually a suspension (this is called isolated suspension in \cite{bestvina2008quasiflats}).

Now we restrict ourselves to the case of finite dimensional $CAT(0)$ cube complexes of finite dimension. Then the spaces of directions are finite dimensional all-right spherical $CAT(1)$ complexes (see Definition \ref{cell structure} for the definition of cell structure on the spaces of directions). 

\begin{lem}
\label{3.28}
Suppose $\mathcal{F}$ is a family of all-right spherical $CAT(1)$ complexes with dimension at most $d$. Then for every $\alpha>0$ and $N>0$, there is a constant $\beta=\beta(d,N,\alpha)>0$ such that if $K'$ satisfies:
\begin{enumerate}
\item $K'\subset K$ is a subcomplex of some $K\in\mathcal{F}$.
\item The number of cells in $K'$ is bounded above by $N$.
\item $K'$ has the geodesic extension property in the sense of Lemma \ref{3.8}.
\item There exists $v\in K$ such that $diam(Ant(v,K'))<\beta$.
\end{enumerate} 
Then $K'$ is a metric suspension (in the induced metric from $K$) and $v$ lies at distance $<\alpha$ from a suspension point of $K'$.
\end{lem}

\begin{proof}
We prove by contradiction, suppose there exist $\alpha>0$, $N>0$ and a sequence $\{K'_{i}\}_{i=1}^{\infty}$ such that for each $i$, $K'_{i}$ satisfies (2) and (3), $K'_{i}$ is a subcomplex of some $K_{i}\in\mathcal{F}$, and there exists $v_{i}\in K_{i}$ such that  
\begin{equation}
\label{3.29}
diam(Ant(v_{i},K'_{i}))<\frac{1}{i}\,.
\end{equation} 
but no point in the $\alpha$-neighbourhood of $v_{i}$ is suspension point of $K'_{i}$. 

Let $w_{i}$ be the point in $K'_{i}$ which realizes the minimal distance to $v_{i}$ in the length metric of $K_{i}$ (note that the original metric on $K_{i}$ is the length metric truncated by $\pi$). If $l_{i}$ is the geodesic segment (in the length metric) joining $v_{i}$ and $w_{i}$, then by (\ref{3.29}) and the geodesic extension property of $K'_{i}$, there exists $C$ such that $length(l_{i})<C$ for all $i$. So for any $i$, there exists a subcomplex $L_{i}\subset K_{i}$ such that $l_{i}\subset L_{i}$ and the number of cells in $L_{i}$ are uniformly bounded by constant $N_{1}$ (by Lemma \ref{2.3}). 

Let $M_{i}$ be the full subcomplex spanned by $K'_{i}\cup L_{i}$, i.e. $M_{i}$ is the union of simplexes in $K_{i}$ whose vertex sets are in $K'_{i}\cup L_{i}$. Then $M_{i}$ is a $\pi$--convex (hence $CAT(1)$) subcomplex of $K_{i}$, and the number of cells in $M_{i}$ is uniformly bounded above by some constant $N_{2}$. Without loss of generality, we can replace $K_{i}$ by $M_{i}$. Since $M_{i}$ has only finite possible isometry types, after passing to a subsequence, we can assume there exist a finite $CAT(1)$ complex $M$ and a subcomplex $K'\subset M$ such that for every $i$, there is a simplicial isomorphism $\phi_{i}\co M_{i}\to M$ mapping $K'_{i}$ onto $K'$ ($\phi_{i}$ is also an isometry).

Since $M$ is compact, there is a subsequence of $\{\phi_{i}(v_{i})\}_{i=1}^{\infty}$ converging to a point $v\in M$. We claim $Ant(v,K')$ is exactly one point. First $Ant(v,K')\neq\emptyset$ by the geodesic extension property of $K'$. If there are two distinct points $u,u'\in Ant(v,K')$, then we can extend the geodesic segment $\overline{v_iu}$, $\overline{v_{i}u'}$ into $K'$, which yields a contradiction with $(\ref{3.29})$ for large $i$. 

Suppose $Ant(v,K')=\{v'\}$. Then $Ant(v',K')=v$. In fact, if this is not true, then we have some point $w\in Ant(v',K')$ such that $0<d(v,w)<\pi$. Then we can extent the geodesic segment $\overline{vw}$ into $K'$ to get a local geodesic $\overline{vw'}$ with $w'\in K'$ and $length(\overline{vw'})=\pi$. This is actually a geodesic since we are in a $CAT(1)$ space, thus $w'\in Ant(v,K')$ and $w'\neq v'$, contradictory to $Ant(v,K')=v'$. 

Now pick point $k\in K'$, $k\neq v$ and $k\neq v'$, then $d(k,v')<\pi$ and $d(k,v)<\pi$. We extend the geodesic segment $\overline{v'k}$ into $K'$ to get a geodesic of length $\pi$, then the other end must hit $v$ since $Ant(v',K')=v$. Thus $\overline{kv}\subset K'$ by the uniqueness of geodesic joining $k$ and $v$. Similarly we know $\overline{kv'}\subset K'$, thus there is a geodesic segment in $K'$ passing through $k$ and joining $v$ and $v'$. By $CAT(1)$ geometry, $K'$ (with the induced metric from $M$) splits as metric suspension and $v$, $v'$ are suspension points. However, by the assumption at the beginning of the proof, $\{\phi_{i}(v_{i})\}_{i=1}^{\infty}$ should have distance at least $\alpha$ from a suspension point for every $i$, so $v$ should also be $\alpha$--away from a suspension point, this contradiction finishes the proof.
\end{proof}

\begin{remark}
\label{3.30}
\
\begin{enumerate}
\item The above proof also shows the following result. Let $K$ be a piecewise spherical $CAT(1)$ complex, and let $K'\subset K$ be a subcomplex with geodesic extension property in the sense of Lemma \ref{3.8}. Pick $v\in K$. If $Ant(v,K')$ is exactly one point, then $v\in K'$ and $v$ is a suspension point of $K'$.
\item By the same proof, it is not hard to see Lemma \ref{3.28} is also true when $\mathcal{F}$ is a finite family of finite piecewise spherical $CAT(0)$ complexes (not necessarily all-right).

\end{enumerate}
\end{remark}
From lemma \ref{3.4}, Lemma \ref{3.24} and Lemma \ref{3.28}, we have the following analogue of \cite[Theorem 3.11]{bestvina2008quasiflats}:
\begin{thm}
\label{3.31}
Let $X$ be an $n$--dimensional $CAT(0)$ cube complex, and let $S=S_{[\sigma]}$ where $\sigma\in H^{p}_{n,n}(X)$. Then for every $p\in X$ and every $\epsilon>0$, there is an $r<\infty$ such that if $x\in S\setminus B(p,r)$, then $\Sigma_{x}S$ is a suspension and $\log_{x}p$ is $\epsilon$--close to a suspension point of $\Sigma_{x}S$.
\end{thm}

\begin{remark}
\label{3.32}
By the same proof, the conclusion of theorem \ref{3.31} is also true if $X$ is a proper $n$--dimensional $CAT(0)$ complex with a cocompact (cellular) isometry group.
\end{remark}

\section{The structure of top dimensional support set}
\label{topdim support set}
Throughout this section, $X$ is an $n$--dimensional $CAT(0)$ cube complex. Pick $[\sigma]\in H^{\textmd{p}}_{n,n}(X)$ and let $S=S_{[\sigma]}$. Also recall that $\Sigma_{x}X$ is an all-right spherical $CAT(1)$ complex for each $x\in X$ (see Definition \ref{cell structure}). 

Let $\Delta^{k}$ be the $k$--dimensional all-right spherical simplex, and let $\Delta^{k}_{mod}$ be the quotient of $\Delta^{k}$ by the action of its isometry group ($\Delta^{k}_{mod}$ is endowed with the quotient metric). Define the function $\chi\co\Delta^{k}\to (0,+\infty)$ by
\begin{equation}
\label{4.1}
\chi(v)=\inf\{d(v,v')\ |\ v'\in\Delta^{k}\ and\ Supp(v')\cap v=\emptyset\}\,.
\end{equation}
Recall that $Supp(v')$ denotes the unique closed face of $\Delta^{k}$ which contains $v'$ as an interior point. By symmetry of $\Delta^{k}$, $\chi$ descends to a function $\chi\co\Delta^{k}_{mod}\to(0,+\infty)$.

For any $k'\ge k$, we have a canonical isometric embedding $i\co\Delta^{k}_{mod}\hookrightarrow\Delta^{k'}_{mod}$ with $\chi=\chi\circ i$. Let $\Delta_{mod}=\varinjlim \Delta^{k}_{mod}$ be the corresponding direct limit of metric spaces. 

Let $Y$ be an all-right spherical $CAT(1)$ complex. Then there is a well-defined 1--Lipschitz map:
\begin{center}
$\theta\co Y\to\Delta_{mod}$
\end{center}
such that $\theta$ restricted to any $k$--face $\Delta^{k}\subset Y$ is the map $\Delta^{k}\to\Delta^{k}_{mod}\hookrightarrow\Delta_{mod}$. Moreover, for $ v\in Y$, 
\begin{enumerate}
\item $v\in Supp(v')$ if $d(v,v')<\chi(\theta(v))$.
\item $\chi\circ\theta$ is continuous on the interior of each face of $Y$.
\end{enumerate}

When $Y=\Sigma_{x}X$ for some $x\in X$ and $v\in \Sigma_{x}X$, we also call $\theta(v)$ the \textit{$\Delta_{mod}$ direction of $v$}.

\subsection{Producing orthants}
\label{orthants}
In this section, we study geodesic ray with constant $\Delta_{mod}$ direction, i.e. a unit speed geodesic ray $l\co[0,\infty)\to S$ with $\theta(l^{-}(t))=\theta(l^{+}(t))=\theta(l^{-}(t'))=\theta(l^{+}(t'))$ for any $t\neq t'$. Here are two examples.
\begin{enumerate}
\item If a geodesic ray $l$ stays inside an orthant subcomplex of $O\subset Y$ (or more generally a straight orthant), then it has constant $\Delta_{mod}$ direction. Moreover, the $\Delta_{mod}$ direction of $\partial_{T}l$ in $\partial_T O$ is equal to $\theta(l^{\pm}(t))$.
\item If $Y$ is a product of trees, then each geodesic ray in $l\in Y$ has constant $\Delta_{mod}$ direction. Again, the $\Delta_{mod}$ direction of $\partial_{T}l$ in $\partial_T Y$ (in this case $\partial_T Y$ is an all-right spherical complex) is equal to $\theta(l^{\pm}(t))$.
\end{enumerate}
 
Later geodesic rays with constant $\Delta_{mod}$ direction will play an important role in the construction of orthants (see Lemma \ref{4.9}). First we show such geodesics exist in the support set of a top dimensional proper cycle and there are plenty of them.
\begin{lem}
\label{4.2}
If $Y$ is a $k$--dimensional all-right spherical $CAT(1)$ complex and if $K\subset Y$ is the support set of some top dimensional homology class, then for any $v\in K$, there exists $v'\in K$ such that $d(v,v')=\pi$ and $\theta(v)=\theta(v')$.
\end{lem}

Recall that the metric on $Y$ is the length metric on $Y$ truncated by $\pi$.
\begin{proof}
The lemma is clear when $k=1$ by Lemma \ref{3.8}. We assume it is true for $i\le k-1$. Denote $k'=\dim(Supp(v))$. We endow $\Bbb S^{k'}$ with the structure of all-right spherical complex and pick $w\in \Bbb S^{k'}$ such that $\theta(v)=\theta(w)$. Suppose $w'=Ant(w,\Bbb S^{k'})$ and suppose $\gamma'\co[0,\pi]\to \Bbb S^{k'}$ is a unit speed geodesic joining $w$ and $w'$. It is clear that $\theta(w)=\theta(w')$. Our goal is to construct a unit speed local geodesic $\gamma\co[0,\pi]\to K$ such that $\gamma(0)=v$ and $\theta(\gamma(s))=\theta(\gamma'(s))$ for all $s\in [0,\pi]$ (see the following diagram), then $\gamma$ is actually a geodesic and we can take $v'=\gamma(\pi)$ to finish the proof.
\begin{center}
$\begin{CD}
[0,\pi]                        @>\gamma'>>       \Bbb S^{k'}\\
@V\gamma VV                              @VV\theta V\\
K                          @>\theta>>       \Delta_{mod}
\end{CD}$
\end{center}

There exists a sequence of faces $\{\Delta'_{j}\}_{j=1}^{N}$ in $\Bbb S^{k'}$ and $0=t_{0}<t_{1}<...<t_{N-1}<t_{N}=\pi$ such that each $\Delta'_{j}$ contains $\{\gamma'(t)\mid t_{j-1}<t<t_{j}\}$ as interior points. Let $\Delta_{1}=Supp(v)$. Since $\theta(v)=\theta(w)$, we can find $v_{1}\in \Delta_{1}$ such that there exists an isometry $\Phi\co\Delta_{1}\to\Delta'_{1}$ with $\Phi(v)=w$ and $\Phi(v_{1})=\gamma'(t_{1})$, in particular $\theta(v_{1})=\theta(\gamma'(t_{1}))$. Define $\gamma\co[0,t_{1}]\to K$ to be the geodesic segment $\overline{vv_{1}}$.

Recall that we have identified $\Sigma_{v_{1}}Y$ with $Lk(v_1,Y)$ (see Definition \ref{cell structure}). Let $\sigma_1=Supp(v_1)$ and let $k_1=\dim(\sigma)-1$. Then $\Sigma_{v_{1}}Y=Lk(v_1,\sigma_1)\ast Lk(\sigma_1,Y)=\Bbb S^{k_1}\ast Lk(\sigma_1,Y)$. Similarly, $\Sigma_{v_{1}}K=Lk(v_1,\sigma_1)\ast Lk(\sigma_1,K)=\Bbb S^{k_1}\ast Lk(\sigma_1,K)$. Let $K_1=Lk(\sigma_1,K)$ and $Y_1=Lk(\sigma_1,Y)$. Then they are all-right spherical complexes, $K_1$ is a subcomplex of $Y_1$, and $Y_1$ is $CAT(1)$. Moreover, since $\Sigma_{v_{1}}K$ is the support set of some top dimensional homology class in $\Sigma_{v_{1}}Y$ (Lemma \ref{link of support set}), so is $K_{1}$ in $Y_{1}$. As $\gamma^{-}(t_{1})\in \Sigma_{v_{1}}K=\Bbb S^{k_1}\ast K_1$, we write
\begin{equation}
\label{4.3}
\gamma^{-}(t_{1})=(\cos\alpha_{1})x_{1}+(\sin\alpha_{1})y_{1}
\end{equation}
for $x_{1}\in \Bbb S^{k_{1}}$ and $y_{1}\in K_{1}$. By induction assumption, we can find $y'_{1}\in Ant(y_{1},K_{1})$ such that $\theta(y'_{1})=\theta(y_{1})$. Let $x'_{1}=Ant(x_{1},\Bbb S^{k_{1}})$. Suppose $\Delta_{2}\subset K$ is the unique face ($v_{1}\in\Delta_{2}$) such that $Supp((\cos\alpha_{1})x'_{1}+(\sin\alpha_{1})y'_{1})=\Sigma_{v_{1}}\Delta_{2}$. Let $\overline{v_{1}v_{2}}\subset\Delta_{2}$ be the geodesic segment such that it starts at $v_{1}$, and goes along the direction $(\cos\alpha_{1})x'_{1}+(\sin\alpha_{1})y'_{1}$ until it hits some boundary point $v_{2}$ of $\Delta_{2}$. Note that $\overline{vv_{1}}$ and $\overline{v_{1}v_{2}}$ fit together to form a local geodesic in $K$.

On the other hand, at $\gamma'(t_{1})\in \Bbb S^{k'}$, we have $\Sigma_{\gamma'(t_{1})}\Bbb S^{k'}=Lk(\gamma'(t_{1}),\sigma'_1)\ast Lk(\sigma'_1,\Bbb S^{k'})=\Bbb S^{k_{1}}\ast Lk(\sigma'_1,\Bbb S^{k'})$ where $\sigma'_1=Supp(\gamma'(t_{1}))$ and $k_{1}$ is the same as the previous paragraph. Write $\gamma'^{-}(t_{1})=(\cos\alpha_{1})u_{1}+(\sin\alpha_{1})v_{1}$ (here we have the same $\alpha_{1}$ as (\ref{4.3}) since $\Phi$ is an isometry) for $u_{1}\in \Bbb S^{k_{1}}$ and $v_{1}\in Lk(\sigma'_1,\Bbb S^{k'})$, then $\gamma'^{+}(t_{1})=(\cos\alpha_{1})u'_{1}+(\sin\alpha_{1})v'_{1}$ for $u'_{1}=Ant(u_{1},\Bbb S^{k_{1}})$ and $v'_{1}=Ant(v_{1},Lk(\sigma'_1,\Bbb S^{k'}))$. Note that $\theta(v'_{1})=\theta(v_{1})=\theta(y_{1})=\theta(y'_{1})$, so we can extent the isometry $\Phi$ to get $\Phi'\co\Delta_{1}\cup\Delta_{2}\to\Delta'_{1}\cup\Delta'_{2}$ such that $\Phi'$ is an isometry with respect to the length metric on both sides and $\Phi'(\overline{v_{1}v_{2}})=\gamma([t_{1},t_{2}])$. Thus $d(v_{1},v_{2})=t_{2}-t_{1}$ and we can define $\gamma\co[t_{1},t_{2}]\to K$ to be the geodesic segment $\overline{v_{1}v_{2}}$. It is clear that $\theta(\gamma(s))=\theta(\gamma'(s))$ for all $s\in [0,t_{2}]$. We can repeat this process to define the required local geodesic $\gamma\co[0,\pi]\to K$.
\end{proof}

\begin{cor}
\label{4.4}
For any $x\in S$ and $v\in \Sigma_{x}S$, there exists a geodesic ray $\overline{x\xi}\subset S$ which has constant $\Delta_{mod}$ direction and $\log_{x}\xi=v$. 
\end{cor}

\begin{proof}
Since $x$ has a cone neighbourhood in $S$, we can find a short geodesic segment $xx'$ in the cone neighbourhood such that $\log_{x}x'=v$. There is a unique closed cube $C_{1}\subset S$ such that $\overline{xx'}\subset C_{1}$ and $v$ is an interior point of $\Sigma_{x}C_{1}$. We extend $\overline{xx'}$ in $C_{1}$ until it hits the boundary of $C_{1}$ at $x_{1}$. By Lemma \ref{4.2}, there exists $v_{1}\in Ant(\log_{x_{1}}x,\Sigma_{x_{1}}S)$ with $\theta(v_{1})=\theta(\log_{x_{1}}x)=\theta(v)$. Now we choose cube $C_{2}\subset S$ and segment $\overline{x_{1}x_{2}}\subset C_{2}$ with $\log_{x_{1}}x_{2}=v_{1}$ as before. Note that $\overline{xx_{1}}$ and $\overline{x_{1}x_{2}}$ together form a local geodesic segment (hence a geodesic segment). We repeat the previous process to extend the geodesic. Since $S$ is a closed set, the extension can not terminate, which will give us the geodesic ray $\overline{x\xi}$ as required.
\end{proof}

\begin{cor}
\label{4.5}
The set of points in $\partial_{T}S$ which can be represented by a geodesic ray in $S$ with constant $\Delta_{mod}$ direction is dense.
\end{cor}

\begin{proof}
Fix a base point $p$, pick some $\xi\in\partial_{T}S$. For any $\epsilon>0$, by (\ref{3.23}), we can find $r_{1}$ such that 
\begin{equation}
\label{4.6}
d_{H}(\partial_{p,r}S,\partial_{T}S)<\frac{\epsilon}{2}
\end{equation} 
for all $r>r_{1}$. By Lemma \ref{3.26}, we can find $r_{2}$ such that for $r>r_{2}$, 
\begin{equation}
\label{4.7}
diam(Ant_{\infty}(\log_{x}p,S))<\frac{\epsilon}{2}
\end{equation}
for any $x\in S\setminus B(p,r)$.

If $r_{0}=\max\{r_{1},r_{2}\}$, then we can find $\overline{p\xi_{1}}\subset B(p,r_{0})\cup S$ such that $\angle_{T}(\xi_{1},\xi)\le\frac{\epsilon}{2}$ by (\ref{4.6}). Pick $x\in\overline{p\xi_{1}}$ such that $d(p,x)>r_{0}$. By corollary \ref{4.4}, we can find geodesic ray $\overline{x\xi_{2}}\subset S$ of constant $\Delta_{mod}$ direction such that $\overline{x\xi_{2}}$ fits together with $\overline{px}$ to form a geodesic ray $\overline{p\xi_{2}}$, thus $\angle_{T}(\xi_{1},\xi_{2})<\frac{\epsilon}{2}$ by (\ref{4.7}). Then $\angle_{T}(\xi,\xi_{2})<\epsilon$, which finishes the proof since $\epsilon$ and $\xi$ are arbitrary.
\end{proof}

Let $l$ be a geodesic ray of constant $\Delta_{mod}$ direction. Then we define $\theta(l)$ to be $\theta(l^{\pm}(t))$. The definition does not depend on the choice of sign $\pm$ and $t$.

\begin{lem}
\label{4.8}
If $l\subset S$ is a unit speed geodesic ray of constant $\Delta_{mod}$ direction, then there exists $t_{0}<\infty$ $(t_{0}$ depends on the position of $l(0)$ and $\theta(l))$ such that for any $t>t_{0}$, $\Sigma_{l(t)}S=\Sigma_{l(t)}l\ast Y_{t}$ for some $Y_{t}\subset \Sigma_{l(t)}S$.
\end{lem}

\begin{proof}
We apply Theorem \ref{3.31} where $p=l(0)$ and $\epsilon=\chi(\theta(l))$ (see (\ref{4.1}) for the definition of $\chi$) to get $t_{0}<\infty$ such that $\Sigma_{l(t)}S$ is a metric suspension and the suspension point is $\chi(\theta(l))$--close to $l^{+}(t)$ (or $l^{-}(t)$) for all $t>t_{0}$. By Lemma \ref{2.9}, $l^{+}(t)$ and $l^{-}(t)$ are suspension points, thus $\Sigma_{l(t)}S=\Sigma_{l(t)}l\ast Y_{t}$.
\end{proof}

Based on Lemma \ref{4.8}, we define a parallel transport of $\Sigma_{l(t)}S$ along $l$ as follows. Let $t_0$ be as in Lemma \ref{4.8}. For any $t>t_{0}$, $l(t)$ has a product neighbourhood in $S$ of form $X_{t}\times (t-\epsilon,t+\epsilon)$, here $X_{t}$ is some subset of $X$ with the induced metric. So for any $|t'-t|<\epsilon$, we can identify $\Sigma_{l(t)}S$ and $\Sigma_{l(t')}S$. Moreover, for any $t_{1}>t_{0}$ and $t_{2}>t_{0}$, we can cover the geodesic segment $\overline{l(t_{1})l(t_{2})}$ by finitely many product neighbourhoods, which will induce an identification of $\Sigma_{l(t_{1})}S$ and $\Sigma_{l(t_{2})}S$. This identification does not depend on the covering we choose.

To see this identification more concretely, take $t>t_{0}$, a product neighbourhood $X_{t}\times (t-\epsilon,t+\epsilon)$ of $l(t)$ in $S$, and $v\in\Sigma_{l(t)}S$, we construct a short geodesic $\overline{l(t)x_{t}}\subset S$ in the cone neighbourhood of $l(t)$ going along the direction $v$. If $|t'-t|<\epsilon$, we can find an isometrically embedded parallelogram in the product neighbourhood such that $\overline{l(t)x_{t}}$ and $\overline{l(t')x_{t'}}$ are opposite sides of the parallelogram and $\overline{l(t)l(t')}$ is one of the remaining sides (we might have to shorten $\overline{l(t)x_{t}}$ a little).

In general, for any $t'>t_{0}$, we can cover the geodesic segment $\overline{l(t)l(t')}$ by finitely many product neighbourhoods as before and construct a local isometric embedding $\phi$ from a parallelogram to $X$ such that two opposite sides of the parallelogram is mapped to some geodesic segments $\overline{l(t)x_{t}}$ and $\overline{l(t')x_{t'}}$ and one of the remaining sides is mapped to $\overline{l(t)l(t')}$. Since $X$ is $CAT(0)$, $\phi$ is actually an isometric embedding. So we have a well-defined \textit{parallel transport} of $\Sigma_{l(t)}S$ along $l(t)$ for $t>t_{0}$.

In the construction of the above parallelogram, the length of $\overline{l(t)x_{t}}$ (or $\overline{l(t')x_{t'}}$) may go to $0$ as $|t'-t|\to\infty$. However, in the special case where there exists $s_0>0$ such that $\Sigma_{l(t)}l$ is contained in the $0$--skeleton of $\Sigma_{l(t)}S$ for all $t>s_{0}$ (or equivalently $l([s_{0},\infty))$ is parallel to some geodesic ray in the $1$--skeleton of $X$), $l([t_{0},\infty))$ has a product neighbourhood of form $X'\times [t_{0},\infty)$ in $S$, here $X'$ is a subcomplex of $X$ with the induced metric. Therefore for any $t>t_{0}$ and a segment $\overline{l(t)x_{t}}\subset S$ short enough, we can parallel transport $\overline{l(t)x_{t}}$ along $l$ to infinity, i.e. there is an isometrically embedded \textquotedblleft infinite parallelogram\textquotedblright\ with one side $\overline{l(t)x_{t}}$ and one side $l([t,\infty))$.

\begin{lem}
\label{4.9}
If $l\co[0,\infty)\to S$ is a unit speed geodesic ray of constant $\Delta_{mod}$ direction, then there exists an orthant subcomplex $O\subset X$ satisfying
\begin{enumerate}
\item $\partial_{T}l\in \partial_{T}O$.
\item If $\dim(O)=k$ and if $\{l_{i}\}_{i=1}^{k}$ are the geodesic rays emanating from the tip of the orthant such that $O$ is the convex hull of $\{l_{i}\}_{i=1}^{k}$, then $\partial_{T}l_{i}\in\partial_{T}S$ for all $i$.
\end{enumerate} 
\end{lem}

\begin{proof}
By the previous lemma, we can choose some $r>0$ such that for $t>r$, $l^{\pm}(t)$ are suspension points in $\Sigma_{l(t)}S$. Pick some $t>r$ and let $k$ be the dimension of $Supp(l^{+}(t))$. Let $\{v_{i}^{t}\}_{i=1}^{k}$ be vertices of $Supp(l^{+}(t))$. Suppose $\alpha_{i}=d_{\Sigma_{l(t)}S}(v_{i}^{t},l^{+}(t))$ (the values of $k$ and $\alpha_{i}$ are the same for all $t>r$ by the splitting in Lemma \ref{4.8} and Lemma \ref{2.9}). Moreover, we would like the labels $v_{i}^{t}$ to be consistent under parallel transportation, i.e. for $t'\neq t$ ($t'>r$), $v_{i}^{t'}$ is the parallel transport of $v_{i}^{t}$ along $l$. By Theorem \ref{3.31}, we can choose $r'\geq r$ such that if $x\in S\setminus B(l(0),r')$, then $\log_{x}l(0)$ is $\epsilon$--close to some suspension point in $\Sigma_{x}S$ for $\epsilon=\min_{1\leq i\leq k}\{\frac{1}{2}(\frac{\pi}{2}-\alpha_{i})\}$.

Now we pick some $t>r'$, and construct a short geodesic segment $\overline{l(t)x_{i}}\subset S$ going along the direction of $v_{i}^{t}$ in the cone neighbourhood of $l(t)$. We choose an arbitrary extension of $\overline{l(t)x_{i}}$ into $S$ and call the geodesic ray $l_{i}^{t}$ for $1\leq i\leq k$. We claim for any $y\in l_{i}^{t}$ ($y\neq l(t)$),
\begin{equation}
\label{4.10}
\Sigma_{y}S=\Sigma_{y}l_{i}^{t}\ast Y
\end{equation} 
for some $Y\subset\Sigma_{y}S$, hence the extension is unique and $l_{i}^{t}$ is a straight geodesic.

If the claim is not true, pick the first point $y_{i}\in l_{i}^{t}$ such that (\ref{4.10}) is not satisfied. Since $\angle_{l(t)}(y_{i},l(0))\geq \pi-\alpha_{i}>\pi/2$, hence $d(y_{i},l(0))>d(l(0),l(t))>r'$. By our choice of $r'$, there is a suspension point in $\Sigma_{y_{i}}S$ which has distance less than $\frac{1}{2}(\frac{\pi}{2}-\alpha_{i})$ from $\log_{y_{i}}l(0)$. Since $\angle_{y_i}(l(0),l(t))<\alpha_i$, $\log_{y_{i}}l(t)$ has distance less than $\alpha_i+\frac{1}{2}(\frac{\pi}{2}-\alpha_{i})<\frac{\pi}{2}$ from a suspension point. Since all points in $l_{i}^{t}$ between $l(t)$ and $y_{i}$ satisfy (\ref{4.10}), $\log_{y_{i}}l(t)$ is a vertex in the all-right spherical complex $\Sigma_{y_{i}}S$. Thus $\log_{y_{i}}l(t)$ is also a suspension point and (\ref{4.10}) must hold at $y=y_i$, which is a contradiction. 

We claim next $l_{i}^{t}$ is parallel to $l_{i}^{t'}$ for any $t>r'$ and $t'>r'$. In fact, let us fix $t$, by the discussion before Lemma \ref{4.9} and the uniqueness of $l_{i}^{t'}$, we know the claim is true for $|t'-t|<\epsilon$ ($\epsilon$ depends on $t$). For the general case, we can apply a covering argument as before. 

Fix $t_{0}>r'$. By Lemma \ref{2.4}, for all $i$,
\begin{equation}
\label{4.11}
\angle_{T}(l_{i}^{t_{0}},l)=\lim_{t\to +\infty}\angle_{l(t)} (l_{i}^{t},l)=\alpha_{i}<\frac{\pi}{2}\,.
\end{equation} 
Thus $\angle_{T}(l_{i}^{t_{0}},l)=\angle_{l(t_{0})}(l_{i}^{t_{0}},l)=\alpha_{i}$ for all $i$. It follows that $l$ and $l_{i}^{t_{0}}$ bound a flat sector by Lemma \ref{2.5}.

We fix a pair $i,j$ ($i\neq j$) and parametrize $l_{i}^{t_{0}}$ by arc length. We can assume without loss of generality that $l(t_{0})$ is in the $0$--skeleton. Let $\{h_{m}\}_{m=1}^{\infty}$ be the collection of hyperplanes such that $h_{m}\cap l_{i}^{t_{0}}=l_{i}^{t_{0}}(m+\frac{1}{2})$. (\ref{4.11}) and Lemma \ref{2.4} imply the $CAT(0)$ projection of $l$ onto $l_{i}^{t_{0}}$ is surjective, thus there exists a sequence $\{t_{m}\}_{m=1}^{\infty}$ such that $l(t_{m})\in h_{m}$. Note that $l_{j}^{t_{m}}$ (recall that $j\neq i$) starts at $l(t_{m})$, $\angle_{l(t_{m})}(l_{i}^{t_{m}},l_{j}^{t_{m}})=\frac{\pi}{2}$ and $l_{i}^{t_{m}}$ is orthogonal to $h_{m}$, so $l_{j}^{t_{m}}\subset h_{m}$. 

By convexity of $h_{m}$, we can find a geodesic ray $c_{m}$ which starts at $l_{i}^{t_{0}}(m+\frac{1}{2})$, stays inside $h_{m}$ and is asymptotic to $l_{j}^{t_{m}}$ for every $m$, thus by Lemma \ref{2.4},
\begin{equation}
\label{4.12}
\angle_{T}(l_{i}^{t_{0}},l_{j}^{t_{0}})=\lim_{m\to+\infty}\angle_{l_{i}^{t_{0}}(m+\frac{1}{2})}(l_{i}^{t_{0}},c_{m})=\frac{\pi}{2}\,.
\end{equation}

By (\ref{4.12}), $\angle_{T}(l_{i}^{t_{0}},l_{j}^{t_{0}})=\angle_{l(t_{0})}(l_{i}^{t_{0}},l_{j}^{t_{0}})=\frac{\pi}{2}$ for $i\neq j$. By Lemma \ref{2.18}, we know the geodesic rays  $\{l_{i}^{t_{0}}\}_{i=1}^{k}$ span a straight orthant $O$. Moreover, Lemma \ref{2.8} together with (\ref{4.11}) implies $\partial_{T}l\in \partial_{T}O$. By Remark \ref{2.16}, we can replace $O$ by an orthant subcomplex which is Hausdorff close to $O$.
\end{proof}

\subsection{Cycle at infinity}
\label{cycle at infinity}
By Lemma \ref{4.9} and Corollary \ref{4.5}, there exists a dense subset $G$ of $\partial_{T}S$ such that for any $v\in G$, there exists an orthant subcomplex $O_{v}\in X$ such that $v\in\bigtriangleup_{v}=\partial_{T}O_{v}$. Denote the vertices of $\bigtriangleup_{v}$ by $Fr(v)$, then $Fr(v)\subset \partial_{T}S$ by Lemma \ref{4.9}.

It is clear that $G\subset\cup_{v\in G}\bigtriangleup_{v}$. We claim $\cup_{v\in G}\bigtriangleup_{v}$ is a finite union of all-right spherical simplexes. In fact, it suffices to show $\cup_{v\in G}Fr(v)$ is a finite set, which follows from Lemma \ref{2.17}, Lemma \ref{3.14} and Lemma \ref{4.9} (note that each point in $\cup_{v\in G}Fr(v)$ is represented by a straight geodesic contained in $S$). 

Moreover, $\cup_{v\in G}\Delta_{v}$ has the structure of a finite simplicial complex. Take two simplexes $\Delta_{v_{1}}$ and $\Delta_{v_{2}}$, we know $\Delta_{v_{i}}=\partial_{T}O_{v_{i}}$ for orthant subcomplex $O_{v_{i}}$, and Remark \ref{2.13} implies $\Delta_{v_{1}}\cap \Delta_{v_{2}}$ is a face of $\Delta_{1}$ (or $\Delta_{2}$).

We endow $K=\cup_{v\in G}\bigtriangleup_{v}\subset \partial_{T}Y$ with the angular metric and denote the Euclidean cone over $K$ by $CK$, which is a subset of $C_{T}X$.

\begin{lem}\
\label{4.13}
\begin{enumerate}
\item $K$ is a topologically embedded finite simplicial complex in $\partial_{T}X$.
\item $CK$ is linearly contractible.
\end{enumerate}
\end{lem}

Recall that linearly contractible means there exists a constant $C$ such that for any $d>0$, every cycle of diameter $\le d$ can be filled in by a chain of diameter $\le Cd$.
\begin{proof}
Let $\angle_{T}$ be the angular metric on $K$ and $d_{l}$ be the length metric on $K$ as an all-right spherical complex. Our goal is to show $\textmd{Id}\co(K,\angle_{T})\to(K,d_{l})$ is a bi-Lipschitz homeomorphism. Let $\{\Delta_{i}\}$ be the collection of faces of $K$ (each $\Delta_{i}$ is an all-right spherical simplex). Suppose $\{O_{i}\}_{i=1}^{N}$ are orthant subcomplexes of $X$ such that $\partial_{T}O_{i}=\Delta_{i}$. If points $x$ and $y$ are in the same $\Delta_{i}$ for some $i$, then 
\begin{equation}
\label{4.14}
d_{l}(x,y)=\angle_{T}(x,y)\,.
\end{equation}
If $x$ and $y$ are not in the same simplex, then we put $\Delta_{i}=Supp(x)$, $\Delta_{j}=Supp(y)$ and $\Delta_{k}=\Delta_{i}\cap\Delta_{j}$. Assume without loss of generality that $d_{l}(x,\Delta_{k})\ge\frac{1}{2}d_{l}(x,y)$. Let $(Y_{1},Y_{2})=\inc (O_{i},O_{j})$. Then $\partial_{T}Y_{1}=\partial_{T}Y_{2}=\Delta_{k}$, moreover, it follows from (\ref{2.11}) and Lemma \ref{2.4} that: 
\begin{equation}
\label{4.15}
\angle_{T}(x,y)\ge 2\arcsin(\frac{A}{2}\sin(d_{l}(x,\Delta_{k})))\ge 2\arcsin(\frac{A}{2}\sin(\frac{1}{2}d_{l}(x,y)))\,,
\end{equation}
here $A$ can be chosen to be independent of $i$ and $j$ since $\{O_{i}\}_{i=1}^{N}$ is a finite collection. (\ref{4.14}) and (\ref{4.15}) imply $\textmd{Id}\co(K,\angle_{T})\to(K,d_{l})$ is a bi-Lipschitz homeomorphism, thus (1) is true.

To see (2), it suffices to prove $(K,\angle_{T})$ is linearly locally contractible, i.e. there exists $C<\infty$ and $R>0$ such that for any $d<R$, every cycle of diameter $\le d$ can be filled in by a chain of diameter $\le Cd$. By above discussion, we only need to prove $(K,d_{l})$ is locally linearly contractible. 

Since $(K,d_{l})$ is compact and can be covered by finitely many cone neighbourhood (see Theorem \ref{2.1}), it suffices to show each cone neighbourhood is linearly contractible, but any cone neighbourhood is isometric to a metric ball in the spherical cone of some lower dimensional finite piecewise spherical complex, thus we can finish the proof by induction on dimension.
\end{proof}

Since $G$ is a dense subset of $\partial_{T}S$ and $K$ is compact, then $\partial_{T}S\subset K$ and $C_{T}S\subset CK\subset C_{T}X$. We denote the base point of $C_{T}X$ by $o$. 

\begin{lem}
\label{4.16}
$K$ has the structure of a $(n-1)$--simplicial cycle.
\end{lem}

\begin{proof}
In the following proof, we will use $d$ to denote the metric on $X$, and use $\bar{d}$ to denote the metric on $C_TX$.

Pick a base point $p\in X$, by the proof of Lemma \ref{3.15}, we know for any $\epsilon>0$, there exists a finite collection of constant speed geodesic rays $\{l_{i}\}_{i=1}^{N}$ and $R_{\epsilon}<\infty$ such that $l_{i}(t)\in S$ and $\{l_{i}(t)\}_{i=1}^{N}$ is a $\epsilon t$--net in $B(p,t)\cap S$ for $t\ge R_{\epsilon}$. Denote $\xi_{i}=\partial_{T}l_{i}$ and define $f_{\epsilon}\co S\to C_{T}S\subset CK$ by sending $l_{i}(t)$ to the point in $\overline{o\xi_{i}}\subset C_{T}S$ which has distance $d(l_{i}(t),p)$ from $o$ ($t\ge R_{\epsilon}$). For $x\notin\cup_{i=1}^{N}l_{i}[R_{\epsilon},\infty)$, we pick a point $y\in\cup_{i=1}^{N}l_{i}[R_{\epsilon},\infty)$ which is nearest to $x$ and define $f_{\epsilon}(x)=f_{\epsilon}(y)$.

It is clear that
\begin{equation}
\label{4.17}
|d(x,p)-\bar{d}(f_{\epsilon}(x),o)|\le \epsilon\max\{d(x,p),R_{\epsilon}\}
\end{equation}
for any $x\in S$ and 
\begin{equation}
\label{4.18}
|d(x,y)-\bar{d}(f_{\epsilon}(x),f_{\epsilon}(y))|\le\epsilon\max\{d(p,x),d(p,y),R_{\epsilon}\}.
\end{equation}
for any $x\in S$ and $y\in S$. Moreover,
\begin{equation}
\label{4.19}
\bar{d}_{H}(f_{\epsilon}(B(p,r)\cap S),B(o,r)\cap C_{T}S)\le \epsilon\max\{r,R_{\epsilon}\}\,.
\end{equation}
We might need to pick a larger $R_{\epsilon}$ for (\ref{4.19}).

We want to approximate $f_{\epsilon}$ by a continuous map. Let us cover $S$ by collection of open sets $\{B(x,r_{x})\cap S\}_{x\in S}$, here $r_{x}=\epsilon\max\{d(x,p),R_{\epsilon}\}$. Since $S$ has topological dimension $\le n$, this covering has a refinement $\{U_{i}\}_{i=1}^{\infty}$ of order $\le n$ (see Chapter V of \cite{MR0006493}). Note that $diam(U_{i})\le 2\epsilon\max\{d(p,U_{i}),R_{\epsilon}\}$. Denote the nerve of $\{U_{i}\}_{i=1}^{\infty}$ by $N$, which is a simplicial complex of dimension $\le n$. 

Now we define a map $b'\co N\to CK$ as follows: for any vertex $v_{i}\in L$, pick $x_{i}\in U_{i}$ where $U_{i}$ is the open set associated with vertex $v_{i}$, then set $b'(v_{i})=f_{\epsilon}(x_{i})$. Then we use the linear contractibility of $CK$ to extend the map skeleton by skeleton to get $b'$. By choosing a partition of unity subordinate to the covering $\{U_{i}\}_{i=1}^{\infty}$, we obtain a barycentric map $b$ from $S$ to the nerve $N$ (see Chapter V of \cite{MR0006493}), then the continuous map $b'\circ b\co S\to CK$ also satisfies (\ref{4.17})--(\ref{4.19}) with $\epsilon$ replaced by $L'\epsilon$, $L'$ is some constant which only depends on the linear contractibility constant of $CK$. So we can assume without loss of generality that $f_{\epsilon}\co S\to CK$ is continuous and (\ref{4.17})--(\ref{4.19}) still holds for $f_{\epsilon}$.

Recall that $S$ is the support set of some top dimensional proper homology class $[\sigma]$. We can also view $[\sigma]$ as the fundamental class of $S$ and assume $\sigma$ is the proper singular cycle representing this class. If $\alpha=f_{\epsilon}(\sigma)$, then $[\alpha]\in H^{\textmd{p}}_{n}(CK)$ since $f_{\epsilon}$ is a proper map by (\ref{4.17}). Our next goal is to show
\begin{equation}
\label{4.20}
S_{[\alpha]}=CK
\end{equation}
for $\epsilon$ small enough. Since $K$ is a simplicial complex, (\ref{4.20}) would imply $K$ also has a fundamental class whose support set is exactly $K$ itself, hence Lemma \ref{4.16} follows.

Recall that we have a 1--Lipschitz logarithmic map $\log_{p}\co C_{T}X\to X$ sending base point $o$ to $p$. By (\ref{4.17}) and (\ref{4.18}), there exists a constant $L<\infty$ such that:
\begin{equation}
\label{4.21}
\bar{d}(z,o)=d(\log_{p}(z),p)
\end{equation}
for all $z\in\textmd{Im}\ f_{\epsilon}$ and 
\begin{equation}
\label{4.22}
|\bar{d}(z,w)-d(\log_{p}(z),\log_{p}(w))|\le L\epsilon\max\{\bar{d}(o,z),\bar{d}(o,w),R_{\epsilon}\}
\end{equation}
for all $z,w\in\textmd{Im}\ f_{\epsilon}$. Moreover,
\begin{equation}
\label{4.23}
d(\log_{p}\circ f_{\epsilon}(x),x)\le L\epsilon\max\{d(x,p),R_{\epsilon}\}
\end{equation}
for all $x\in S$. 

By (\ref{4.21}), $\log_{p}$ is proper. Let $\beta=\log_{p}(\alpha)=\log_{p}\circ f_{\epsilon}(\sigma)$. By (\ref{4.23}), the geodesic homotopy between $\log_{p}\circ f_{\epsilon}\co S\to X$ and the inclusion map $i\co S\to X$ is proper, thus $[\beta]=[\sigma]$ and $S_{[\beta]}=S_{[\sigma]}=S$. By Lemma \ref{3.2},
\begin{equation}
\label{4.24}
\log_{p}(S_{[\alpha]})\supset S_{[\beta]}=S\,.
\end{equation}
(\ref{4.24}), (\ref{4.22}) and (\ref{4.23}) imply there exists $L<\infty$ such that
\begin{equation}
\label{4.25}
\bar{d}_{H}(B(o,r)\cap S_{[\alpha]},B(o,r)\cap\textmd{Im}\ f_{\epsilon})\le L\epsilon\max\{r,R_{\epsilon}\}\,.
\end{equation}
This together with (\ref{4.19}) imply
\begin{equation}
\label{4.26}
\bar{d}_{H}(B(o,r)\cap S_{[\alpha]},B(o,r)\cap C_{T}S)\le L\epsilon\max\{r,R_{\epsilon}\}\,.
\end{equation}

Since $K$ is a simplicial complex, $S_{[\alpha]}=CK'$ where $K'$ is some subcomplex of $K$. Recall that by the construction of $K$, the only subcomplex of $K$ that contains $\partial_{T}S$ is $K$ itself. (\ref{4.26}) implies the Hausdorff distance between $\partial_{T}S$ and $K'$ is bounded above by $L\epsilon$, thus for $\epsilon$ small enough, $K'=K$ and (\ref{4.20}) holds. We also know $\partial_{T}S$ is dense in $K$ from this. 
\end{proof}

We actually defined a boundary map 
\begin{equation}
\label{4.27}
\partial\co H^{\textmd{p}}_{n,n}(X)\to H_{n-1}(\partial_{T}X)
\end{equation}
in the proof of above lemma, namely we send $[\sigma]\in H^{\textmd{p}}_{n,n}(X)$ to $f_{\epsilon\ast}[\sigma]\in H^{\textmd{p}}_{n}(C_{T}X)$ (for $\epsilon$ small enough), which would pass to an element in $H_{n-1}(\partial_{T}X)$ via $H^{\textmd{p}}_{n}(C_{T}X)\to H_{n}(C_{T}X,C_{T}X\setminus\{o\})\cong H_{n-1}(\partial_{T}X)$.

In the construction of $f_{\epsilon}$, we have to choose a base point, the geodesic rays $\{l_{i}(t)\}_{i=1}^{N}$, the covering $\{U_{i}\}_{i=1}^{\infty}$ and the maps $b$, $b'$, but different choices give maps in the same proper homotopy class if the corresponding $\epsilon$ is small enough. Also the geodesic homotopy from $f_{\epsilon_{1}}$ to $f_{\epsilon_{2}}$ is proper if $\epsilon_{1}$ and $\epsilon_{2}$ are small enough, so the above boundary map is well-defined.

Next we construct a map in the opposite direction as follows. Let $\eta'$ be a Lipschitz ($n-1$)-cycle in $\partial_{T}X$. Let $\alpha'$ be the cone over $\eta'$ (note that one can cone off elements in $C_{n-1}(\partial_{T}X)$ to obtain elements in $C^{\textmd{p}}_{n}(C_{T}X)$, which induces a homomorphism $H_{n-1}(\partial_{T}X)\to H^{\textmd{p}}_{n}(C_{T}X)$). Actually $[\alpha']\in H^{\textmd{p}}_{n,n}(C_{T}X)$ since the cone over a Lipschitz cycle would satisfy the required growth condition. If $\sigma'=\log(\alpha')$, then $[\sigma']\in H^{\textmd{p}}_{n,n}(X)$ since $\log$ is 1--Lipschitz. Now we define the \textquotedblleft conning off\textquotedblright\ map
\begin{equation}
\label{4.28}
c\co H_{n-1}(\partial_{T}X)\to H^{\textmd{p}}_{n,n}(X)
\end{equation}
by sending $[\eta']$ to $[\sigma']$. The base point in the definition of $\log$ does not matter because different base points give maps which are of bounded distance from each other. It is easy to see $c$ is a group homomorphism.

For $\epsilon>0$, pick a finite $\epsilon$--net of $\textmd{Im}\ \eta'$ and denote it by $\{\xi_{i}\}_{i=1}^{N}$. Suppose $p=\log(o)$ and suppose $\{l_{i}\}_{i=1}^{N}$ are the unit speed geodesic rays emanating from $p$ with $\partial_{T}l_{i}=\xi_{i}$. Pick $R_{\epsilon}>0$ such that 
\begin{equation}
\label{4.29}
|\frac{d(l_{i}(t),l_{j}(t))}{t}-\lim_{t\to +\infty}\frac{d(l_{i}(t),l_{j}(t))}{t}|<\epsilon
\end{equation}
for $t>R_{\epsilon}$. Let $I_{\sigma'}$ be the smallest subcomplex of $X$ which contains $\textmd{Im}\ \sigma'$. By using the rays $\{l_{i}\}_{i=1}^{N}$ as in the proof of Lemma \ref{4.16}, we can construct a continuous proper map $g_{\epsilon}\co I_{\sigma'}\to C_{T}X$ skeleton by skeleton such that
\begin{equation}
\label{4.30}
d(g_{\epsilon}\circ\log(x),x)\le L\epsilon\max\{d(x,o),R_{\epsilon}\}
\end{equation}
for $x\in\textmd{Im}\ \alpha'$, which implies $g_{\epsilon\ast}[\sigma']=[\alpha']$ for $\epsilon$ small. 

Let $[\sigma'']$ be the fundamental class of $S_{[\sigma']}$ and let $f_{\epsilon\ast}\co S_{[\sigma']}\to C_{T}X$ be the map in Lemma \ref{4.16}. We claim that $g_{\epsilon\ast}[\sigma']=f_{\epsilon\ast}[\sigma'']$ for $\epsilon$ small, which would imply
\begin{equation}
\label{4.31}
\partial\circ c=\textmd{Id}\,.
\end{equation}
To see the claim, note that $[\sigma']=[\sigma'']$ in $H^{\textmd{p}}_{n}(I_{\sigma'})$. For $\epsilon$ small, there is a proper geodesic homotopy between $g_{\epsilon}|_{S_{[\sigma']}}$ and $f_{\epsilon}$ by $(\ref{4.23})$ and $(\ref{4.30})$, thus $g_{\epsilon\ast}[\sigma'']=f_{\epsilon\ast}[\sigma'']$, moreover $g_{\epsilon\ast}[\sigma'']=g_{\epsilon\ast}[\sigma']$, so $f_{\epsilon\ast}[\sigma'']=g_{\epsilon\ast}[\sigma']=[\alpha']$.

From $(\ref{4.23})$ and the discussion after we know 
\begin{equation}
\label{4.32}
c\circ\partial=\textmd{Id}\,.
\end{equation}
Thus $\partial$ is also a group homomorphism and we have the following

\begin{cor}
\label{4.33}
If $X$ is an $n$--dimensional $CAT(0)$ cube complex, then:
\begin{enumerate}
\item $\partial\co H^{\textmd{p}}_{n,n}(X)\to H_{n-1}(\partial_{T}X)$ is a group isomorphism, and the inverse is given by $ c\co H_{n-1}(\partial_{T}X)\to H^{\textmd{p}}_{n,n}(X)$.
\item If $q\co X\to X'$ is a quasi-isometric embedding from $X$ to another $n$--dimensional $CAT(0)$ cube complex $X'$, then $q$ induces a monomorphism $q_{\ast}\co H_{n-1}(\partial_{T}X)\to H_{n-1}(\partial_{T}X')$. If $q$ is a quasi-isometry, then $q_{\ast}$ is an isomorphism.
\end{enumerate}
\end{cor}

\begin{proof}
We only need to prove $(2)$, let us approximate $q$ by a Lipschitz quasi-isometric embedding and denote the smallest subcomplex of $X'$ that contains Im $q$ by $I_{q}$, now we have a homomorphism 
\begin{equation}
\label{4.34}
q_{\ast}\co H^{\textmd{p}}_{n,n}(X)\to H^{\textmd{p}}_{n,n}(I_{q})\hookrightarrow H^{\textmd{p}}_{n,n}(X')\,.
\end{equation}
We can define a continuous map $p\co I_{q}\to X$ skeleton by skeleton such that $d(x,p\circ q(x))<D$ for all $x\in X$ ($D$ is some positive constant), which induces $p_{\ast}\co H^{\textmd{p}}_{n,n}(I_{q})\to H^{\textmd{p}}_{n,n}(X)$. It is easy to see $p_{\ast}\circ q_{\ast}=\textmd{Id}$ and $q_{\ast}\circ p_{\ast}=\textmd{Id}$, so the first map in $(\ref{4.34})$ is an isomorphism. Note that the second map in $(\ref{4.34})$ is a monomorphism, thus $q_{\ast}$ is injective and $(2)$ follows from $(1)$.
\end{proof}

We refer to Theorem \ref{6.20} and the remarks after for generalizations of the above corollary.

\begin{remark}
\label{4.36}
Though we are working with $\Bbb Z/2$ coefficients, it is easy to check the analogue of Corollary \ref{4.33} for arbitrary coefficient is also true (the same proof goes through).
\end{remark}

\begin{remark}
\label{4.37}
By the above proof and the argument in Lemma \ref{4.16}, there exists a positive $D'$ which depends on the quasi-isometry constant of $q$ such that
\begin{equation}
\label{4.38}
d_{H}(q(S_{[\tilde{\sigma}]}),S_{q_{\ast}[\tilde{\sigma}]})<D'
\end{equation}
for any $[\tilde\sigma]\in H^{\textmd{p}}_{n,n}(X)$.
\end{remark}

\subsection{Cubical coning}
\label{cubical filling}
Note that the above coning map $c$ does not give us much information about the combinatorial structure of the support set, now we introduce an alternative coning procedure based on the cubical structure. By Lemma \ref{4.16}, we can assume $K=\cup_{i=1}^{N}\Delta_{i}$ where each $\Delta_{i}$ is an all-right spherical $(n-1)$--simplex. Let $\{O_{i}\}_{i=1}^{N}$ be the collection of top-dimensional orthant subcomplexes in $X$ such that $\partial_{T}O_{i}=\Delta_{i}$. By (\ref{2.11}), we can pass to sub-orthants and assume $\{O_{i}\}_{i=1}^{N}$ is a disjoint collection. The nature quotient map $\sqcup_{i=1}^{N}\Delta_{i}\to K$ induces a quotient map $Q\co \sqcup_{i=1}^{N}O_{i}\to CK$ sending the tip of each $O_{i}$ to the cone point of $CK$. We define an inverse map $F\co CK\to\sqcup_{i=1}^{N}O_{i}\subset X$ by sending each $x\in CK$ to some point in $Q^{-1}(x)$.

\begin{lem}
\label{4.39}
$F\co CK\to X$ is a quasi-isometric embedding.
\end{lem}
Recall that $CK$ is endowed with the induced metric from $C_{T}X$.

\begin{proof}
Let $o_{i}$ be the tip of $O_{i}$ and $L=\max_{i\neq j}d(o_{i},o_{j})$. For $x\in O_{i}$ and $y\in O_{j}$, let $c_{i,x}\subset O_{i}$ be the constant speed geodesic ray with $c_{i,x}(0)=o_{i}$ and $c_{i,x}(1)=x$. We can define $c_{j,y}\subset O_{j}$ similarly. Let $c'_{j}$ be the geodesic ray such that (1) it is asymptotic to $c_{j,y}$; (2) it has the same speed as $c_{j,y}$; (3) $c'_{j}(0)=o_{i}$. Then by Lemma \ref{2.4} and convexity of $d(c_{i}(t),c'_{j}(t))$,
\begin{align}
\label{4.40}
d(Q(x),Q(y)) &=\lim_{t\to\infty}\frac{d(c_{i,x}(t),c_{j,y}(t))}{t}=\lim_{t\to\infty}\frac{d(c_{i,x}(t),c'_{j}(t))}{t} \\
&\ge d(c_{i,x}(1),c'_{j}(1))\ge d(c_{i,x}(1),c_{j,y}(1))-d(c'_{j}(1),c_{j,y}(1)) \nonumber \\
&\ge d(x,y)-d(o_{i},o_{j})\ge d(x,y)-L \nonumber
\end{align}
It follows that
\begin{equation}
\label{4.41}
d(F(x),F(y))\leq d(x,y)+L
\end{equation}
for any $x,y\in CK$.

For the other direction, pick $x\in O_{i}$ and $y\in O_{j}$, let us assume without loss of generality that $i\neq j$ and $x,y$ are interior points of $O_{i}$ and $O_{j}$. We extend $\overline{o_{i}x}$ (or $\overline{o_{j}y}$) to get a ray $\overline{o_{i}\xi_{1}}\subset O_{i}$ (or $\overline{o_{j}\xi_{2}}\subset O_{j}$). Let $(Y_{1},Y_{2})=\inc (O_{i},O_{j})$. Since
$d(x,y)\le d(x,Y_{1})+d(Y_{1},Y_{2})+d(y,Y_{2})\le d(x,Y_{1})+d(y,Y_{2})+L$, we can assume with out loss of generality that 
\begin{equation}
\label{4.42}
d(x,Y_{1})\ge \frac{1}{2}(d(x,y)-L)\,.
\end{equation}
From (\ref{4.15}), we have
\begin{align}
\label{4.43}
d(F(x),F(y))&\ge d(x,o_{i})\sin(\angle_{T}(\xi_{1},\xi_{2}))\ge\frac{A}{2}d(x,o_{i})\sin(\angle_{T}(\xi_{1},\partial_{T}Y_{1}))\\ 
&\ge\frac{A}{2}d(x,Y_{1})-L'\ge\frac{A}{4}d(x,y)-L'-\frac{1}{2}L \nonumber
\end{align}
if $\angle_{T}(\xi_{1},\xi_{2})<\pi/2$ and
\begin{align}
\label{4.44}
d(F(x),F(y))\ge d(x,o_{i})\ge d(x,Y_{1})-L'\ge d(x,y)-L'-\frac{1}{2}L
\end{align}
if $\angle_{T}(\xi_{1},\xi_{2})\ge\pi/2$, here $A$ and $L'$ depends on $O_{i}$ and $O_{j}$, but there are finitely many orthants, we can make $A$ and $L'$ uniform.
\end{proof}

Since $X$ is linearly contractible, we can approximate $F$ by a continuous quasi-isometry embedding $F'$ such that $d(F(x),F'(x))\leq L$ for some constant $L$ and any $x\in CK$. Let $K^{(n-2)}$ be the $(n-2)$--skeleton of $K$ and define $\rho\co  CK\to [0,1]$ to be
\begin{center}
$\rho(x)$=$\begin{cases}
1 & \text{if $d(x,CK^{(n-2)})\le 1$}\\
2-d(x,CK^{(n-2)}) & \text{if $1<d(x,CK^{(n-2)})<2$}\\
0 & \text{if $d(x,CK^{(n-2)})\ge 2$}
\end{cases}$
\end{center}
Let
\begin{center}
$F_{1}(x)=\rho(x)F'(x)+(1-\rho(x))F(x)$
\end{center}
for $x\in CK$, here $\rho(x)F'(x)+(1-\rho(x))F(x)$ denotes the point in the geodesic segment $\overline{F'(x)F(x)}$ which has distance $\rho(x)d(F'(x),F(x))$ from $F(x)$. Though $F$ may not be continuous, $F_1$ is continuous. Since the only discontinuity points of $F$ are in the $1$--neighbourhood of $CK^{n-2}$, however, inside such neighbourhood we have $F_1=F'$ by definition. Also note that $d(F(x),F_{1}(x))\le L'$ for all $x\in CK$.

Since $F_{1}=F$ outside the $2$--neighbourhood of $CK^{(n-2)}$, there exists an orthant subcomplex $O'_{i}\subset O_{i}$ such that $F_{1}^{-1}(O'_{i})$ is an orthant in $CK$ for $1\le i\le N$ and
\begin{equation}
\label{4.45}
d_{H}(\textmd{Im}\ F_{1},\cup_{i=1}^{N}O'_{i})<\infty\,.
\end{equation}
Let $[CK]\in H^{\textmd{p}}_{n}(CK)$ be the fundamental class. If $[\tau]=(F_{1})_{\ast}[CK]\in H^{\textmd{p}}_{n,n}(X)$, then 
\begin{equation}
\label{4.46}
\cup_{i=1}^{N}O'_{i}\subset S_{[\tau]}\subset \textmd{Im}\ F_{1}\,.
\end{equation}
The first inclusion follows from the construction of $O'_{i}$ and the second follows from Lemma \ref{3.2}. (\ref{4.45}) and (\ref{4.46}) immediately imply:

\begin{lem}
\label{4.47}
$d_{H}(S_{[\tau]},\cup_{i=1}^{N}O'_{i})<\infty.$
\end{lem}

Now we are ready to prove the main result.
\begin{thm}
\label{4.48}
Let $X$ be a $CAT(0)$ cube complex of dimension $n$. Pick $[\sigma]\in H^{\textmd{p}}_{n,n}(X)$ and suppose $S=S_{[\sigma]}$. Then there is a finite collection {$O_{1},...,O_{k}$} of n--dimensional orthant subcomplexes of $S$ such that
\begin{center}
$d_{H}(S, \cup_{i=1}^{k}O_k)<\infty$.
\end{center}
\end{thm}

\begin{proof}
By Lemma \ref{4.47}, it suffices to show $[\sigma]=[\tau]$ in $H^{\textmd{p}}_{n}(X)$. Note that (\ref{4.46}) implies $\partial_{T}S_{[\tau]}=K$, so $\partial([\tau])=[K]=\partial([\sigma])$, here $[K]$ is the fundamental class of $K$ and $\partial$ is the map in (\ref{4.27}). Thus $[\sigma]=[\tau]$ by Corollary \ref{4.33}.
\end{proof}

In particular, by Lemma \ref{3.4} and Theorem \ref{4.48}, we have: 

\begin{thm}
\label{4.49}
If $X$ is a $CAT(0)$ cube complex of dimension n, then for every $n$--quasiflat Q in X, there is a finite collection {$O_{1},...,O_{k}$} of $n$--dimensional orthant subcomplexes in X such that 
\begin{center}
$d_{H}(Q, \cup_{i=1}^{k}O_k)<\infty$.
\end{center}
\end{thm}

\section{Preservation of top dimensional flats}
\label{topdim flats}
\subsection{The lattice generated by top dimensional quasiflats}
\label{lattice}
We investigate the coarse intersection of the top dimensional quasiflats in this section. 

Let $X$ be a finite dimensional $CAT(0)$ cube complex. For two subsets $A$ and $B$, we say they are \textit{coarsely equivalent} (denoted $A\sim B$) if $d_{H}(A,B)<\infty$. We assume the empty subset is coarsely equivalent to any bounded subset. Denote by $[A]$ the coarse equivalence class contains $A$, we say $[A]\subset[B]$ if there exists $r<\infty$ such that $A\subset N_{r}(B)$. If $[A]\subset[B]$ and $[A]\neq[B]$, we will write $[A]\subsetneq[B]$. Also we define the union $[A]\cup [B]$ to be $[A\cup B]$, but intersection is not well-defined in general. 

The class $[A]$ is \textit{admissible} if it can be represented by a subset which is a finite union of (not necessarily top dimensional) orthant subcomplexes in $X$ ($A$ is allowed to be empty). Let $\mathcal{A}(X)$ be the collection of admissible classes of subsets in $X$. Pick $[A_{1}],[A_2]\in \mathcal{A}(X)$. We define another two operations between $[A_1]$ and $[A_2]$ as follows.
\begin{enumerate}
\item By Lemma \ref{2.10} (4), there exists $r<\infty$ such that $[N_{r_{1}}(A_{1})\cap N_{r_1}(A_2)]=[N_{r_{2}}(A_{1})\cap N_{r_2}(A_2)]$ for any $r_{1}\ge r$ and $r_{2}\ge r$. We define the intersection $[A_{1}]\cap[A_2]$ to be $[N_{r}(A_{1})\cap N_{r}(A_2)]$, which is also admissible. 
\item By Lemma \ref{2.10} (4), there exists $r<\infty$ such that $[A_1\setminus N_{r_1}(A_2)]=[A_1\setminus N_{r_2}(A_2)]$ for any $r_{1}\ge r$ and $r_{2}\ge r$. We define the subtraction $[A_{1}]-[A_2]$ to be $[A_{1}\setminus N_{r}(A_2)]$, which is also admissible. 
\end{enumerate}

If $Y$ is another $CAT(0)$ cube complex with $\dim(Y)=\dim(X)$ and there is a quasi-isometry $f\co X\to Y$, then we define $f_{\sharp}([A])$ to be $[f(A)]$ (this is well-defined since $A\sim B$ implies $f(A)\sim f(B)$). Note that 
\begin{enumerate}
\item $f_{\sharp}([A])\cup f_{\sharp}([B])=f_{\sharp}([A]\cup[B])$;
\item if $[A],[B],[f(A)]$ and $[f(B)]$ are all admissible, then $f_{\sharp}([A])\cap f_{\sharp}([B])=f_{\sharp}([A]\cap[B])$ and $f_{\sharp}([A])- f_{\sharp}([B])=f_{\sharp}([A]-[B])$.
\end{enumerate}
We only verify the last equality. Since $f$ is a quasi-isometry, there exist constants $a>1,b>0$ such that for $r$ large enough, we have 
\begin{equation*}
f(A)\setminus N_{ar+b}(f(B))\subset f(A\setminus N_r(B))\subset f(A)\setminus N_{\frac{r}{a}-b}(f(B)).
\end{equation*}
Since $[f(A)]$ and $[f(B)]$ are admissible, the first term and the last term of the above inequality are in the same coarse class for $r$ large enough. This finishes the proof.

Let $\mathcal{Q}(X)$ be the collection of top dimensional quasiflats in $X$, modulo the above equivalence relation. Theorem \ref{4.49} implies $\mathcal{Q}(X)\subset\mathcal{A}(X)$. Let $\mathcal{KQ}(X)$ be the smallest subset of $\mathcal{A}(X)$ which contains $\mathcal{Q}(X)$, and is closed under union, intersection and subtraction defined above. Moreover precisely, each element $\mathcal{KQ}(X)$ can be written as a finite string of elements of $\mathcal{Q}(X)$ with union, intersection or subtraction between adjacent terms and braces which indicate the order of these operations. Let $f\co  X\to Y$ be a quasi-isometry. Then by induction on the length of the string, one can show $[f(A)]$ is admissible and $[f(A)]\in \mathcal{KQ}(Y)$ for each $[A]\in \mathcal{KQ}(X)$. By considering the quasi-isometry inverse of $f$, we have the following:

\begin{thm}
\label{5.1}
Let $X$ and $Y$ be $n$--dimensional $CAT(0)$ cube complexes. If $f\co X\to Y$ is a quasi-isometry, then $f$ induces a bijection $f_{\sharp}\co \mathcal{KQ}(X)\to\mathcal{KQ}(Y)$. Moreover, for $[A],[B]\in \mathcal{KQ}(X)$, we have $f_{\sharp}([A])\cup f_{\sharp}([B])=f_{\sharp}([A]\cup[B])$, $f_{\sharp}([A])\cap f_{\sharp}([B])=f_{\sharp}([A]\cap[B])$ and $f_{\sharp}([A])- f_{\sharp}([B])=f_{\sharp}([A]-[B])$.
\end{thm}

For $[A]$ admissible, pick a representative in $[A]$ which is a finite union of orthant complexes, we define the \textit{order} of $[A]$, denoted $|[A]|$, to be the number of top dimensional orthant complexes in the representative. By Lemma \ref{2.10}, this definition does not depend on the choice of the representative. Since every element in $\mathcal{KQ}(X)$ is admissible, we have a map $\mathcal{KQ}(X)\to \{0\}\cup \Bbb Z^{+}$ with the following properties:
\begin{enumerate}
\item $|[Q]|\ge 2^{dimX}$ for $[Q]\in\mathcal{Q}(X)$.
\item $|[A]\cup[B]|=|[A]|+|[B]|-|[A]\cap[B]|$ for $[A],[B]\in\mathcal{KQ}(X)$.
\item Let $f$ be as in Theorem \ref{5.1}. Then $|[A]|=0$ if and only if $|f_{\sharp}([A])|=0$ for $[A]\in\mathcal{KQ}(X)$.
\end{enumerate}
The first assertion follows from (\ref{3.11}).

We say an element $[A]\in\mathcal{KQ}(X)$ is \textit{essential} if $|[A]|>0$. $[A]$ is a \textit{minimal essential element} if for any $[B]\in\mathcal{KQ}(X)$ with $[B]\subsetneq[A]$, we have $|[B]|=0$. Minimal essential elements have the following properties.
\begin{enumerate}
\item For any $[A]\in\mathcal{KQ}(X)$, there is a decomposition $[A]=(\cup_{i=1}^{N}[A_{i}])\cup [B]$ such that each $[A_{i}]$ is a minimal essential element and $|[B]|=0$. We also require $[B]$ and each $[A_i]$ are in $\mathcal{KQ}(X)$.
\item For two different minimal essential element $[A_{1}],[A_{2}]\in\mathcal{KQ}(X)$, $|[A_{1}]\cap[A_{2}]|=0$, thus $|[A_{1}]\cup[A_{2}]|=|[A_{1}]|+|[A_{2}]|$.
\item Let $f$ be as above. If $[A]$ is a minimal essential element in $\mathcal{KQ}(X)$, then $f_{\sharp}([A])$ is also a minimal essential element.
\end{enumerate}
We only prove (1). For each top dimensional orthant subcomplex $[O_i]$ such that $[O_i]\subset[A]$, let $[A_i]$ be the minimal element in $\mathcal{KQ}(X)$ which contains $[O_i]$. We claim $[A_i]$ is minimal essential. Suppose the contrary is true, i.e. there exists $[A'_i]\in \mathcal{KQ}(X)$ such that $|[A'_i]|\neq 0$ and $[A'_i]\subsetneq[A_i]$. The minimality of $[A_i]$ implies $[O_i]\subset [A'_i]$ does not hold. However, in such case $[O_i]\subset [A_i]-[A'_i]\subsetneq [A_i]$, which contradicts the minimality of $[A_i]$. We choose $[B]=[A]-[\cup_{i=1}^{N}A_{i}]$.

\begin{lem}
\label{5.2}
Let $X$ and $Y$ be $n$--dimensional $CAT(0)$ cube complexes and let $f\co X\to Y$ be a $(L',A')$--quasi-isometry. If $|f_{\sharp}([A])|=|[A]|$ for any minimal essential element $[A]\in\mathcal{KQ}(X)$, then there exists a constant $C=C(L',A')$ such that for any top dimensional flat $F\subset X$, there exists a top dimensional flat $F'\subset Y$ such that $d_{H}(f(F),F')<C$.
\end{lem}

\begin{proof}
By Theorem \ref{5.1} and the above discussion, we know $|f_{\sharp}([A])|=|[A]|$ for any $[A]\in\mathcal{KQ}(X)$, in particular $|[f(F)]|=|[F]|=2^{n}$ ($n=\dim(X)=\dim(Y)$). By Lemma \ref{3.4}, let $[\sigma]\in H^{\textmd{p}}_n(Y)$ be the class such that $d_{H}(S_{[\sigma]},f(F))<\infty$. By Theorem \ref{4.48}, $S_{[\sigma]}$ is Hausdorff close to a union of $2^{n}$ orthant subcomplexes. Thus $\partial_{T}S_{[\sigma]}$ is contained in $2^{n}$ right-angled spherical simplexes of dimension $n-1$. Then $\mathcal{H}^{n-1}(\partial_{T}S_{[\sigma]})\le \mathcal{H}^{n-1}(\Bbb S^{n-1})$. Pick base point $p\in S_{[\sigma]}$ and we consider the logarithmic map $\log_{p}\co C_{T}Y\to Y$. Lemma \ref{3.6} implies $S_{[\sigma]}\subset \log_{p}(C_TS_{[\sigma]})$. Thus
\begin{equation*}
\frac{\mathcal{H}^{n}(B(p,r)\cap S_{[\sigma]})}{r^{n}}\le \frac{\mathcal{H}^{n}(B(p,r)\cap \log_{p}(C_TS_{[\sigma]}))}{r^{n}}\le \frac{\mathcal{H}^{n}(B(o,r)\cap C_TS_{[\sigma]})}{r^{n}}\le \omega_{n}.
\end{equation*}
Here $o$ is the cone point in $C_T Y$ and $\omega_n$ is the volume of unit ball in $\Bbb E^{n}$. The second inequality follows from that $\log_{p}$ is 1--Lipschitz and the third inequality follows from $\mathcal{H}^{n-1}(\partial_{T}S_{[\sigma]})\le \mathcal{H}^{n-1}(\Bbb S^{n-1})$. By Theorem \ref{3.9} (2), $S_{[\sigma]}$ is isometric to $\Bbb E^{n}$.
\end{proof}

\subsection{The weakly special cube complexes}
\label{weakly special}
It is shown in \cite{MR2421136} that the assumption of Lemma \ref{5.2} is satisfied for $2$--dimensional RAAGs. Our goal in this section is to find an appropriate class of cube complexes which shares some key properties of the canonical $CAT(0)$ cube complexes of RAAGs such that assumption of Lemma \ref{5.2} is satisfied. In \cite{MR2377497}, Haglund and Wise introduced a class of RAAG-like cube complexes, which are called special cube complexes. We adjust their definition for our purposes in the following way.

\begin{definition}
\label{5.3}
A cube complex $K$ is \textit{weakly special} if it satisfies:
\begin{enumerate}
\item $K$ is non-positively curved.
\item No hyperplane \textit{self-osculates} or \textit{self-intersect}.
\end{enumerate}
\end{definition}

The second condition means for any vertex $v$ and two distinct edges $e_{1}$ and $e_{2}$ such that $v\in e_{1}\cap e_{2}$, the hyperplanes dual to $e_{1}$ and $e_{2}$ are different.

If $K$ is compact, then there exists a finite sheet weakly special cover $\bar{K}$ of $K$ such that every hyperplane in $\bar{K}$ are \textit{two-sided}, i.e. there exists a small neighbourhood of the hyperplane which is a trivial interval bundle over the hyperplane. This follows from the argument in \cite[Proposition 3.10]{MR2377497}. 

In the rest of this section, we will denote by $W'$ a compact weakly special cube complex, and $W$ the universal cover of $W'$. Since we mainly care about $W$, so there is no loss of generality by assuming every hyperplane in $W'$ is two-sided. The goal of this section is to prove the following
\begin{thm}
\label{5.18}
Let $W'_{1}$ and $W'_{2}$ be two compact weakly special cube complexes with $\dim(W'_{1})=\dim(W'_{2})=n$. Suppose $W_{1}$, $W_{2}$ are the universal covers of $W'_{1}$, $W'_{2}$ respectively. If $f\co W_{1}\to W_{2}$ is a $(L,A)$--quasi-isometry, then there exists a constant $C=C(L,A)$ such that for any top dimensional flat $F\subset W_{1}$, there exists a top dimensional flat $F'\subset W_{2}$ with $d_{H}(f(F),F')<C$.
\end{thm}

This theorem follows from Lemma \ref{5.2} and the following lemma.

\begin{lem}
\label{5.12}
Let $W_1,W_2$ and $f$ be as in Theorem \ref{5.18}. If $f_{\sharp}\co \mathcal{KQ}(W_{1})\to\mathcal{KQ}(W_{2})$ is the induced bijection in Theorem \ref{5.1}, then $|f_{\sharp}([A])|=|[A]|$ for any minimal essential element $[A]\in\mathcal{KQ}(W_{1})$.
\end{lem}

In the rest of this section, we will prove Lemma \ref{5.12}. 

We label the vertices and edges of $W'$ by $\{\bar{v}_{i}\}_{i=1}^{N_{v}}$ and $\{\bar{e}_{i}\}_{i=1}^{N_{e}}$ such that (1) different vertices have different labels; (2) two edges have the same label if and only if they are dual to the same hyperplane. We also choose an orientation for each edge such that if two edges are dual to the same hyperplane, the their orientations are consistent with parallelism (this is possible since each hyperplane is two-sided). All the labellings and orientations lift to the universal cover $W$. The edges in $W$ dual to the same hyperplane also share the same label.

For every edge path $\omega$ in $W'$ or $W$, define $L(\omega)$ to be the word $\bar{v}_{i}\bar{e}^{\epsilon_{i_{1}}}_{i_{1}}\bar{e}^{\epsilon_{i_{2}}}_{i_{2}}\bar{e}^{\epsilon_{i_{3}}}_{i_{3}}\cdots$ where $\bar{v}_{i}$ is the label of the initial vertex of $\omega$, $\bar{e}_{i_{j}}$ is the label of the $j$--th edge and $\epsilon_{i_{j}}=\pm 1$ records the orientation of the $j$--th edge. 

Definition \ref{5.3} and the way we label $W'$ imply:
\begin{enumerate}
\item For two edges $e'_1$ and $e'_2$ in $W'$ dual to the same hyperplane, $e'_1$ is embedded if and only if $e'_2$ is embedded, i.e. its end points are distinct.
\item Pick any vertex $v'_{i}\in W'$, then two distinct edges $e'_1$ and $e'_2$ with $v'_{i}\in e'_1\cap e'_2$ have different labels.
\item If $\omega'_{1}$ and $\omega'_{2}$ are two edge paths in $W'$ such that $L(\omega'_{1})=L(\omega'_{2})$, then $\omega'_{1}=\omega'_{2}$. If $\omega_{1}$ and $\omega_{2}$ are two edge paths in $W$ such that $L(\omega_{1})=L(\omega_{2})$, then there exists a unique deck transformation $\gamma$ such that $\gamma(\omega_{1})=\omega_{2}$.
\end{enumerate}

We will be using the following simple observation repeatedly.
\begin{lem}
\label{5.4}
Pick vertices $v_{1}$ and $v_{2}$ in $W$ which has the same label. For $i=1,2$, let $\{l_{ij}\}_{j=1}^{k}$ be a collection such that each $l_{ij}$ is a geodesic ray, a geodesic segment or a complete geodesic that contains $v_{i}$. Suppose 
\begin{enumerate}
\item Each $l_{ij}$ is a subcomplex of $W$.
\item For each $j$, there is a graph isomorphism $\phi_{j}\co l_{1j}\to l_{2j}$ which preserves the labels of vertices and edges and the orientations of edges, moreover $\phi_{j}(v_{1})=v_{2}$.
\item The convex hull of $\{l_{1j}\}_{j=1}^{k}$, which we denote by $K_{1}$, is a subcomplex isometric to $\prod_{j=1}^{k}l_{1j}$.
\end{enumerate}
Then the convex hull of $\{l_{2j}\}_{j=1}^{k}$, which we denote by $K_{2}$, is a subcomplex isometric to $\prod_{j=1}^{k}l_{2j}$. Moreover, let $\gamma$ be the deck transformation such that $\gamma(v_{1})=v_{2}$. Then $\gamma(K_{1})=K_{2}$.
\end{lem}

Let $\dim(W)=n$ and let $O$ be a top dimensional orthant subcomplex in $W$. We now construct a suitable doubling of $O$ which will serve as a basic move to analyze the minimal essential elements in $\mathcal{KQ}(W)$. 

Let $\{r_{j}\}_{i=1}^{n}$ be the geodesic rays emanating from the tip of $O$ such that $O$ is the convex hull of $\{r_{j}\}_{j=1}^{n}$. We parametrize $r_{1}$ by arc length. Since the labelling of $W$ is finite, we can find a sequence of non-negative integers $\{n_{j}\}_{j=1}^{\infty}$ with $n_{j}\to\infty$ such that the label and orientation of incoming edge at $r_{1}(n_{j})$, the label and orientation of outcoming edge at $r_{1}(n_{j})$ and the label of $r_{1}(n_{j})$ do not depend on $j$.

We identify $O$ with $[0,\infty)\times O'$ where $O'$ is a $(n-1)$--dimensional orthant orthogonal to $r_{1}$. By our choice of $r_{1}(n_{1})$ and $r_{1}(n_{2})$, we can extend $\overline{r_{1}(n_{2})r_{1}(n_{1})}$ over $r_{1}(n_{1})$ to reach a vertex $v$ such that $L(\overline{r_{1}(n_{1})v})=L(\overline{r_{1}(n_{2})r_{1}(n_{1})})$ ($v$ does not need to lie on $r_1$, see Figure \ref{fig:1}).
\begin{figure}[ht!]
	\labellist
	\small\hair 2pt
	\pinlabel $u$ at 11 137
	\pinlabel $v$ at 56 81
	\pinlabel $\overline{r_1(n_1)v}$ at 122 107
	\pinlabel $\overline{r_1(n_2)r_1(n_1)}$ at 195 88
	\pinlabel $\overline{r_1(n_3)r_1(n_2)}$ at 286 88
	\pinlabel $r_1(0)$ at 65 18
	\pinlabel $r_1(n_1)$ at 149 16
	\pinlabel $r_1(n_2)$ at 249 16
	\pinlabel $r_1(n_3)$ at 337 16
	\endlabellist
	\centering
\includegraphics[scale=0.8]{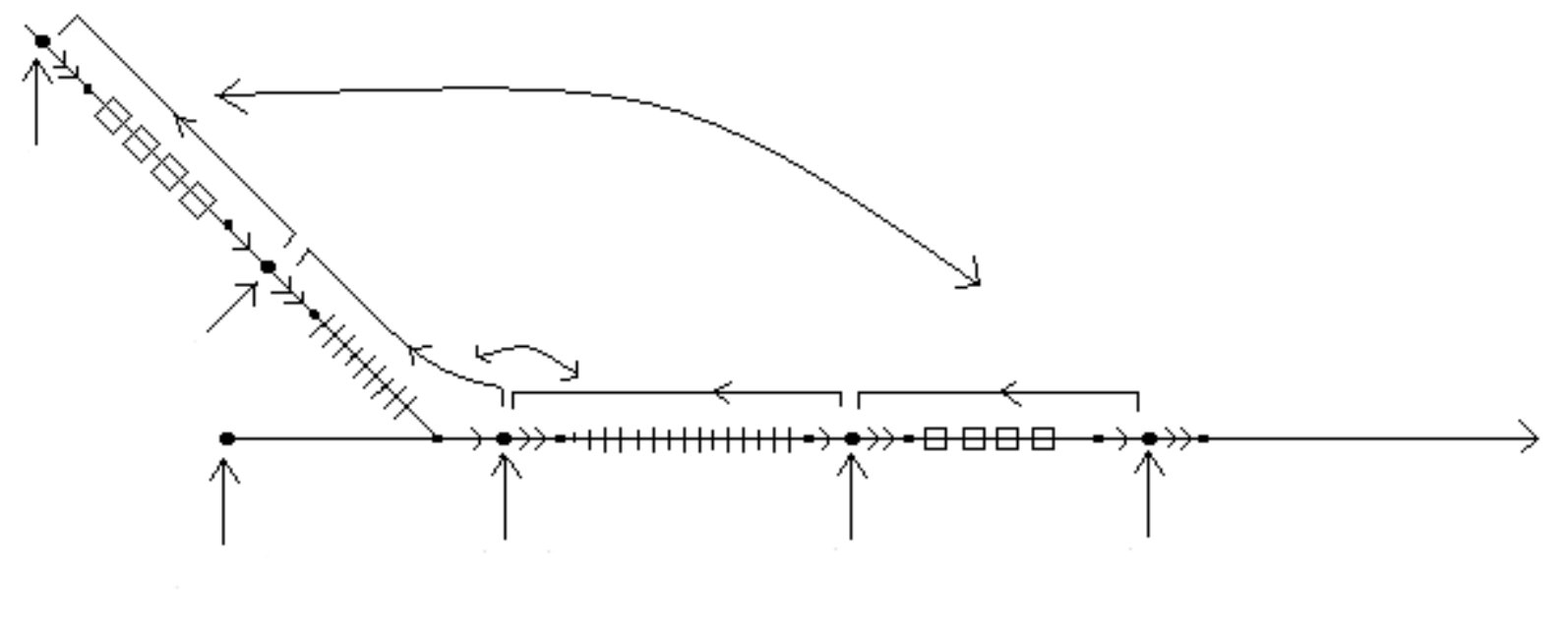}
\caption{}\label{fig:1}
\end{figure}

Let $K_{1}$ be the convex hull of $\{n_{2}\}\times O'$ and $\overline{r_{1}(n_{2})r_{1}(n_{1})}$. Then $K_1$ is of form $K_{1}=[n_{1},n_{2}]\times O'$. Note that the parallelism map between $\{n_1\}\times O'$ and $\{n_2\}\times O'$ preserves labelling and orientation of edges. Then it follows from Lemma \ref{5.4} that the convex hull of $\{n_{1}\}\times O'$ and $\overline{r_{1}(n_{1})v}$ is a subcomplex isometric to $\overline{r_{1}(n_{1})v}\times O'$ (actually if $\gamma\in\pi_1(W')$ is the deck transformation satisfying $\gamma(r_1(n_2))=r_1(n_1)$, then $\gamma(K_1)$ is the convex hull of $\{n_{1}\}\times O'$ and $\overline{r_{1}(n_{1})v}$). We call this subcomplex the \textit{mirror} of $K_{1}$ and denote it by $K'_{1}$. Since $K_1\cap K'_1=\{n_1\}\times O'$, $K'_{1}\cup ([n_1,\infty)\times O')$ is again an orthant, see Figure \ref{fig:2}.
\begin{figure}[ht!]
	\labellist
	\small\hair 2pt
	\pinlabel $u$ at 89 362
	\pinlabel $v$ at 142 321
	\pinlabel $u_1$ at 226 262
	\pinlabel $u_2$ at 315 262
	\pinlabel $u_3$ at 399 262
	\pinlabel $K_1$ at 277 103
	\pinlabel $K_2$ at 359 97
	\pinlabel $K_3$ at 452 92
	\pinlabel $K'_1$ at 162 228
	\pinlabel $K'_2$ at 108 273
	\pinlabel $K'_3$ at 53 321
	\endlabellist
	\centering
\includegraphics[scale=0.5]{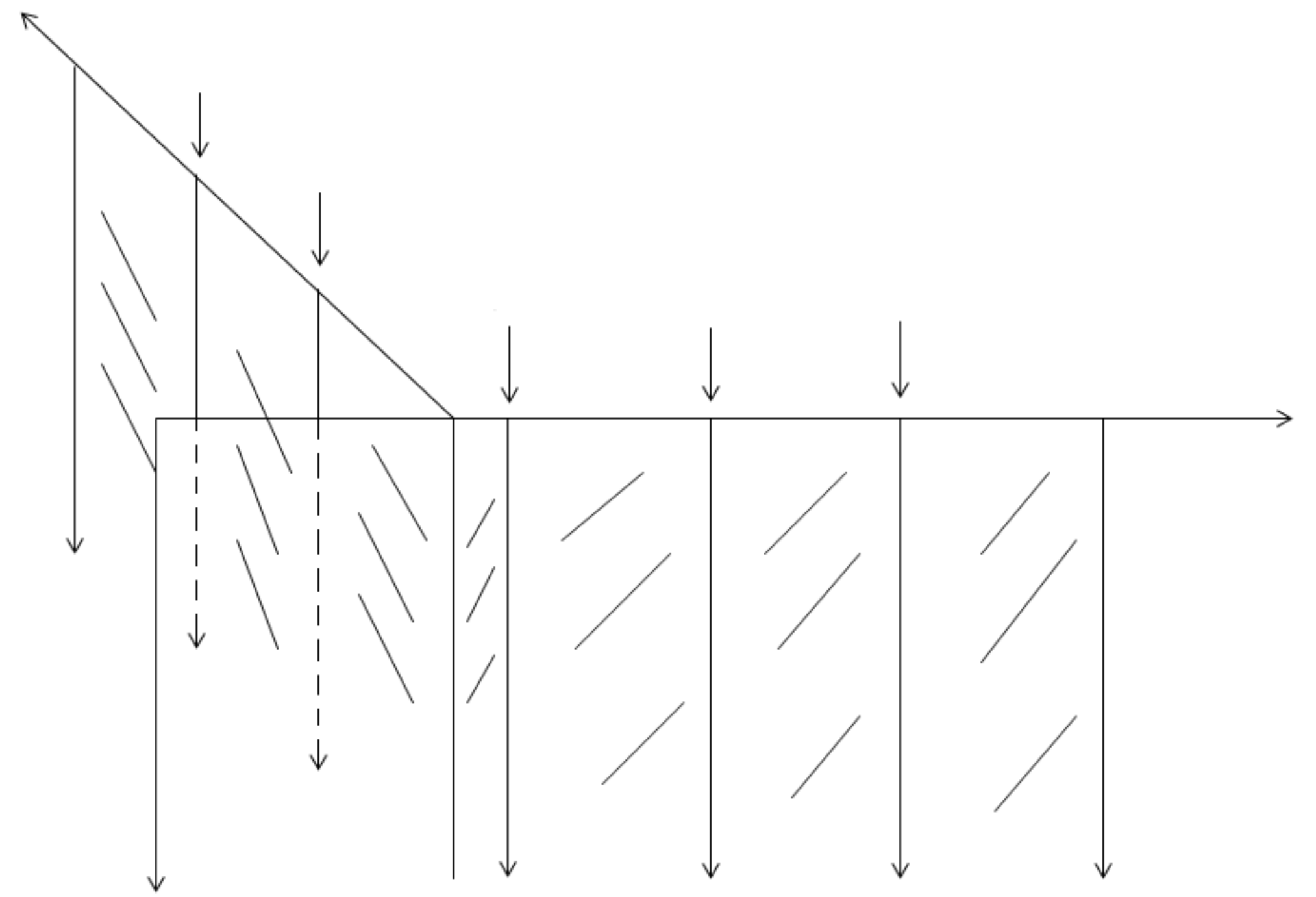}
\caption{}\label{fig:2}
\end{figure}

Let $K_{2}=[n_{2},n_{3}]\times O'$. We extend $\overline{r_1(n_1)v}$ over $v$ to reach a vertex $u$ such that $L(\overline{vu})=L(\overline{r_1(n_3)r_1(n_2)})$. Note that the parallelism map between $\{v\}\times O'$ and $\{n_3\}\times O'$ preserves labelling and orientation of edges. Then it follows from Lemma \ref{5.4} that the convex hull of $\{v\}\times O'$ and $\overline{vu}$ is a subcomplex isometric to $K_2$ (actually if $\gamma\in\pi_1(W')$ is the deck transformation satisfying $\gamma(r_1(n_3))=v$, then $\gamma(K_2)$ is the convex hull of $\{v\}\times O'$ and $\overline{vu}$). This convex hull is called the mirror of $K_2$, and it denoted by $K'_2$. Since $\overline{vu}$ and $\overline{r_{1}(n_{1})v}$ fit together to form a geodesic segment, $K'_1\cap K'_2=\{v\}\times O'$. Thus $K'_2\cup K'_{1}\cup ([n_1,\infty)\times O')$ is again an orthant. We can continue this process, and consecutively construct the mirror of $K_{i}=[n_{i},n_{i+1}]\times O'$ in $W$ (denoted $K'_{i}$) arranged in the pattern indicated in the above picture. Similarly one can verify that $K'_i$ is isometric to $K_i$ and $K'_i\cap K'_{i+1}$ is isometric to $O'$.

Now we obtain a subcomplex $K=(\cup_{i=1}^{\infty}K_{i})\cup(\cup_{i=1}^{\infty}K'_{i})$. It is clear that $[O]\subset[K]$. The discussion in the previous paragraph implies that $\cup_{i=1}^{\infty}K'_{i}$ is also a top dimensional orthant. We will call it the \textit{mirror} of $O$. Moreover, $K$ is isometric to $\Bbb R\times(\Bbb R_{\ge 0})^{n-1}$. More generally, by the same argument as above and Lemma \ref{5.4}, we have the following result.
\begin{lem}
\label{5.6}
If $K\subset W$ is a convex subcomplex isometric to $(\Bbb R_{\ge 0})^{k}\times \Bbb R^{n-k}$, then there exists a convex subcomplex $K'$ isometric to $(\Bbb R_{\ge 0})^{k-1}\times \Bbb R^{n-k+1}$ such that $[K]\subset[K']$.
\end{lem}

Pick a minimal essential element $[A]\in\mathcal{KQ}(W)$, then there exists a top dimensional orthant $O$ with $[O]\subset[A]$ ($[O]$ may not be an element in $\mathcal{KQ}(W)$). Using Lemma \ref{5.6}, we can double the orthant $n$ times to get a top dimensional flat $F$ with $[O]\subset[F]$. Since $[A]$ is minimal, $[A]\cap [F]=[A]$, which implies the following result:
\begin{cor}
\label{5.7}
If $[A]\in\mathcal{KQ}(W)$ is a minimal essential element, then there exists a top dimensional flat $[F]$ such that $[A]\subset[F]$. In particular, $|[A]|\le 2^{\dim(W)}=2^{n}$.
\end{cor}

Pick a top dimensional orthant subcomplex $O$ and denote the $(n-1)$--faces of $O$ by $\{O_{i}\}_{i=1}^{n}$. $[O_{i}]$ is \textit{branched} if there exist top dimensional orthant subcomplexes $O'$ and $O''$ such that $[O]$, $[O']$ and $[O'']$ are distinct elements and $[O]\cap[O']=[O]\cap[O'']=[O_{i}]$, otherwise $[O_{i}]$ is called \textit{unbranched}.

\begin{lem}
\label{5.8}
If $O$ and $O_{i}$ are as above, then $[O_{i}]$ is branched if and only if there exists a sub-orthant $O'_{i}\subset O_{i}$ and geodesic rays $l_{1}$, $l_{2}$ and $l_{3}$ emanating from the tip of $O'_{i}$ such that
\begin{enumerate}
\item $[O'_{i}]=[O_{i}]$.
\item $[l_{1}]$, $[l_{2}]$ and $[l_{3}]$ are distinct.
\item The convex hull of $l_{j}$ and $O'_{i}$ is a top dimensional orthant for $1\le j\le 3$.
\end{enumerate} 
\end{lem}

\begin{proof}
If $O_{i}$ is branched, let $O'$ and $O''$ be the orthant subcomplexes as above, we can assume $O'\cap O=O''\cap O=\emptyset$. Let $(Y_{1},Y_{2})=\inc (O,O')$. Since $Y_{1}$ and $Y_{2}$ bound a copy of $Y_{1}\times [0,d(O,O')]$ inside $W$, $\dim(Y_{1})$=$\dim(Y_{2})\le n-1$. However, (\ref{2.11}) implies $[Y_{1}]=[O]\cap[O']=[O_{i}]$, so $Y_{1}$ and $Y_2$ are $(n-1)$--dimensional orthant subcomplexes. We can find a copy of $Y_{2}\times[0,\infty)$ inside $O'$ and we claim $Y_{1}\times [0,d(O,O')]\cup Y_{2}\times[0,\infty)$ is also a top dimensional orthant subcomplex.

To see this, note that $(Y_{1}\times [0,d(O,O')])\cap (Y_{2}\times[0,\infty))=Y_{2}$. Pick $y\in Y_{2}$, let $\{v_{i}\}_{i=1}^{n-1}$ be mutually orthogonal directions in $\Sigma_{y}Y_{2}$. Moreover, we can assume each $v_{i}$ is in the $0$--skeleton of $\Sigma_{y}Y_{2}\subset\Sigma_{y}W$. Let $v\in\Sigma_{y}(Y_{1}\times [0,d(O,O')])$ be the direction corresponding to the $[0,d(O,O')]$ factor and let $v'\in\Sigma_{y}(Y_{2}\times[0,\infty))$ be the direction corresponding to the $[0,\infty)$ factor. It is clear that $v$ and $v'$ are distinct points in the $0$--skeleton of $\Sigma_{y}W$. If $d(v,v')=\pi/2$, then $v$, $v'$ and $\{v_{i}\}_{i=1}^{n-1}$ would be mutually orthogonal directions, which yields a contradiction with the fact that $\dim(W)=n$. Thus $d(v,v')=\pi$ and $Y_{1}\times [0,d(O,O')]\cup Y_{2}\times[0,\infty)$ is indeed a top dimensional orthant subcomplex. 

Note that the orthant constructed above is the convex hull of $Y_{1}$ and some geodesic ray $l$ emanated from the tip of $Y_{1}$. We can repeat this argument for $O''$ to obtain the required sub-orthants and geodesic rays in the lemma. The other direction of the lemma is trivial.
\end{proof}

Let $O$, $\{r_{j}\}_{j=1}^{n}$, $\{n_{i}\}_{i=1}^{\infty}$, $K_{i}$ and $K'_{i}$ be as in the discussion before Lemma \ref{5.6}. Let $a_{j}=n_{j+1}-n_{1}$ for $j\ge 0$. We identify $(\cup_{i=1}^{\infty}K_{i})\cup(\cup_{i=1}^{\infty}K'_{i})$ with $\Bbb R\times\prod_{j=2}^{n}r_{j}$ such that $K_{i}=[a_{i-1},a_{i}]\times\prod_{j=2}^{n}r_{j}$. Thus $K'_{i}=[-a_{i},-a_{i-1}]\times\prod_{j=2}^{n}r_{j}$. Let $l$ be the unit speed complete geodesic line in $W$ such that $l(0)=r_1(n_1)$ and it is parallel to the $\Bbb R$ factor. For $x\in \Bbb R$, we denote the geodesic ray in $(\cup_{i=1}^{\infty}K_{i})\cup(\cup_{i=1}^{\infty}K'_{i})$ that starts at $l(x)$ and goes along the $r_{j}$ factor by $\{x\}\times r_{j}$.

Let $\gamma_i\in \pi_1(W')$ be the deck transformation satisfying $\gamma_i(l(a_i))=l(-a_{i-1})$. Then by our construction, $\gamma_i(K_i)=K'_{i}$. Moreover, under the product decomposition $K_{i}=[a_{i-1},a_{i}]\times\prod_{j=2}^{n}r_{j}$ and $K'_{i}=[-a_{i},-a_{i-1}]\times\prod_{j=2}^{n}r_{j}$, $\gamma_i$ maps $[a_{i-1},a_i]$ to $[-a_{i},-a_{i-1}]$ and fixes the factor $\prod_{j=2}^{n}r_{j}$ pointwise.

Let $\tilde{O}=\cup_{i=1}^{\infty}K'_{i}$ be the mirror of $O$. There is an isometry $\rho$ acting on $(\cup_{i=1}^{\infty}K_{i})\cup(\cup_{i=1}^{\infty}K'_{i})=\Bbb R\times\prod_{j=2}^{n}r_{j}$ by flipping the $\Bbb R$ factor (the other factors are fixed). For $1\le j\le n$, let $O_{j}$ be the $(n-1)$--face of $O$ which is orthogonal to $r_{j}$ and let $\tilde{O}_{j}$ be the $(n-1)$--face of $\tilde{O}$ such that $[\rho(\tilde{O}_{j})]=[O_{j}]$ (Recall that $[O]=[\cup_{i=1}^{\infty}K_{i}]$).

\begin{lem}
\label{5.9}
$[O_{j}]$ is branched if and only if $[\tilde{O}_{j}]$ is branched.
\end{lem}

\begin{proof}
If $j=1$, then $[O_{1}]=[\tilde{O}_{1}]=[O]\cap[\tilde{O}]$ and the lemma is trivial, so we assume $j\neq 1$. If $[O_{j}]$ is branched, then by Lemma \ref{5.8}, we can assume without loss of generality (one might need to modify $K_{i}$ and $K'_{i}$ by cutting off suitable pieces and replace $l$ by a geodesic in $\Bbb R\times\prod_{j=2}^{n}r_{j}$ which is parallel to $l$) that there exist $i_{0}\ge 0$ and geodesic rays $c_{1}$, $c_{2}$, $c_{3}$ emanating from $l(a_{i_{0}})$ such that $[c_{1}]$, $[c_{2}]$, $[c_{3}]$ are distinct elements and the convex hull of $c_{m}$, $l([a_{i_{0}},\infty))$ and $\{a_{i_{0}}\}\times r_{k}$ ($k\neq 1,j$), which we denote by $H_{m}$, is a top dimensional orthant subcomplex for $1\le m\le 3$.

Let $\gamma$ be the deck transformation satisfying $\gamma(l(a_{i_{0}}))=l(-a_{i_{0}})$. Such $\gamma$ exists by the construction of $l$ (in the previous paragraph, we possibly replace the original $l$ by a geodesic parallel to $l$, however, the same $\gamma$ works). Let $\tilde{c}_{m}=\gamma(c_{m})$ for $1\le m\le 3$. Then $[\tilde{c}_{1}]$, $[\tilde{c}_{2}]$, $[\tilde{c}_{3}]$ are distinct since $\gamma$ is an isometry. Since $\gamma$ is label and orientation preserving, $c_{m}$ and $\tilde{c}_{m}$ correspond to the same word for $1\le m\le 3$, moreover $\gamma(\{a_i\}\times r_k)=\{-a_i\}\times r_k$ for $k\neq 1$. To prove $[\tilde{O}_{j}]$ is branched, it suffices to show the convex hull of $\tilde{c}_{m}$, $l((-\infty,-a_{i_{0}}])$ and $\{-a_{i_{0}}\}\times r_{k}$ ($k\neq 1,j$) is a top dimensional orthant subcomplex. 

For $m=1$, we chop up $H_{1}$ into pieces such that $H_{1}=\cup_{i=i_{0}+1}^{\infty} L_{i}$ and $L_{i}=c_{1}\times l([a_{i-1},a_{i}])\times\prod_{k\neq 1,j} r_{k}$. Let $\gamma_{i}$ be the deck transformation defined before Lemma \ref{5.9} and let $L'_{i}=\gamma(L_{i})$. We claim $\gamma_{i}(c_1\times \{a_{i-1}\}\times\prod_{k\neq 1,j} r_{k})=\gamma_{i+1}(c_1\times \{a_{i+1}\}\times \prod_{k\neq 1,j} r_{k})$ for $i\ge i_0+1$. This claim is a consequence of the following two observations: (1) both sides of the equality contain $l(-a_{i})$; (2) $\gamma_i$, $\gamma_{i+1}$ and the parallelism between $c_1\times \{a_{i-1}\}\times\prod_{k\neq 1,j} r_{k}$ and $c_1\times \{a_{i+1}\}\times\prod_{k\neq 1,j} r_{k}$ preserve labelling and orientation of edges. It follows from the claim that $H'_{1}=\cup_{i=i_{0}+1}^{\infty}L'_{i}$ is a top dimensional orthant subcomplex. By a similar argument as before, we know $\gamma(c_1\times \{a_{i}\}\times\prod_{k\neq 1,j} r_{k})=\gamma_{i_0+1}(c_1\times \{a_{i_0+1}\}\times \prod_{k\neq 1,j} r_{k})$, thus $H'_{1}$ is the convex hull of $\tilde{c}_{1}$, $l((-\infty,-a_{i}])$ and $\{-a_{i}\}\times r_{k}$ ($k\neq 1,j$). Moreover $[H'_{1}]\cap[\tilde{O}]=[\tilde{O}_{j}]$. We can repeat this construction for $\tilde{c}_{2}$ and $\tilde{c}_{3}$, which implies $[\tilde{O}_{j}]$ is branched. If $[\tilde{O}_{j}]$ is branched, by the same argument we can prove $[O_{j}]$ is branched.
\end{proof}

\begin{remark}
\label{5.10}
It is important that we keep track of information from the labels of $O$ while constructing the mirror of $O$, in other words, if we construct $\tilde{O}$ by the pattern indicated in Figure \ref{fig:3}:
\begin{figure}[ht!]
\begin{center}
\begin{tikzpicture}
\draw [dotted, thick] (0,1) -- (1,1);
\draw [thick] (1.5,0) -- (1.5,2);
\node at (2.25,1) {$K'_{1}$};
\draw [thick] (3,0) -- (3,2);
\node at (3.75,1) {$K'_{1}$};
\draw [thick] (4.5,0) -- (4.5,2);
\node at (5.25,1) {$K'_{1}$};
\draw [thick] (6,0) -- (6,2);
\node at (6.75,1) {$K_{1}$};
\draw [thick] (7.5,0) -- (7.5,2);
\node at (8.25,1) {$K_{2}$};
\draw [thick] (9,0) -- (9,2);
\node at (9.75,1) {$K_{3}$};
\draw [thick] (10.5,0) -- (10.5,2);
\draw [dotted, thick] (11,1) -- (12,1);
\end{tikzpicture}
\end{center} 
\caption{}\label{fig:3}
\end{figure}
we will not be able to conclude $[O_{j}]$ is branched from $[\tilde{O}_{j}]$ is branched.
\end{remark}

\begin{lem}
\label{5.11}
If $[A]\in\mathcal{KQ}(W)$ is a minimal essential element, then 
\begin{enumerate}
\item $|[A]|=2^{i}$ for some integer $i$ with $1\le i\le n$.
\item There exists a top dimensional flat $F$ and another $2^{n-i}-1$ minimal essential elements $\{A_{j}\}_{j=1}^{2^{n-i}-1}$ with $|[A_{j}]|=|[A]|$ such that $[F]=[A]\cup(\cup_{j=1}^{2^{n-i}-1}[A_{j}])$.
\end{enumerate}
\end{lem}

\begin{proof}
We find a top dimensional orthant subcomplex $O$ such that $[O]\subset [A]$. By the argument before Lemma \ref{5.6}, we can double this orthant $n$ times to a get top dimensional flat $F$ such that $[O]\subset[F]$. Assume without loss of generality $O\subset F$. Denote by $\{O_{i}\}_{i=1}^{n}$ the $(n-1)$--faces of $O$ and let $\rho_{i}\co F\to F$ be the isometry that fixes $O_{i}$ pointwise and flips the direction orthogonal to $O_{i}$.

Let $G$ be the group generated by $\{\rho_{i}\}_{i=1}^{n}$. Then $G\cong (\Bbb Z/2)^{n}$. We define $\Lambda_{b}=\{1\le i\le n\ |\ [O_{i}]\textmd{\ is\ branched}\}$ and $\Lambda_{u}=\{1\le i\le n\ |\ [O_{i}]\textmd{\ is\ unbranched}\}$. Let $G_{b}$ be the subgroup generated by $\{\rho_{i}\}_{i\in\Lambda_{b}}$ and let $G_{u}$ be the subgroup generated by $\{\rho_{i}\}_{i\in\Lambda_{u}}$. We denote by $G_{i}$ the subgroup generated by $\{\rho_{1}\cdots\rho_{i-1},\rho_{i+1}\cdots\rho_{n}\}$.

\begin{claim}
For any $\gamma\in G$, $[O_{i}]$ is branched if and only if $[\gamma(O_{i})]$ is branched.
\end{claim}

\begin{proof}
Write $\gamma=\rho_{i_{1}}\rho_{i_{2}}\cdots\rho_{i_{k}}$, we prove by induction on $k$. The case $k=0$ is trivial. In general, suppose $[O_{i}]$ is branched if and only if $[\rho_{i_{2}}\cdots\rho_{i_{k}}(O_{i})]$ is branched. It follows from the way we construct $F$ that $[\rho_{i_{1}}\rho_{i_{2}}\cdots\rho_{i_{k}}(O)]$ is the mirror of $[\rho_{i_{2}}\cdots\rho_{i_{k}}(O)]$. So by Lemma \ref{5.9}, $[\rho_{i_{1}}\rho_{i_{2}}\cdots\rho_{i_{k}}(O_{i})]$ is branched if and only if $[\rho_{i_{2}}\cdots\rho_{i_{k}}(O_{i})]$ is branched, thus the claim is true.
\end{proof}

\begin{claim}
$[A]\subset[\cup_{\gamma\in G_{u}}\gamma(O)]$.
\end{claim}

\begin{proof}
If $[O_{i}]$ is branched, by Lemma \ref{5.8}, there exists a subcomplex $M_{i}$ isometric to $(\Bbb R_{\ge 0})^{n-1}\times \Bbb R$ such that $[M_{i}]\cap[F]=[O]$. By Lemma \ref{5.6} we can find a top dimensional flat $F_{i}$ such that $[M_{i}]\subset[F_{i}]$. Since $F_{i}\cap F\neq\emptyset$, by Lemma \ref{2.10} $[F\cap F_{i}]=[F]\cap[F_{i}]$. Note that $F\cap F_{i}$ is a convex subcomplex of $F$ with $|[F\cap F_{i}]\cap[\rho_{i}(O)]|=0$, so $[F\cap F_{i}]\subset[\cup_{\gamma\in G_{i}}\gamma(O)]$. Recall that $[A]$ is a minimal essential element, so $[A]\subset [F]\cap(\cap_{i\in\Lambda_{b}}[F_{i}])=\cap_{i\in\Lambda_{b}}([F_{i}]\cap[F])\subset\cap_{i\in\Lambda_{b}}[\cup_{\gamma\in G_{i}}\gamma(O)]=[\cup_{\gamma\in G_{u}}\gamma(O)]$.
\end{proof}

\begin{claim}
$[\cup_{\gamma\in G_{u}}\gamma(O)]\subset[A]$.
\end{claim}

\begin{proof}
First we need the following observation. Let $[P_1]$ and $[P_2]$ be two different top dimensional orthant complexes. Suppose each $[Q]\in \mathcal{Q}(X)$ satisfies the property that either $[P_1]\subset[Q]$ and $[P_2]\subset[Q]$, or $[P_1]\nsubseteq[Q]$ and $[P_2]\nsubseteq[Q]$. Then this property is also true for each element in $\mathcal{KQ}(X)$. To see this, let $\mathcal{A}_{P_1,P_2}(X)$ be the collection of elements in $\mathcal{A}(X)$ which satisfies this property. Then one readily verifies that $\mathcal{A}_{P_1,P_2}(X)$ is closed under union, intersection and subtraction. Moreover, $\mathcal{Q}(X)\subset \mathcal{A}_{P_1,P_2}(X)$. Thus $\mathcal{KQ}(X)\subset \mathcal{A}_{P_1,P_2}(X)$.

Pick an unbranched face $[O_{i}]$. By Lemma \ref{4.16}, Lemma \ref{4.34} and Lemma \ref{4.37}, for every top dimensional quasiflat $Q$ with $[O]\subset[Q]$, there exists another top dimensional orthant complex $O'$ such that $[O']\subset[Q]$ and $\partial_{T}O'\cap\partial_{T}O=\partial_{T}O_{i}$. This together with Lemma \ref{2.10} (see also Remark \ref{2.13}) imply $[O]\cap[O']=[O_{i}]$, thus $[O']=[\rho_{i}(O)]$ and $[\rho_{i}(O)]\subset[Q]$ (recall that $[O_{i}]$ is unbranched). Similarly, one can prove if $[\rho_{i}(O)]\subset[Q]$ for a top dimensional quasiflat $Q$, then $[O]\subset[Q]$. It follows the above observation that $[\rho_{i}(O)]\subset[A]$ for $i\in\Lambda_{u}$. 

Let $\gamma\in G_{u}$. Write $\gamma=\rho_{i_{1}}\rho_{i_{2}}\cdots\rho_{i_{k}}$ with $i_{j}\in\Lambda_{u}$ for $1\le j\le n$. We will prove claim 3 by induction on $k$. The case $k=1$ is already done by the previous paragraph. In general, we assume $[\rho_{i_{1}}\rho_{i_{2}}\cdots\rho_{i_{k-1}}(O)]\subset[A]$. Note that $[O]\cap[\rho_{i_{k}}(O)]=[O_{i_{k}}]$ where $[O_{i_{k}}]$ is unbranched, so 
\begin{center}
$[\rho_{i_{1}}\rho_{i_{2}}\cdots\rho_{i_{k-1}}(O)]\cap[\rho_{i_{1}}\rho_{i_{2}}\cdots\rho_{i_{k}}(O)]=[\rho_{i_{1}}\rho_{i_{2}}\cdots\rho_{i_{k-1}}(O_{i_{k}})]$. 
\end{center}
Claim 1 implies $[\rho_{i_{1}}\rho_{i_{2}}\cdots\rho_{i_{k-1}}(O_{i_{k}})]$ is also unbranched, so $[\rho_{i_{1}}\rho_{i_{2}}\cdots\rho_{i_{k}}(O)]\subset[A]$ by the same argument as in the previous paragraph.
\end{proof}

Claim 2 and Claim 3 imply $[\cup_{\gamma\in G_{u}}\gamma(O)]=[A]$. So $|[A]|=|G_{u}|$ where $|G_{u}|$ is the order of $G_{u}$. Now the first assertion of the lemma follows. Moreover, for any $\gamma\in G$, let $[A_{\gamma}]\in\mathcal{KQ}(W)$ be the unique minimal essential element such that $[\gamma(O)]\subset[A_{\gamma}]$. Claim 1 implies $\{[\gamma(O_{i})]\}_{i\in\Lambda_{b}}$ and $\{[\gamma(O_{i})]\}_{i\in\Lambda_{u}}$ are the branched faces and unbranched faces of $[\gamma(O)]$ respectively. By the same argument as in Claim 2 and Claim 3, we can show $[A_{\gamma}]=[\cup_{\gamma'\in \gamma G_{u}}\gamma'(O)]$, here $\gamma G_{u}$ denotes the corresponding coset of $G_{u}$. Since there are $|G|/|G_{u}|$ cosets of $G_{u}$, the second assertion of the lemma also follows.
\end{proof}

\begin{proof}[Proof of Lemma \ref{5.12}]
If $|[A]|=2^{n}$, by Corollary \ref{5.7}, we know there exists a top dimensional flat $F$ such that $[A]\subset[F]$, so actually $[A]=[F]$. Then $f(A)$ is a top dimensional quasiflat, thus $|f_{\sharp}([A])|\ge 2^{n}$. However, $f_{\sharp}([A])$ is also minimal essential, so by Corollary \ref{5.7} we actually have $|f_{\sharp}([A])|=2^{n}=|[A]|$. Let $g$ be a quasi-isometry inverse of $f$. If $[A']\in\mathcal{KQ}(W_{2})$ is a minimal essential element, then by the same argument, we know that $|[A']|=2^{n}$ implies $|g_{\sharp}([A'])|=2^{n}=|[A']|$. So $|[A]|=2^{n}$ if and only if $|f_{\sharp}([A])|=2^{n}$ for minimal essential element $[A]\in\mathcal{KQ}(W_{1})$.

In general, we assume inductively $|[A]|=k$ if and only if $|f_{\sharp}([A])|=k$ for any $k\ge 2^{n-i+1}$ and any minimal essential element $[A]\in\mathcal{KQ}(W_{1})$ (we are doing induction on $i$). If $[B_{1}]\in\mathcal{KQ}(W_{1})$ is a minimal essential element with $|[B_{1}]|=2^{n-i}$, then by lemma \ref{5.11}, we can find a top dimensional flat $F$ and another $2^{i}-1$ minimal essential elements $\{[B_{j}]\}_{j=2}^{2^{i}}$ such that $|[B_{j}]|=|[B_1]|$ and
\begin{equation}
\label{5.13}
[F]=\cup_{j=1}^{2^{i}}[B_{j}]\,.
\end{equation}
Since $f(F)$ is a top dimensional flat, we have
\begin{equation}
\label{5.14}
|f_{\sharp}(F)|=|f_{\sharp}(\cup_{j=1}^{2^{i}}[B_{j}])|=|\cup_{j=1}^{2^{i}}f_{\sharp}([B_{j}])|=\sum_{j=1}^{2^{i}}|f_{\sharp}([B_{j}])|\ge 2^{n}\,.
\end{equation}
But our induction assumption implies
\begin{equation}
\label{5.15}
|f_{\sharp}([B_{j}])|<2^{n-i+1}\,.
\end{equation}
Since $f_{\sharp}([B_{j}])$ is minimal essential element for each $j$, (\ref{5.15}) together with assertion $(1)$ of Lemma \ref{5.11} imply
\begin{equation}
\label{5.16}
|f_{\sharp}([B_{j}])|\le 2^{n-i}\,.
\end{equation}
Now (\ref{5.14}) and (\ref{5.16}) imply 
\begin{equation}
\label{5.17}
|[B_{j}]|=|f_{\sharp}([B_{j}])|=2^{n-i}
\end{equation}
for all $j$. By considering the quasi-isometry inverse, we know $|[B]|=2^{n-i}$ if and only if $|f_{\sharp}([B])|=2^{n-i}$ for minimal essential element $[B]\in\mathcal{KQ}(W_{1})$. By Lemma \ref{5.11} $(1)$ and our induction assumption, we have actually proved that $|[B]|=k$ if and only if $|f_{\sharp}([B])|=k$ for any $k\ge 2^{n-i}$ and any minimal essential element $[B]\in\mathcal{KQ}(W_{1})$.
\end{proof}

\subsection{Application to right-angled Coxeter groups and Artin groups}
\label{applications}
\subsubsection{The right-angled Coxeter group case}
For finite simplicial graph $\Gamma$ with vertex set $\{v_{i}\}_{i\in I}$, there is an associated right-angled Artin group (RACG), denoted by $C(\Gamma)$, with the following presentation:
\begin{center}
\{$v_i$, for $i\in{I}\mid v^{2}_{i}=1$ for all $i$; $[v_i,v_j]=1$ if $v_{i}$ and $v_{j}$ are joined by an edge\}.
\end{center}

The group $C(\Gamma)$ has a nice geometric model $D(\Gamma)$, called the \textit{Davis complex}. The $1$--skeleton of $D(\Gamma)$ is the Cayley graph of $C(\Gamma)$ with edges corresponding to $v_i,v^{-1}_i$ identified. For $n\ge 2$, the $n$-skeleton $D^{(n)}(\Gamma)$ of $D(\Gamma)$ is obtained from $D^{(n-1)}(\Gamma)$ by attaching a $n$--cube whenever one finds a copy of the $(n-1)$--skeleton of a $n$--cube inside $D^{(n-1)}(\Gamma)$. This process will terminate after finitely many steps and one obtains a $CAT(0)$ cube complex where $C(\Gamma)$ acts properly and cocompactly. 

The action of $C(\Gamma)$ on $D(\Gamma)$ is not free, however, $D(\Gamma)$ can be realized as the universal cover of a compact cube complex. The following construction is from \cite{davis1999notes}. Let $\{e_{i}\}_{i\in I}$ be the standard basis of $\Bbb R^{I}$ and let $\Box^{I}=[0,1]^{I}\subset \Bbb R^{I}$ be the unit cube with the standard cubical structure. Let $F(\Gamma)$ be the flag complex of $\Gamma$. For each simplex $\Delta\subset F(\Gamma)$, let $\Bbb R^{\Delta}$ be the linear subspace spanned by $\{e_{i}\}_{v_{i}\in\Delta}$. Define $K(\Gamma)$=$\cup_{\Delta}\{$faces of $\Box^{I}$ parallel to $\Bbb R^{\Delta}\}$, here $\Delta$ varies among all simplexes in $F(\Gamma)$. Then the Davis complex $D(\Gamma)$ is exactly the universal cover of $K(\Gamma)$ (\cite[Proposition 3.2.3]{davis1999notes}). 

One can verify that $K(\Gamma)$ is weakly special. In order to apply Theorem \ref{5.18} in a non-trivial way, we need the following extra condition.
\begin{center}
$(\ast)$ There is an embedded top dimensional hyperoctahedron in $F(\Gamma)$.
\end{center} 
One can check there exists a top dimensional flat in $D(\Gamma)$ if and only if $(\ast)$ is true.

\begin{cor}
\label{5.20}
Let $\Gamma_{1}$ and $\Gamma_{2}$ be two finite simplicial graph satisfying $(\ast)$. If $\phi\co  D(\Gamma_{1})\to D(\Gamma_{2})$ is an $(L, A)$--quasi-isometry, then $\dim(D(\Gamma_{1}))=\dim(D(\Gamma_{2}))$. And there is a constant $D=D(L, A)$ such that for any top-dimensional flat $F_{1}$ in $D(\Gamma_{1})$, we can find a flat $F_{2}$ in $D(\Gamma_{2})$ such that
\begin{center}
$d_{H}(\phi(F_{1}),F_{2})< D.$
\end{center}
\end{cor}

\begin{proof}
It suffices to show $\dim(D(\Gamma_{1}))=\dim(D(\Gamma_{2}))$. The rest follows from Theorem \ref{5.18} and the above discussion.

We can assume the quasi-isometry $\phi$ is defined on the $0$--skeleton of $D(\Gamma_1)$. Since $D(\Gamma_2)$ is $CAT(0)$, we can extend $\phi$ skeleton by skeleton to obtain a continuous quasi-isometry. Similarly, we assume quasi-isometry inverse $\phi'$ is also continuous. Since $\phi$ and $\phi'$ are proper, there are induced homomorphisms for the proper homology $\phi_{\ast}\co H^{\textmd{p}}_{\ast}(D(\Gamma_{1}))\to H^{\textmd{p}}_{\ast}(D(\Gamma_{2}))$ and $\phi'_{\ast}\co H^{\textmd{p}}_{\ast}(D(\Gamma_{2}))\to H^{\textmd{p}}_{\ast}(D(\Gamma_{1}))$ (see Section \ref{proper homology}). Note that the geodesic homotopy between $\phi'\circ\phi$ (or $\phi\circ\phi'$) and the identity map is proper, so $\phi_{\ast}\circ\phi'_{\ast}= \textmd{Id}$ and $\phi'_{\ast}\circ\phi_{\ast}= \textmd{Id}$. Hence $\phi_{\ast}$ is an isomorphism.

By symmetry, it suffices to show $\dim(D(\Gamma_{1}))\ge \dim(D(\Gamma_{2}))$. If $\dim(D(\Gamma_{1}))<\dim(D(\Gamma_{2}))$, then $H^{\textmd{p}}_{n}(D(\Gamma_{1}))$ is trivial ($n=\dim(D(\Gamma_{2}))$) since there are no $n$--dimensional cells in $D(\Gamma_1)$. On the other hand, $(\ast)$ implies there is a top dimensional flat in $D(\Gamma_{2})$, thus $H^{\textmd{p}}_{n}(D(\Gamma_{2}))$ is non-trivial, which yields a contradiction. 
\end{proof}

\subsubsection{The right-angled Artin group case} Recall that for every simplicial graph $\Gamma$, there is a corresponding RAAG $G(\Gamma)$. Suppose $\bar{X}(\Gamma)$ is the Salvetti complex of $G(\Gamma)$. Then the $1$--cells and $2$--cells of $\bar{X}(\Gamma)$ are in 1--1 correspondence with the vertices and edges in $\Gamma$ receptively. The closure of each $k$--cell in $\bar{X}(\Gamma)$ is a $k$--torus, which we call a \textit{standard $k$--torus}. One can verify that the Salvetti complex $\bar{X}(\Gamma)$ is a weakly special cube complex.

We label the vertices of $\Gamma$ by distinct letters (they correspond to the generators of $G(\Gamma)$), which induces a labelling of the edges of the Salvetti complex. We choose an orientation for each edge in the Salvetti complex and this would give us a directed labelling of the edges in $X(\Gamma)$. If we specify some base point $v\in X(\Gamma)$ ($v$ is a vertex), then there is a 1--1 correspondence between words in $G(\Gamma)$ and edge paths in $X(\Gamma)$ which starts at $v$. 

A subgraph $\Gamma'\subset\Gamma$ is a \textit{full subgraph} if there does not exist edge $e\subset\Gamma$ such that the two endpoints of $e$ belong to $\Gamma'$ but $e\nsubseteq\Gamma'$. In this case, there is an embedding $\bar{X}(\Gamma')\hookrightarrow\bar{X}(\Gamma)$ which is locally isometric. If $p\co X(\Gamma)\to\bar{X}(\Gamma)$ is the universal cover, then each connect component of $p^{-1}(\bar{X}(\Gamma'))$ is a convex subcomplex isometric to $X(\Gamma')$. Following \cite{MR2421136}, we call these components \textit{standard subcomplexes associated with $\Gamma'$}. Note that there is a 1--1 correspondence between standard subcomplexes associated with $\Gamma'$ and left cosets of $G(\Gamma')$ in $G(\Gamma)$. A \textit{standard $k$--flat} is the standard complex associated with a complete subgraph of $k$ vertices. When $k=1$, we also call it a \textit{standard geodesic}.

Given subcomplex $K\subset X(\Gamma)$, we denote the collection of labels of edges in $K$ by $label(K)$ and the corresponding collection of vertices in $\Gamma$ by $V(K)$. 

Let $V\subset\Gamma$ be a set of vertices. We define the \textit{orthogonal complement} of $V$, denoted $V^{\perp}$, to be the set $\{w\in\Gamma\mid d(w,v)=1$ for any $v\in V\}$. 

The following theorem follows from Theorem \ref{5.18}:

\begin{thm}
\label{5.21}
Let $\Gamma_{1}$, $\Gamma_{2}$ be finite simplicial graphs, and let $\phi\co  X(\Gamma_{1})\to X(\Gamma_{2})$ be an $(L, A)$--quasi-isometry. Then $\dim(X(\Gamma_{1}))=\dim(X(\Gamma_{2}))$. And there is a constant $D=D(L, A)$ such that for any top-dimensional flat $F_{1}$ in $X(\Gamma_{1})$, we can find a flat $F_{2}$ in $X(\Gamma_{2})$ such that
\begin{center}
$d_{H}(\phi(F_{1}),F_{2})< D\,.$
\end{center}
\end{thm}

One can argue as in Corollary \ref{5.20} or using the invariance of cohomological dimension to show $\dim(X(\Gamma_{1}))=\dim(X(\Gamma_{2}))$.

Using Theorem \ref{5.21}, we can set up some immediate quasi-isometry invariant for RAAGs. Let $F(\Gamma)$ be the flag complex of $\Gamma$. We will assume $n=\dim(F(\Gamma))$ in the following discussion, then $\dim(X(\Gamma))=n+1$.

We construct a family of new graphs $\{\mathcal{G}_{d}(\Gamma)\}_{d=1}^{n}$, where the vertices of $\mathcal{G}_{d}(\Gamma)$ are in 1--1 correspondence with the top dimensional flats in $X(\Gamma)$ and two vertices $v_{1}$ and $v_{2}$ are joined by an edge if and only if the associated flats $F_{1}$ and $F_{2}$ satisfy the condition that there exists $r>0$ such that $N_{r}(F_{1})\cap N_{r}(F_{2})$ contains a flat of dimension $d$. Let $\mathcal{G}^{s}_{d}(\Gamma)$ be the full subgraph of $\mathcal{G}_{d}(\Gamma)$ spanned by those vertices representing standard flats of top dimension.

Lemma \ref{2.10} and Theorem \ref{5.21} yield the following result:

\begin{cor}
\label{5.22}
Given finite simplicial graph $\Gamma_{1}$, $\Gamma_{2}$ and a quasi-isometry $q\co X(\Gamma_{1})\to X(\Gamma_{2})$, there is an induced graph isomorphism $q_{\ast}\co \mathcal{G}_{d}(\Gamma_{1})\to \mathcal{G}_{d}(\Gamma_{2})$ for $1\le d\le \dim(F(\Gamma_{1}))=\dim(F(\Gamma_{2}))$.
\end{cor}

The relation between $\mathcal{G}_{d}(\Gamma)$ and $\Gamma$ is complicated, but several basic properties of $\mathcal{G}_{d}(\Gamma)$ can be directly read from $\Gamma$. We first investigate the connectivity of $\mathcal{G}_{d}(\Gamma)$.

\begin{lem}
\label{5.23}
Suppose $1\le d\le n$. Then $\mathcal{G}_{d}(\Gamma)$ is connected if and only if $\mathcal{G}^{s}_{d}(\Gamma)$ is connected.
\end{lem}

\begin{proof}
For the $\Leftarrow$ direction, it suffices to show every point $v\in \mathcal{G}_{d}(\Gamma)$ is connected to some point in $\mathcal{G}^{s}_{d}(\Gamma)$. Let $F_{v}$ be the associated top dimensional flats. Pick vertex $x\in F_{v}$ and suppose $\{e_{i}\}_{i=1}^{n+1}$ are mutually orthogonal edges in $F_{v}$ emanating from $x$. Let $e^{\perp}_{1}$ be the subspace of $F_{v}$ orthogonal to $e_{1}$ and let $l_{i}$ be the unique standard geodesic such that $e_{i}\subset l_{i}$. Then by Lemma \ref{5.4}, the convex hull of $l_{1}$ and $e^{\perp}_{1}$ is a top dimensional flat $F_{v,1}$. By construction, $F_{v,1}$ is adjacent to $F_{v}$ in $\mathcal{G}_{d}(\Gamma)$. Now we can replace $F_v$ by $F_{v,1}$, and run the same argument with respect to $l_{2}$. After finitely many steps, we will reach at a standard flat $F$ which is the convex hull of $\{l_{i}\}_{i=1}^{n+1}$, moreover $F$ is connected to $F_v$ in $\mathcal{G}_d(\Gamma)$. Note that $F$ only depends on the choice of base vertex $x$ and the frame $\{e_{i}\}_{i=1}^{n+1}$ at $x$. So we also denote $F$ by $F=F_v(x,e_1,\cdots,e_{n+1})$.

Now we prove the other direction. Pick a different base point $x'\in F_{v}$ and frame $\{e'_{i}\}_{i=1}^{n+1}$ at $x'$, we claim $F_v(x,e_1,\cdots,e_{n+1})$ and $F_v(x',e'_1,\cdots,e'_{n+1})$ are connected in $\mathcal{G}^{s}_{d}(\Gamma)$. Note that
\begin{enumerate}
\item If $x'=x$, $e'_{i}=e_{i}$ for $2\le i\le n+1$ and $e'_{1}=-e_{1}$, then $F$ and $F'$ are adjacent inside $\mathcal{G}^{s}_{d}(\Gamma)$.
\item If $x'$ is the other end point of $e_{1}$, $e'_{1}=\overrightarrow{x'x}$ and $e'_{i}$ is parallel to $e_{i}$ for $2\le i\le n+1$, then $F=F'$.
\end{enumerate}
In general, we can connected $x$ and $x'$ by an edge path $\omega\subset F_{v}$ and use the previous two properties to induct on the combinatorial length of $\omega$.

Let $\{F_{i}\}_{i=1}^{m}$ be a chain of top dimensional flats representing an edge path in $\mathcal{G}_{d}(\Gamma)$ such that $F_{1}$ and $F_{m}$ are standard flats. Pick $i$, let $(Y,Y')=\inc (F_{i},F_{i+1})$ and let $\phi\co Y\to Y'$ be the isometry induced by $CAT(0)$ projection as in Lemma \ref{2.11} (2). Since $Y$ contains a $d$--dimensional flats, for vertex $x\in Y$, there are $d$ mutually orthogonal edges $\{e_{i}\}_{i=1}^{d}$ such that $x\in e_{i}\subset Y$. Let $x'=\phi(x)$ and let $e'_{i}=\phi(e_{i})$. We add more edges such that $\{e_{i}\}_{i=1}^{n+1}$ and $\{e'_{i}\}_{i=1}^{n+1}$ become basis in $F_{i}$ and $F_{i+1}$ respectively. Let $F_{i,i}=F_i(x,,e_1,\cdots,e_{n+1})$ and $F_{i+1,i}=F_{i+1}(x',e'_1,\cdots,e'_{n+1})$. By Lemma \ref{2.14}, $F_{i,i}$ and $F_{i+1,i}$ are adjacent in $\mathcal{G}^{s}_{d}(\Gamma)$ for $1\le i\le m-1$. Moreover, for $2\le i\le m-1$, $F_{i,i}$ and $F_{i,i-1}$ are connected by a path inside $\mathcal{G}^{s}_{d}(\Gamma)$ by the previous claim. Thus $F_{1}$ and $F_{m}$ are connected inside $\mathcal{G}^{s}_{d}(\Gamma)$. 
\end{proof}

Recall that the notion of $k$--gallery is defined in Definition \ref{1.5}.

\begin{lem}
\label{5.24}
$\mathcal{G}^{s}_{d}(\Gamma)$ is connected if and only if $\Gamma$ satisfies the following conditions:
\begin{enumerate}
\item For any vertex $v\in F(\Gamma)$, there is top dimensional simplex $\Delta\subset F(\Gamma)$ such that $\Delta\cap v^{\perp}$ contains at least $d$ vertices.
\item Any two top dimensional simplexes $\Delta_{1}$ and $\Delta_{2}$ in $F(\Gamma)$ are connected by a $(d-1)$--gallery.
\end{enumerate}
\end{lem}

\begin{proof}
For the only if part, pick vertex $x\in X(\Gamma)$ and let $\Gamma_{d,x}$ be the full subgraph of $\mathcal{G}^{s}_{d}(\Gamma)$ generated by those vertices representing top dimensional standard flats containing $x$. Then there is a canonical surjective simplicial map $\phi\co \mathcal{G}^{s}_{d}(\Gamma)\to \Gamma_{d,x}$ by sending any top dimensional standard flat $F$ to the unique standard flat $F'$ with $x\in F'$ and $label(F)=label(F')$. Since $\mathcal{G}^{s}_{d}(\Gamma)$ is connected, $\Gamma_{d,x}$ is also connected and (2) is true.

To see (1), suppose there exists a vertex $v\in F(\Gamma)$ such that for any top dimensional simplex $\Delta\subset F(\Gamma)$, $\Delta\cap v^{\perp}$ contains less than $d$ vertices. Pick vertex $x_{1}\in X(\Gamma)$, if $e\subset X(\Gamma)$ is the unique edge such that $V(e)=v$ and $x_{1}\in e$, then by our assumption, $e$ is not contained in any top dimensional standard flat. This is also true for any edge parallel to $e$. Let $h$ be the hyperplane dual to $e$. Suppose $x_{2}$ is the other endpoint of $e$. For $i=1,2$, let $F_{i}$ be the top dimensional standard flat such that $x_{i}\in F_{i}$. Then $F_{1}$ and $F_{2}$ are separated by $h$. Since $F_{1}$ and $F_{2}$ are joined by a chain of top dimensional standard flats such that each flat in the chain has trivial intersection with $h$ (otherwise some edge parallel to $e$ will be contained in a top dimensional standard flat), we can find $F'_{1}$ and $F'_{2}$ in this chain such that they are adjacent in $\mathcal{G}^{s}_{d}(\Gamma)$ and they are separated by $h$. Let $(Y_{1},Y_{2})=\inc (F'_{1},F'_{2})$. Then for vertex $y\in Y_{1}$, there are $d$ mutually orthogonal edges $\{e_{i}\}_{i=1}^{d}$ such that $y\in e_{i}\subset Y_{1}$. Let $h_{i}$ be the hyperplane dual to $e_{i}$. Then $h_{i}\cap h\neq\emptyset$ for $1\le i\le d$ by Lemma \ref{2.14}, hence in $\Gamma$ we have $d(V(e_{i}),V(e))=d(V(e_{i}),v)=1$ for $1\le i\le d$, which yields a contradiction.

For the other direction, note that (2) implies $\Gamma_{d,x}$ is connected for any vertex $x\in X(\Gamma)$ and (1) implies for adjacent vertices $x_{1},x_{2}\in X(\Gamma)$, there exist $v_{i}\in\Gamma_{d,x_{i}}$ for $i=1,2$ such that $v_{1}$ and $v_{2}$ are either adjacent or identical in $\mathcal{G}^{s}_{d}(\Gamma)$, thus $\mathcal{G}^{s}_{d}(\Gamma)$ is connected.
\end{proof}

The next result follows from Corollary \ref{5.22}, Lemma \ref{5.23} and Lemma \ref{5.24}.

\begin{thm}
\label{5.25}
Given $G(\Gamma_{1})$ and $G(\Gamma_{2})$ quasi-isometric to each other, for $1\le d\le \dim(F(\Gamma_{1}))$, $\Gamma_{1}$ satisfies condition (1) and (2) in Lemma \ref{5.24} if and only if $\Gamma_{2}$ also satisfies these conditions.
\end{thm}

Now we turn to the diameter of $\mathcal{G}_{1}(\Gamma)$.

If $\Gamma$ admits a non-trivial join decomposition $\Gamma=\Gamma_{1}\circ\Gamma_{2}$, then $diam(\mathcal{G}_{1}(\Gamma))\le 2$. To see this, take two arbitrary top dimensional flats $F_{1}$ and $F_{2}$ in $X(\Gamma)$, then $F_{i}=A_{i}\times B_{i}$, here $A_{i}$ and $B_{i}$ are top dimensional flats in $X(\Gamma_{1})$ and $X(\Gamma_{2})$ respectively for $i=1,2$ (see \cite[Lemma 2.3.8]{kleiner1997rigidity}). Let $F=A_{1}\times B_{2}$. Then $diam(N_{r}(F_{i})\cap N_{r}(F))=\infty$ for some $r>0$ and $i=1,2$, thus $diam(\mathcal{G}_{1}(\Gamma))\le 2$. Our next goal is to prove the following converse.

\begin{lem}
\label{5.26}
If $diam(\mathcal{G}_{1}(\Gamma))\le 2$ and if $\Gamma$ is not one point, then $\Gamma$ admits a nontrivial join decomposition $\Gamma=\Gamma_{1}\circ\Gamma_{2}$.
\end{lem}

In the first part of the following proof, we will use the argument in \cite[Section 4.1]{dani2012divergence}.
\begin{proof} 
Following \cite[Section 4.2]{dani2012divergence}, let $\Gamma^{c}$ be the complement graph of $\Gamma$. $\Gamma_{c}$ and $\Gamma$ have the same vertex set. Two vertex are adjacent in $\Gamma_{c}$ if and only if they are not adjacent in $\Gamma$. It suffices to show $\Gamma_{c}$ is disconnected.

We argue by contradiction and suppose $\Gamma_{c}$ is connected. Pick a top dimensional simplex $\Delta$ in the flag complex $F(\Gamma)$ of $\Gamma$ (we identify $\Gamma$ with the $1$--skeleton of $F(\Gamma)$). Note that $\Delta$ corresponds to a top dimensional standard torus $T_{\Delta}$ in the Salvetti complex. For any vertex $x\in X(\Gamma)$, we denote the unique standard flat in $X(\Gamma)$ which contains $x$ and covers $T_{\Delta}$ by $F_{\Delta,x}$. 

If $\Gamma$ does not contain vertices other than those in $\Delta$, then we are done, otherwise we can find vertex $\tilde{v}\in\Gamma$ with 
\begin{equation}
\label{5.27}
\tilde{v}\notin\Delta\,.
\end{equation}
Let $\omega$ be an edge path of $\Gamma^{c}$ which starts at $\tilde{v}$, ends at $\tilde{v}$ and travels through every vertices in $\Gamma^{c}$. By recording the labels of consecutive vertices in $\omega$, we obtain a word $W$. Let $W'$ be the concatenation of eight copies of $W$.

Pick vertex $x_{1}\in X(\Gamma)$ and let $l$ be the edge path which starts at $x_{1}$ and corresponds to the word $W'$. Let $x_{2}$ be the other endpoint of $l$. Note that $l$ is actually a geodesic segment by our construction of $W'$. For $i=1,2$, let $F_{i}=F_{\Delta,x_{i}}$ and let $w_{i}$ be the vertex in $\mathcal{G}_{1}(\Gamma)$ corresponding to $F_{i}$. We claim $d(w_{1},w_{2})>2$.

If $d(w_{1},w_{2})\le 2$, then there exists a top dimensional flat $F$ such that 
\begin{equation}
\label{5.28}
diam(N_{r}(F_{i})\cap N_{r}(F))=\infty 
\end{equation}
for some $r>0$ and $i=1,2$. Let $(Y_{1},Y)=\inc (F_{1},F)$. By (\ref{5.28}) and Lemma \ref{2.10}, $Y_{1}$ (and $Y$) is not a point and we can identify the convex hull of $Y\cup Y_{1}$ with $Y_{1}\times[0,d(F_{1},F)]$. Pick an edge $e_{a}\subset Y_{1}$ and let $K_{1}$ be the strip $e_{a}\times[0,d(F_{1},F)]$ inside $Y_{1}\times[0,d(F_{1},F)]$. By considering the pair $F$ and $F_{2}$, we can similarly find an edge $e_{b}\subset F_{2}$ and a strip $K_{2}$ isometric to $e_{b}\times[0,d(F_{2},F)]$ which joins $F$ and $F_2$. See Figure \ref{fig:4}.
\begin{figure}[ht!]
\labellist
\small\hair 2pt
\pinlabel $e_a$ at 43 67
\pinlabel $F_1$ at 56 22
\pinlabel $K_1$ at 146 65
\pinlabel $F$ at 252 33
\pinlabel $K_2$ at 366 55
\pinlabel $e_b$ at 471 56
\pinlabel $F_2$ at 481 33
\endlabellist
\centering
\includegraphics[scale=0.5]{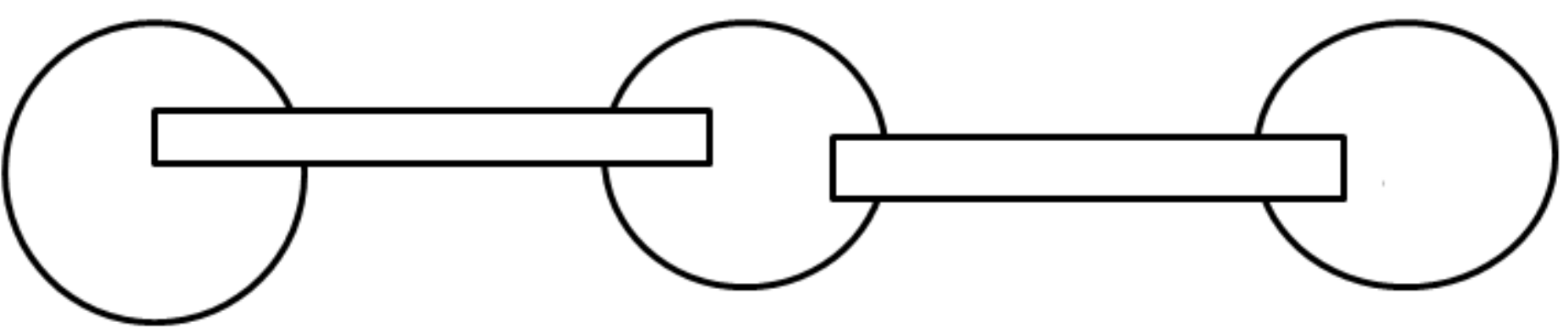}
\caption{}\label{fig:4}
\end{figure}

We parametrize the geodesic segment $l=\overline{x_{1}x_{2}}$ by arc length such that $l(0)=v_{1}$. Assume $l(N)=x_{2}$. Let $h_{i}$ be the hyperplane dual to the edge $\overline{l(i-1)l(i)}$ for $1\le i\le N$. Note that
\begin{equation}
\label{5.29}
h_{j}\ \textmd{separates}\ h_{i}\ \textmd{and}\ h_{k}\ \textmd{for}\ i<j<k\,.
\end{equation}
Moreover, each $h_{i}$ separates $F_{1}$ and $F_{2}$ by (\ref{5.27}), hence also separates $e_{a}$ and $e_{b}$. Consider the set $K_{1}\cup F\cup K_{2}$, which is connected and contains $e_{a}$ and $e_{b}$, so each $h_{i}$ must have non-trivial intersection with at least one of $K_{1}$, $F$ and $K_{2}$.

We claim each of $K_{1}$, $F$ and $K_{2}$ could intersect at most $2M$ hyperplanes from the collection $\{h_{i}\}_{i=1}^{N}$ where $M$ is the length of word $W$ ($M>1$ since $\Gamma$ contains more than one vertices). This will yield a contradiction since $N=8M$. We prove the claim for $K_{1}$, the case of $K_{2}$ is similar. 

Let $h_{a}$ be the hyperplane dual to $e_{a}$ and let $\Lambda=\{1\le i\le N\mid h_{a}\cap h_{i}\neq\emptyset\}$. Then
\begin{equation}
\label{5.30}
\{1\le i\le N\mid K_{1}\cap h_{i}\neq\emptyset\}\subset\Lambda\,.
\end{equation}
If $h_{a}=h_{i_{0}}$ for some $i_{0}$, then by (\ref{5.29}), $h_{i_0}$ is the only hyperplane in $\{h_{i}\}_{i=1}^{N}$ intersecting $K_1$. Hence we are done in this case. Now we assume $h_a\notin \{h_{i}\}_{i=1}^{N}$. Let $e_{i}$ be an edge dual to $h_{i}$. Then it follows from $h_{a}\cap h_{i}\neq\emptyset$ that for any $i\in\Lambda$,
\begin{equation}
\label{5.31}
d(V(e_{i}),V(e_{a}))=1
\end{equation}
in $\Gamma$. If the claim for $K_{1}$ is not true, then (\ref{5.30}) implies $\Lambda$ has cardinality bigger than $2M$, moreover, it follows from (\ref{5.29}) that if $i\in\Lambda$ and $j\in\Lambda$, then $k\in\Lambda$ for any $i\le k\le j$. By the construction of the word $W$, we know every vertex of $\Gamma$ is contained in the collection $\{V(e_{i})\}_{i\in\Lambda}$, which is contradictory to (\ref{5.31}).

Now prove the claim for $F$. Suppose $F\cap h_{i}\neq\emptyset$, then $V(e_{i})\in\Delta$. By the construction of $W'$, we know there exist positive integers $a,a'<M$ such that $d(V(e_{i+a}),V(e_{i}))\ge 2$ and $d(V(e_{i-a'}),V(e_{i}))\ge 2$ in $\Gamma$. Then $F\cap h_{j}=\emptyset$ for $j=i+a$ and $j=i-a'$. By (\ref{5.29}), $F\cap h_{j}=\emptyset$ for $j>i+a$ and $j<i-a'$. Thus the claim is true for $F$.
\end{proof}

\begin{thm}
\label{5.32}
The following are equivalent:
\begin{enumerate}
\item $diam(\mathcal{G}_{1}(\Gamma))<\infty$.
\item $diam(\mathcal{G}_{1}(\Gamma))\le 2$.
\item $\Gamma$ admits a non-trivial join decomposition or $\Gamma$ is one point.
\end{enumerate}
Moreover, these properties are quasi-isometry invariants for right-angled Artin groups.
\end{thm}

Note that $(1)\Rightarrow(3)$ follows by considering the concatenation of arbitrary many copies of $W$ in Lemma \ref{5.26} and applying the same argument. 

\begin{remark}
\label{5.33}
It is shown in \cite{behrstock2012divergence} and \cite{abrams2013pushing} that $G(\Gamma)$ has linear divergence if and only if $\Gamma$ is a non-trivial join, which also implies $\Gamma$ being a non-trivial join is quasi-isometry invariant. Moreover, their results together with Theorem B and Proposition 4.7 of \cite{kapovich1998quasi} implies the following stronger statement.

Given $X=X(\Gamma)$ and $X'=X(\Gamma')$. Let $\Gamma=\Delta\circ\Gamma_1\circ\cdots\circ\Gamma_k$ be the join decomposition such that $\Delta$ is the maximal clique factor, and each $\Gamma_i$ does not allow non-trivial further join decomposition. Similarly, let $\Gamma=\Delta'\circ\Gamma'_1\circ\cdots\circ\Gamma'_{k'}$. Let $X=\Bbb R^{n}\times\prod_{i=1}^{k}X(\Gamma_{i})$ and let $X'=\Bbb R^{n'}\times\prod_{j=1}^{k'}X(\Gamma'_{j})$ be the corresponding product decomposition. If $\phi\co X\to X'$ is an $(L,A)$ quasi-isometry, then $n=n'$, $k=k'$ and there exist constants $L'=L'(L,A)$, $A'=A'(L,A)$, $D=D(L,A)$ such that after re-indexing the factors in $X'$, we have $(L',A')$ quasi-isometry $\phi_{i}\co  X(\Gamma_{i})\to X(\Gamma'_{i})$ so that $d(p'\circ\phi, \prod_{i=1}^{k}\phi_{i}\circ p)<D$, where $p\co  X\to \prod_{i=1}^{k}X(\Gamma_{i})$ and $p'\co X'\to \prod_{i=1}^{k}X(\Gamma'_{i})$ are the projections.

More generally, let $X$ and $X'$ be locally compact $CAT(0)$ cube complexes which admit a cocompact and essential action. Let $X=\prod_{j=1}^{n} Z_j\times \prod_{i=1}^{k} X_{i}$ be the finest product decomposition of $X$, where $Z_j$'s are exactly the factors which are quasi-isometric to $\Bbb R$. Suppose $Z=\prod_{j=1}^{n} Z_j$. Then $X=Z\times \prod_{i=1}^{k} X_{i}$. Similarly, we decompose $X'$ as $X'=Z'\times \prod_{i=1}^{k'} X'_{i}$. Then any quasi-isometry between $X$ and $X'$ respects such product decompositions in the sense of the previous paragraph. This is a consequence of \cite[Theorem B]{kapovich1998quasi}, \cite[Proposition 4.7]{kapovich1998quasi} and \cite[Theorem 6.3]{caprace2011rank}.
\end{remark}

%% file: Appendix.tex
\section{Top dimensional support sets in spaces of finite geometric dimension and application to Euclidean building}
In this section we adjust previous arguments to study the structure of top dimensional quasiflats in Euclidean buildings and prove the following result.
\begin{thm}
\label{6.1}
If $Y$ is a Euclidean building of rank $n$ and $[\sigma]\in H^{\textmd{p}}_{n,n}(Y)$, then there exist finitely many Weyl cones $\{W_{i}\}_{i=1}^{h}$ such that 
\begin{equation*}
d_{H}(S_{[\sigma]},\cup_{i=1}^{h}W_{i})<\infty\,.
\end{equation*}
Moreover, we can assume $W_{i}\subset S_{[\sigma]}$ for all $i$.
\end{thm}

For the case of discrete Euclidean building, our previous method goes through without much modification. One way to handle the non-discrete case is to use \cite[Lemma 6.2.2]{kleiner1997rigidity} (which says the support set of a top dimensional class locally looks like a polyhedral complex) to reduce to the discrete case. But this lemma relies on the local structure of Euclidean building. We introduce another way, based on the generalization of results in Section \ref{growth condition} to $CAT(0)$ space of finite geometric dimension (or homological dimension), which is of independent interest. 

\begin{lem}
\label{6.2}
Lemma \ref{3.4} is true under the assumption that $Y$ is a CAT(0) space of homological dimension $\le n$.
\end{lem}

\begin{proof}
Let $\phi\co  \Bbb E^{n}\to Y$ be a top dimensional $(L,A)$--quasiflat. We can assume $\phi$ is Lipschitz as before since $Y$ is $CAT(0)$. Let $[\Bbb E^{n}]$ be the fundamental class of $\Bbb E^{n}$. Pick $\sigma=\phi_{\ast}([\Bbb E^{n}])$ and let $S=S_{[\sigma]}$ be the support set. Pick $\epsilon>0$, suppose $U$ is the $1$--neighbourhood of $\textmd{Im}\ \phi$ and suppose $\{U_{\lambda}\}_{\lambda\in\Lambda}$ is a covering of $U$, where each $U_{\lambda}$ is a open subset of $U$ with diameter $< 1$. Since every metric space is paracompact, we can assume this covering is locally finite and define a continuous map $\varphi\co  U\to \Bbb E^{n}$ via the nerve complex of this covering as in Lemma \ref{4.16} such that there exists a constant $C$ such that
\begin{equation}
\label{6.3}
d(\varphi\circ\phi(x),x)<C
\end{equation}
for any $x\in \Bbb E^{n}$, thus $\varphi_{\ast}([\sigma])=[\Bbb E^{n}]$. Then we have $S_{[\Bbb E^{n}]}=\Bbb E^{n}\subset\varphi(S_{[\sigma]})$ by Lemma \ref{3.2}. It follows that $d_{H}(S,\textmd{Im}\ \phi)<D=D(L,A)$.
\end{proof}

\begin{remark}
\label{6.4}
In the above proof, we need to define $\varphi$ in an open neighbourhood of $\textmd{Im}\ \phi$ since $S_{[\sigma],\textmd{Im}\ \phi}$ might be strictly smaller than $S_{[\sigma],Y}$. Also we do not need to bound the dimension of the nerve complex of $\{U_{\lambda}\}_{\lambda\in\Lambda}$ as in Lemma \ref{4.16} since $Y$ is $CAT(0)$, while in Lemma \ref{4.16}, the target space $CK$ is linearly contractible with the contractibility constant possibly greater than $1$. 
\end{remark}

Recall that in a polyhedral complex, every top dimensional homology class can be represented by a cycle with image inside the support of the homology class. However, we do not know if this is still true in the case of arbitrary metric space of homological dimension $n$. The following result helps us to get around this point:

\begin{lem}
\label{6.5}
Let $Y$ be a metric space of homological dimension $\le n$ and let $[\sigma]\in H^{\textmd{p}}_{n}(Y)$. If $O$ is an open neighbourhood of $S_{[\sigma]}$, then there exists a proper cycle $\sigma'$ such that $[\sigma]=[\sigma']$ and $\textmd{Im}\ \sigma'\subset O$.
\end{lem}

\begin{proof}
We first prove a relative version of the above lemma for the usual homology theory. Let $V\subset U$ be open sets in $Y$. Pick $[\alpha]\in H_{n}(U,V)$ and let $K=S_{[\alpha]}$. We claim for any open neighbourhood $O\supseteq K$, there exist chains $\beta$ and $\gamma$ such that $\textmd{Im}\ \beta\subset U$, $\textmd{Im}\ \gamma\subset V\cup O$ and $\alpha=\partial\beta+\gamma$.

Let $K'=\textmd{Im}\ \alpha\setminus (V\cup O)$. For every point $x\in K'$, there exists $\epsilon(x)>0$ such that $\bar{B}(x,2\epsilon(x))\subset U\setminus \textmd{Im}\ \partial\alpha$ and $[\alpha]$ is trivial in $H_{n}(U,U\setminus \bar{B}(x,2\epsilon(x)))$. Since $K'$ is compact, we can find finite points $\{x_{i}\}_{i=1}^{N}$ in $K'$ such that $K'\subset\cup_{i=1}^{N} B(x_{i},\epsilon(x_{i}))$. Suppose $U_{K'}=\cup_{i=1}^{N} B(x_{i},\epsilon(x_{i}))$ and $W=V\cup O\cup U_{K'}$. Then $\textmd{Im}\ \alpha\subset W$ and $[\alpha]$ is trivial in $H_{n}(U,W)$. Let $W'=W\setminus(\cup_{i=1}^{N}\bar{B}(x_{i},2\epsilon(x_{i})))$. Then $\textmd{Im}\ \partial\alpha\subset W'\subset V\cup O$, so it suffices to show $[\alpha]$ is trivial in $H_{n}(U,W')$, but this follows from the Mayer--Vietoris argument in Lemma \ref{3.2}.

Now we turn to the case of $[\sigma]\in H^{\textmd{p}}_{n}(Y)$. Pick a base point $p\in Y$, put $U_{1}:=B(p,4)$, $U_{i}:=B(p,3^{i}+1)\setminus\bar{B}(p,3^{i-1}-1)$ for $i>1$, $U'_{1}:= Y\setminus \bar{B}(p,3)$ and $U'_{i}:=B(p,3^{i-1})\cup(Y\setminus \bar{B}(p,3^{i}))$ for $i>1$. By barycentric subdivision, we can assume every singular simplex in $\sigma$ has image of diameter $\le 1/3$. 

Set $\sigma_{0}=\sigma$, $V_{0}=Y$ and $V_{i}=Y\setminus \bar{B}(p,3^{i})$ for $i\ge 1$. Given $\sigma_{i}$ with $\textmd{Im}\ \sigma_{i}\subset (O\cup V_{i})$ (this is trivially true for $i=0$), we define $\sigma_{i+1}$ inductively as follows. First subdivide $\sigma_{i}$ to get a proper cycle $\sigma'_{i+1}$ such that $\textmd{Im}\ \sigma_{i}=\textmd{Im}\ \sigma'_{i+1}$, $\sigma'_{i+1}=\sigma_{i}+\partial\beta_{i+1}$ with $\textmd{Im}\ \beta_{i+1}\subset U_{i+1}$ and $\sigma'_{i+1}=\sigma'_{i+1,1}+\sigma'_{i+1,2}$ for $\textmd{Im}\ \sigma'_{i+1,1}\subset U_{i+1}$ and $\textmd{Im}\ \sigma'_{i+1,2}\subset U'_{i+1}$. It follows that $\textmd{Im}\ \partial\sigma'_{i+1,1}\subset U_{i+1}\cap U'_{i+1}$ and $\textmd{Im}\ \partial\sigma'_{i+1,1}\subset \textmd{Im}\ \sigma_{i}\subset(O\cup V_{i})$. So we can view $[\sigma'_{i+1,1}]$ as an element in $H_{n}(U_{i+1},U_{i+1}\cap U'_{i+1}\cap(O\cup V_{i}))$. Then by the previous claim, there exists a chain $\beta'_{i+1}$ with $\textmd{Im}\ \beta'_{i+1}\subset U_{i+1}$ such that $\textmd{Im}\ (\sigma'_{i+1,1}+\partial\beta'_{i+1})\subset F$, where
\begin{align*}
F&= (U_{i+1}\cap U'_{i+1}\cap(O\cup V_{i}))\cup(O\cap U_{i+1})=(U_{i+1}\cap U'_{i+1}\cap V_{i})\cup(O\cap U_{i+1})\\
&=(U_{i+1}\cap V_{i+1})\cup(O\cap U_{i+1})=(O\cup V_{i+1})\cap U_{i+1}
\end{align*}
Let $\sigma_{i+1}=\sigma_{i}+\partial(\beta_{i+1}+\beta'_{i+1})$. Then $\textmd{Im}\ \sigma_{i+1}\subset (F\cup \textmd{Im}\ \sigma'_{i+1,2})\subset(F\cup(O\cup V_{i}))\subset (O\cup V_{i+1})$ and the induction goes through.

Let $\sigma'=\sigma+\sum_{i=1}^{\infty}\partial(\beta_{i}+\beta'_{i})$. Since $\textmd{Im}\ (\beta_{i}+\beta'_{i})\subset U_{i}$, the infinite summation makes sense and $\sigma'$ is a proper cycle. Also $\textmd{Im}\ \sigma'\subset O$ by construction. 
\end{proof} 

\begin{remark}
\label{6.6}
The above proof also shows that $[\sigma]\in H^{\textmd{p}}_{n}(Y)$ is nontrivial if and only if $S_{[\sigma]}\neq \emptyset$. This is not true for lower dimensional cycles.
\end{remark}

\begin{cor}
\label{6.7}
Let $Z$ be a $CAT(1)$ space of homological dimension $n$. If $[\sigma]\in H_{n}(Z)$ is a nontrivial element, then for every point $x\in Z$, there exists a point $y\in S_{[\sigma]}$ such that $d(x,y)=\pi$.
\end{cor}

\begin{proof}
We argue by contradiction and assume there exists a point $x\in Z$ such that $S_{[\sigma]}\subset B(x,\pi)$, then by Lemma \ref{6.5}, there exists a cycle $\sigma'$ such that $\textmd{Im}\ \sigma'\subset B(x,\pi)$ and $[\sigma']=[\sigma]$. However, $B(x,\pi)$ is contractible and $[\sigma']$ must be trivial, which yields a contradiction.
\end{proof}

Let $Y$ be a $CAT(0)$ space. Pick $p\in Y$ and let $T_{p}Y$ be the tangent cone at $p$. Denote the base point of $T_{p}Y$ by $o$. Recall that there are logarithmic maps $\log_{p}\co  Y\to T_{p}Y$ and $\log_{p}\co  Y\setminus\{p\}\to\Sigma_{p}Y$. By \cite[Theorem 3.5]{kramer2011local} (see also \cite{uppercurvature,bk}), $\log_{p}\co  Y\setminus\{p\}\to\Sigma_{p}Y$ is a homotopy equivalence. Thus 

\begin{lem}
\label{6.8}
The map $(\log_{p})_{\ast}\co  H_{\ast}(Y,Y\setminus\{p\})\to H_{\ast}(T_{p}Y,T_{p}Y\setminus\{o\})$ is an isomorphism.
\end{lem}

We need a simple observation about support sets in cones before we proceed. Let $Z$ be a metric space and let $CZ$ be the Euclidean cone over $Z$ with base point $o$. Pick $[\sigma]\in H_{i}(CZ,CZ\setminus B(o,r))$. Recall that there is an isomorphism $\partial\co H_{i}(CZ,CZ\setminus B(o,r))\to H_{i-1}(Z)$ induced by the boundary map. 

\begin{lem}
\label{6.10}
Suppose $S=S_{\partial[\sigma],Z}$ and suppose $CS$ is the cone over $S$ inside $CZ$. Then$S_{[\sigma],CZ,CZ\setminus B(o,r)}=CS\cap B(o,r)$.
\end{lem}

The next lemma is an immediate consequence of \cite[Theorem A]{kleiner1999local}.

\begin{lem}
\label{6.11}
If $Z$ is a $CAT(\kappa)$ space of homological dimension $\le n$, then for any $p\in Z$, $\Sigma_{p}Z$ is of homological dimension $\le n-1$.
\end{lem}

Now we are ready to prove the geodesic extension property for top dimensional support sets. The argument is similar to \cite[Lemma 3.1]{bestvina2008quasiflats}.
\begin{lem}
\label{6.12}
Let $Y$ be a $CAT(0)$ space of homological dimension $n$. Pick $[\sigma]\in H^{\textmd{p}}_{n}(Y)$ and let $S=S_{[\sigma]}$. Then geodesic segment $\overline{pq}\subset Y$ with $q\in S$, there exists a geodesic ray $\overline{q\xi}\subset S$ such that $\overline{pq}$ and $\overline{q\xi}$ fit together to form a geodesic ray.
\end{lem}

\begin{proof}
First we claim for any $\epsilon>0$, there exists a point $z\in S\cap S(p,\epsilon)$ such that the concatenation of $\overline{pq}$ and $\overline{qz}$ is a geodesic. Let $\log_{q}\co (Y,Y\setminus B(q,2\epsilon))\to (T_{q}Y,T_{q}Y\setminus B(o,2\epsilon))$ and let $\alpha=\log_{q}(\sigma)$. By Lemma \ref{6.11}, the homological dimension of $T_pY$ is bounded above by $n$. Then by Lemma \ref{3.2}, 
\begin{equation}
\label{6.13}
S_{[\alpha],T_{q}Y,T_{q}Y\setminus B(o,2\epsilon)}\subset f(S_{[\sigma],Y,Y\setminus B(q,2\epsilon)})=f(S\cap B(q,2\epsilon))\,.
\end{equation}
Let $\partial\co H_{n}(T_{q}Y,T_{q}Y\setminus B(o,2\epsilon))\to H_{n-1}(\Sigma_{q}Y)$ be the isomorphism and let $[\beta]=\partial[\alpha]$. Then $[\beta]$ is nontrivial in $H_{n-1}(\Sigma_{q}Y)$ by Lemma \ref{6.8} and the following commuting diagram:
\begin{center}
$\begin{CD}
H_{n}(Y,Y\setminus B(q,2\epsilon))          @>(\log_{q})_{\ast}>>        H_{n}(T_{q}Y,T_{q}Y\setminus B(o,2\epsilon))\\
@VVV                                                              @VVV\\
H_{n}(Y,Y\setminus\{q\})                            @>(\log_{q})_{\ast}>>        H_{n}(T_{q}Y,T_{q}Y\setminus \{o\})
\end{CD}$
\end{center}
Let $CS_{[\beta]}$ be the Euclidean cone over $S_{[\beta]}\subset\Sigma_{q}Y$ inside $T_{q}Y$. Then
\begin{equation}
\label{6.14}
CS_{[\beta]}\cap B(o,2\epsilon)=S_{[\alpha],T_{q}Y,T_{q}Y\setminus B(o,2\epsilon)}\subset f(S\cap B(q,2\epsilon))
\end{equation}
by (\ref{6.13}) and Lemma \ref{6.10}. Moreover, by Corollary \ref{6.7} and Lemma \ref{6.11}, there exists $x\in S_{[\beta]}$ such that 
\begin{equation}
\label{6.15}
d(x,\log_{q}(p))=\pi
\end{equation}
(here $\log_{q}\co Y\setminus\{q\}\to\Sigma_{q}Y$), so the claim follows from (\ref{6.14}).

By repeatedly applying the above claim, for each positive integer $n$, we can obtain unit speed geodesic $c_{n}\co [0,\epsilon]\to Y$ such that $c(0)=q$, $c(m\epsilon/2^{n})\in S$ for any integer $0\le m\le 2^{n}$ and $\log_{q}(c(\epsilon))=x$ ($x$ is the point in (\ref{6.15})). Note that $S\cap \bar{B}(q,\epsilon)$ is compact, so we assume without loss of generality that $r=\lim_{n\to\infty} c_{n}(\epsilon)$. If $c\co [0,\epsilon]\to Y$ is the unit speed geodesic joining $q$ and $r$, then $c_{n}$ converges uniformly to $c$, which implies $c([0,\epsilon])\subset S$. Moreover, $\log_{q}(c(\epsilon))=x$. Thus the concatenation of $\overline{pq}$ and $\overline{qr}$ is a geodesic by (\ref{6.15}). Now we can repeatedly apply this $\epsilon$--extension procedure to obtain the geodesic ray as required.
\end{proof}

In general, the above set $S\cap \bar{B}(q,\epsilon)$ is not equal to the geodesic cone $C_{q}(S\cap S(q,\epsilon))$ based at $q$ over $S\cap S(q,\epsilon)$ no matter how small $\epsilon$ is. However, we have
\begin{align*}
&\lim_{\epsilon\to 0}d_{GH}(\frac{1}{\epsilon}(C_{q}(S\cap S(q,\epsilon))), CS_{[\beta]}\cap\bar{B}(o,1))\\
&=\lim_{\epsilon\to 0}d_{GH}(\frac{1}{\epsilon}(S\cap \bar{B}(q,\epsilon)), CS_{[\beta]}\cap\bar{B}(o,1))=0\,.
\end{align*}
Thus the tangent cone of $S$ exists for every point in $S$.

\begin{remark}
\label{6.16}
By the same proof, we know Lemma \ref{6.12} is still true if $Y$ is a Alexandrov space which has curvature bounded above and has homological dimension $=n$. In this case, $\overline{p\xi}$ is locally geodesic ray.
\end{remark}

\begin{lem}
\label{6.17}
Let $Z$ be a $CAT(1)$ space of homological dimension $\le n$ and $[\sigma]\in \tilde{H}_{n}(Z)$ be a nontrivial class. Then 
\begin{enumerate}
\item $\mathcal{H}^{n}(S_{[\sigma]})\ge \mathcal{H}^{n}(\Bbb S^{n})$. 
\item Let $V(n,r)$ be the volume of $r$--ball in $\Bbb S^{n}$. Then for any $0\le r\le R<\pi$ and any $p\in S_{[\sigma]}$, 
\begin{equation*}
1\le\frac{\mathcal{H}^{n}(B(p,r)\cap S_{[\sigma]})}{V(n,r)}\le \frac{\mathcal{H}^{n}(B(p,R)\cap S_{[\sigma]})}{V(n,R)}\,.
\end{equation*}
\item If $\mathcal{H}^{n}(S_{[\sigma]})=\mathcal{H}^{n}(\Bbb S^{n})$, then $S_{[\sigma]}$ is an isometrically embedded copy of $\Bbb S^{n}$.
\end{enumerate}
\end{lem}

$\Bbb S^{n}$ denotes the $n$--dimensional standard sphere with constant curvature 1.

\begin{proof}
We claim there exists a 1--Lipschitz map from a subset of $S_{[\sigma]}$ to a full measure subset of $\Bbb S^{n}$. Let us assume this is true for $i=n-1$. Pick $p\in S_{[\sigma]}$, let $\Bbb S^{0}\ast \Sigma_{p}Z$ be the spherical suspension of $\Sigma_{p}Z$ and let $o$ be one of the suspension point. Then there is a well-define 1--Lipschitz map $\log_{p}\co B(p,\pi)\to \Bbb S^{0}\ast \Sigma_{p}Z$ sending $p$ to $o$. Let $[\beta]$ be the image of $[\sigma]$ under the map 
\begin{align*}
&\tilde{H}_{n}(Z)\to H_{n}(Z,Z\setminus\{p\})\to H_{n}(B(p,\pi),B(p,\pi)\setminus\{p\})\\
&\xrightarrow{(\log_{p})_{\ast}} H_{n}(B(o,\pi),B(o,\pi)\setminus\{o\})\to \tilde{H}_{n-1}(\Sigma_{p}Z)\,.
\end{align*}
We can slightly adjust the proof of Lemma \ref{6.12} to show that:
\begin{equation}
\label{6.18}
\log_{p}(S_{[\sigma]}\cap B(p,\pi))\supseteq (\Bbb S^{0}\ast S_{[\beta]})\cap B(o,\pi)\,.
\end{equation}
The induction assumption implies that there are subset $K\in S_{[\beta]}$ and 1--Lipschitz map $f\co K\to \Bbb S^{n-1}$ such that $\mathcal{H}^{n-1}(\Bbb S^{n-1}\setminus f(K))=0$. Note that $f$ induces a 1-Lipschitz map $\tilde{f}\co  \Bbb S^{0}\ast K\to  \Bbb S^{0}\ast\Bbb S^{n-1}=\Bbb S^{n}$ whose image also has full measure, thus by (\ref{6.18}), there exists $K'\subset S_{[\sigma]}$ such that the image of $\tilde{f}\circ\log_{p}\co K'\to \Bbb S^{n}$ has full measure. It follows that $\mathcal{H}^{n}(S_{[\sigma]})\ge \mathcal{H}^{n}(\Bbb S^{n})$.

The first inequality of (2) follows from (1) and (\ref{6.18}). The second inequality follows from Remark \ref{6.16} and the proof of \cite[Corollary 3.3]{bestvina2008quasiflats}.

Now we prove (3). By Remark \ref{6.16}, for every point $x\in S_{[\beta]}$, there exists a geodesic segment $l_{x}\subset S_{[\sigma]}$ emanating from $p$ along the direction $x$ such that it has length $=\pi$. Let $A$ be the closure of $\cup_{x\in S_{[\beta]}}l_{x}$. Then $A\subset S_{[\sigma]}$ and $\mathcal{H}^{n}(A)\ge \mathcal{H}^{n}(\Bbb S^{n})$. Then (2) implies actually $A=S_{[\sigma]}$. Pick arbitrary $q\in S_{[\sigma]}\cap B(p,\pi)$, then there exists a sequence $\{q_{n}\}_{n=1}^{\infty}\subset \cup_{x\in S_{[\beta]}}l_{x}$ such that $\lim_{n\to\infty}q_{n}=q$. Since $\overline{q_{n}p}\subset S_{[\sigma]}$ by construction, $\overline{qp}\subset S_{[\sigma]}$. It follows that $S_{[\sigma]}$ is $\pi$--convex in $Y$. Then $S_{[\sigma]}$ can be viewed as a compact and geodesically complete $CAT(1)$ space. By \cite[Proposition 7.1]{nagano2002volume}, $S_{[\sigma]}$ is isometric to $\Bbb S^{n}$.
\end{proof}

We can recover the monotonicity (\ref{3.10}) and the lower density bound (\ref{3.11}) from Lemma \ref{6.12} and Lemma \ref{6.17}, then we can define the group $H^{\textmd{p}}_{n,n}(Y)$ as before when $Y$ is a $CAT(0)$ space of homological dimension $\le n$ and the rest discussion in Section \ref{growth condition} goes through without any change. Recall that the homological dimension of a $CAT(0)$ space is equal to its geometric dimension (\cite[Thoerem A]{kleiner1999local}), so the following result holds.

\begin{thm}
\label{6.19}
Let $Y$ be a $CAT(0)$ space of geometric dimension $n$. Pick $[\sigma]\in H^{\textmd{p}}_{n,n}(Y)$ and let $S=S_{[\sigma]}$. Then
\begin{enumerate}
\item (Local property I) Each point $p\in Y$ has a well-defined tangent cone $T_{p}Y$.
\item (Local property II) $S$ has the geodesic extension property in the sense of Lemma \ref{3.6}.
\item (Monotonicity and lower density bound) For all $0\leq r\leq R$ and $p\in Y$,
\begin{equation*}
\frac{\mathcal{H}^{n}(B(p,r)\cap S)}{r^{n}}\leq \frac{\mathcal{H}^{n}(B(p,R)\cap S)}{R^{n}}\,.
\end{equation*}
If $p\in S$, then
\begin{equation*}
\mathcal{H}^{n}(B(p,r)\cap S)\geq \omega_{n} r^{n}\,,
\end{equation*}
with equality only if $B(p,r)\cap S$ is isometric to a $r$--ball in $\Bbb E^{n}$, here $\omega_{n}$ is the volume of an $n$--dimensional Euclidean ball of radius $1$.
\item (Asymptotically conicality I) Let $B(o,1)$ be the unit ball in $C_{T}S$ centered at the cone point $o$. For any $p\in Y$, 
\begin{equation*}
\lim_{r\to +\infty}d_{GH}(\frac{1}{r}(B(p,r)\cap S),B(o,1))=0\,.
\end{equation*}
Moreover, put $\partial_{p,r}S:=\{\xi\in \partial_{T}S|\ \overline{p\xi}\subset B(p,r)\cup S\}$. Then
\begin{equation*}
\lim_{r\to+\infty}d_{H}(\partial_{p,r}S,\partial_{T}S)=0\,.
\end{equation*}
\item (Asymptotically conicality II) For all $\beta>0$ there is a $r<\infty$ such that if $x\in S\setminus B(p,r)$, then 
\begin{equation*}
diam(Ant_{\infty}(\log_{x}p,S))<\beta\,.
\end{equation*}
The diameter here is with respect to the angular metric on $\partial_{T}Y$.
\end{enumerate}
\end{thm}

Now we reinterpret the group $H^{\textmd{p}}_{n,n}(Y)$. Recall that there is another logarithmic map $\log_{p}\co  C_{T}Y\to Y$ sending the base point $o$ of $C_{T}Y$ to $p\in Y$. Since $\log_{p}$ is proper and 1--Lipschitz, it induces a map $H^{\textmd{p}}_{n,n}(C_{T}Y)\to H^{\textmd{p}}_{n,n}(Y)$.

Next we define a map in the other direction. Pick $[\sigma]\in H^{\textmd{p}}_{n,n}(Y)$, let $S=S_{[\sigma]}$ be the support set and let $U_{S}$ be the 1--neighbourhood of $S$. By Lemma \ref{6.5}, we can assume $\textmd{Im}\ \sigma\subset U_{S}$. For $\epsilon>0$, we define the map $f_{\epsilon}\co  U_{S}\to C_{T}S$ as in Lemma \ref{4.16}. To approximate $f_{\epsilon}$ by a continuous map, we choose a locally finite covering of $U_{S}$ by its open subsets which satisfies the diameter condition in Lemma \ref{4.16}, then proceed as before to obtain a continuous map $f_{\epsilon}\co  U_{S}\to C_{T}X$. Here the image may not stay inside $C_{T}S$, however, it is sub-linearly close to $C_TS$. We define $\exp_{\ast}([\sigma])=\lim_{\epsilon\to 0}f_{\epsilon\ast}([\sigma])$, note that $f_{\epsilon\ast}([\sigma])$ does not depends on $\epsilon$ when it is small. Since (\ref{4.23}) is still true, $(\log_{p})_{\ast}\circ\exp_{\ast}=\textmd{Id}$.

To see $\exp_{\ast}\circ(\log_{p})_{\ast}=\textmd{Id}$, we follow the proof of (\ref{4.30}), the only difference is that we need to replace $I_{\sigma'}$ there by the 1--neighbourhood of $\textmd{Im}\ \sigma'$, then use the nerve complex of suitable covering to approximate $g_{\epsilon}$ as we did for $f_{\epsilon}$. So 
\begin{equation*}
(\log_{p})_{\ast}\co  H^{\textmd{p}}_{n,n}(C_{T}Y)\to H^{\textmd{p}}_{n,n}(Y)
\end{equation*}
is an isomorphism, with the inverse map $\exp_{\ast}$ defined as above.

Let $h_{\lambda}\co  C_{T}Y\to C_{T}Y$ be the homothety map with respect to the base point $o$ by a factor $\lambda$. Then $h_{\lambda}$ is properly homotopic to $h_{1}$ for any $0<\lambda<\infty$, so for any $[\beta]\in H^{\textmd{p}}_{i}(C_{T}Y)$, $h_{\lambda\ast}([\beta])=[\beta]$ and $h_{\lambda}(S_{[\beta]})=S_{[\beta]}$. It follows that every cycle in $H^{\textmd{p}}_{i}(C_{T}Y)$ is conical. Thus the map $j\co H^{\textmd{p}}_{i}(C_{T}Y)\to H_{i}(C_{T}Y,C_{T}Y\setminus\{o\})\to H_{i-1}(\partial_{T}Y)$ is an isomorphism with inverse given by \textquotedblleft conning off\textquotedblright\ procedure. It follows that the map defined in (\ref{4.27}) and (\ref{4.28}) are isomorphisms and the analogues of Corollary \ref{4.33} and Remark \ref{4.37} in the case of $CAT(0)$ spaces with finite homological dimension are still true (again, for our argument to go through, we need to replace the set $I_{q}$ in the proof of Corollary \ref{4.33} by some $r$--neighbourhood of the image of $q$). This discussion can be summarized as follows. 

\begin{thm}
\label{6.20}
Let $q\co Y\to Y'$ be a quasi-isometric embedding, where $Y$ and $Y'$ are $CAT(0)$ spaces of geometric dimension $\le n$. Then:
\begin{enumerate}
\item $\partial:=j\circ(\exp_{\ast})\co H^{\textmd{p}}_{n,n}(Y)\to H_{n-1}(\partial_{T}Y)$ is a group isomorphism, and the inverse is given by the conning off map $ c\co H_{n-1}(\partial_{T}Y)\to H^{\textmd{p}}_{n,n}(Y)$ (see (\ref{4.28})).
\item $q$ induces a monomorphism $q_{\ast}\co H_{n-1}(\partial_{T}Y)\to H_{n-1}(\partial_{T}Y')$. If $q$ is a quasi-isometry, then $q_{\ast}$ is an isomorphism.
\item There exists $D'>0$ which depends on the quasi-isometry constants of $q$ such that
\begin{equation*}
d_{H}(q(S_{[\tilde{\sigma}]}),S_{q_{\ast}[\tilde{\sigma}]})<D'\,.
\end{equation*}
for any $[\tilde\sigma]\in H^{\textmd{p}}_{n,n}(Y)$.
\end{enumerate}
\end{thm}

We refer to the work of Kleiner and Lang \cite{quasimini} for a more general version of the above theorem. 

\begin{remark}
\label{6.21}
Pick $[\tau]\in H_{n-1}(\partial_{T}Y)$, then by Lemma \ref{6.8} and Theorem \ref{6.20}, $S_{c([\tau])}=\{y\in Y\mid [\tau]$ is nontrivial under $(\log_{y})_{\ast}\co  H_{n-1}(\partial_{T}Y)\to H_{n-1}(\Sigma_{y}Y)\}$, here $c$ is the conning off map in (\ref{4.28}).
\end{remark}
Now we are ready to prove Theorem \ref{6.1}. To avoid repetition, we will only sketch the main steps. 
 
\begin{proof}[Proof of Theorem \ref{6.1}]
If $Y$ is a Euclidean building of rank $n$, then by \cite[Corollary 6.1.1]{kleiner1997rigidity}, the homological dimension of $Y$ is less or equal than $n$. This also follows \cite[Theorem A]{kleiner1999local} by noticing $\Sigma_{p}Y$ is a spherical building of dimension $n-1$ for any $p\in Y$. Let $[\sigma]\in H^{\textmd{p}}_{n,n}(Z)$. 

\textit{Step 1:} Let $[\alpha]=\exp_{\ast}([\sigma])\in H^{\textmd{p}}_{n,n}(C_{T}Y)$. $\partial_{T}Y$ is a spherical building, so $S_{[\alpha]}$ is a cone over $K$, where $K=\cup_{i=1}^{h}C_{i}$ and each $C_{i}$ is a chamber in $\partial_{T}Y$.

\textit{Step 2:} Let $W_{i}\subset Y$ be a Weyl cone such that $\partial_{T}W_{i}=C_{i}$. Note that for any $i\neq j$, there is an apartment of $\partial_TY$ which contains $C_{i}$ and $C_{j}$. Thus we can assume $W_{i}$ and $W_{j}$ are contained in a common apartment of $Y$. So $W_i$ and $W_j$ satisfy inequalities similar to (\ref{2.11}). The quotient map $\sqcup_{i=1}^{h}C_{i}\to K$ induces a map $\varphi\co  CK\to Y$ which is a quasi-isometric embedding as in Lemma \ref{4.39}. We can assume $\varphi$ is continuous. Put $[\tau]=\varphi_{\ast}([CK])$, where $[CK]$ is the fundamental class of $CK$. Then it follows from the proof of Lemma \ref{4.47} that $d_{H}(S_{[\tau]},\cup_{i=1}^{h}W_{i})<\infty$. Moreover, we can assume $W_{i}\subset S_{[\tau]}$.

\textit{Step 3:} It suffices to show $[\sigma]=[\tau]$. Pick $p\in Y$. Note that there exists $D>0$ such that $d(\log_{p}(x),\varphi(x))<D$ for any $x\in CK$, then $[\tau]=\varphi_{\ast}([CK])=(\log_{p})_{\ast}([\alpha])=((\log_{p})_{\ast}\circ\exp_{\ast})([\sigma])=[\sigma]$.
\end{proof}

The following result is an immediate consequence of Lemma \ref{6.2} and Theorem \ref{6.1}.
\begin{cor}
\label{6.22}
If $Y$ is a Euclidean building of rank $n$ and $Q\subset Y$ is an $n$--quasiflat, then there exist finitely many Weyl cones $\{W_{i}\}_{i=1}^{h}$ such that 
\begin{equation*}
d_{H}(Q,\cup_{i=1}^{h}W_{i})<\infty\,.
\end{equation*}
\end{cor}